\documentclass[a4paper,11pt,sort&compress]{elsarticle}
\usepackage{amscd}
\usepackage{epstopdf}
\usepackage{color}
\usepackage{latexsym}
\usepackage{epsf}
\usepackage{multicol}
\usepackage{ifpdf}
\usepackage{float}
\usepackage{amsmath,amsfonts,amssymb,epsfig,subfigure}
\usepackage{mathrsfs}
\usepackage{stmaryrd}
\usepackage{graphicx,amsmath}
\usepackage{latexsym}
\usepackage{verbatim}
\usepackage{listings}
\usepackage{afterpage}
\usepackage{graphicx}
\usepackage{subfigure}
 \usepackage{caption}
 \usepackage{ragged2e}
 \usepackage{mdframed} 
\usepackage[framed,numbered,autolinebreaks,useliterate]{mcode}
\ifpdf
\usepackage[%
  pdftitle={Instructions for use of the document class
    elsart},%
  pdfauthor={},%
  pdfsubject={The preprint document class elsart},%
  pdfkeywords={instructions for use, elsart, document class},%
  pdfstartview=FitH,%
  bookmarks=true,%
  bookmarksopen=true,%
  breaklinks=true,%
  colorlinks=true,%
  linkcolor=blue,anchorcolor=blue,%
  citecolor=blue,filecolor=blue,%
  menucolor=blue,pagecolor=blue,%
  urlcolor=blue]{hyperref}
\else
\usepackage[%
  breaklinks=true,%
  colorlinks=true,%
  linkcolor=blue,anchorcolor=blue,%
  citecolor=blue,filecolor=blue,%
  menucolor=blue,pagecolor=blue,%
  urlcolor=blue]{hyperref}
\fi
\makeatletter
\def\elsartstyle{%
    \def\normalsize{\@setfontsize\normalsize\@xiipt{14.5}}
    \def\small{\@setfontsize\small\@xipt{13.6}}
    \let\footnotesize=\small
    \def\large{\@setfontsize\large\@xivpt{18}}
    \def\Large{\@setfontsize\Large\@xviipt{22}}
    \skip\@mpfootins = 18\p@ \@plus 2\p@
    \normalsize
} \@ifundefined{square}{}{}
\makeatother

\newtheorem{theorem}{Theorem}[section]
\newtheorem{assp}{Assumption}
\newtheorem{lemma}[theorem]{Lemma}
\newtheorem{rem}[theorem]{Remark}
\newtheorem{expl}[theorem]{Example}
\newtheorem{cor}[theorem]{Corollary}

\makeatletter \@addtoreset{equation}{section}

\allowdisplaybreaks
\newcommand{\E}{\mathbb{E}}

\newcommand{\RR}{\mathbb{R}}

\def\nn{\nonumber}

  \def\e{\varepsilon}
  
\def\la{\label}\def\be{\begin{equation}}  \def\ee{\end{equation}}
  \def\va{\varsigma} \def\var{\varrho} 

\usepackage[text={6.5in,9in},centering]{geometry}

\journal{Journal of \LaTeX\ Templates}

\begin{document}
\begin{sloppypar}
\begin{frontmatter}

\title{ Numerical approximation to the invariant measure of McKean-Vlasov
stochastic differential equations}
\tnotetext[mytitlenote]{
}

\author[a]{Yuanping Cui}
\ead{cuiyp@tiangong.edu.cn}
\author[b]{Xiaoyue Li\corref{mycorrespondingauthor}}
\cortext[mycorrespondingauthor]{Corresponding author}
\ead{lixy@tiangong.edu.cn }
\author[c]{Yi Liu}
\ead{yzl0274@auburn.edu}
\author[d]{Fengyu Wang}
\ead{wangfy@tju.edu.cn}
\address[a]{School of Mathematical Sciences,
Tiangong  University, Tianjin,  300387, China. }
\address[b]{School of Mathematical Sciences,
Tiangong  University, Tianjin,  300387, China.}
\address[c]{Department of Mathematics and Statistics, Auburn University, Auburn AL 36849, USA.}
\address[d]{Center for Applied Mathematics, Tianjin University,
 Tianjin, 300072, China}
\begin{abstract}
Inspired by the stochastic particle method, this paper develops an easily implementable explicit scheme for McKean-Vlasov stochastic differential equations (MV-SDEs) with superlinear growth coefficients. We prove that the numerical solution of the interacting particle system (IPS) attains the optimal uniform-in-time strong convergence rate of order $1/2$, and that it faithfully captures the long-term dynamics of MV-SDEs, including moment boundedness, stability, and ergodicity. In particular, the existence and uniqueness of an exchangeable numerical invariant probability measure for the  IPS are established via an appropriately constructed operator semigroup. Concerning the approximation of the invariant measure, we derive a non-asymptotic error bound between the distribution of the one-particle numerical solution and the marginal distribution of the IPS's  invariant measure; By the uniform-in-time propagation of chaos, we further obtain an asymptotic error bound between the one-particle marginal of the IPS’s numerical invariant measure and the exact invariant measure of the MV-SDE. Numerical experiments are provided to validate the theoretical results.
\end{abstract}

\begin{keyword}
McKean-Vlasov stochastic differential equations; Numerical method;  Strong convergence; Stability; Invariant probability measure.
\end{keyword}

\end{frontmatter}

\section{Introduction}\label{s-w}

The McKean-Vlasov stochastic differential equations (MV-SDEs)  constitute  a special kind of stochastic differential equations (SDEs) in which the drift and diffusion coefficients depend not only on the current state but also on the marginal law of the solution process,  as shown by 
 \begin{equation}\label{eq3.1}
\mathrm{d}X_{t}=f(X_{t},\mathcal{L}_{t}^{X})\mathrm{d}t+g(X_{t},\mathcal{L}_{t}^{X})\mathrm{d}B_{t},~~~\forall t\geq0,
\end{equation}
where~$\mathcal{L}_{t}^{X}$~denotes the { law of~the solution $X$ at time $t$}. From a modeling perspective, MV-SDE \eqref{eq3.1} has been used to  describe stochastic systems 
whose evolution is determined by both the microcosmic site and the macrocosmic distribution of the particle.  Especially, MV-SDEs  with superlinear coefficients
are often used as models in the fields of biological systems, financial engineering, physics, and others (see, e.g. \cite{BJ2012,GC2018,BRL2017,CR2018}).  In order to satisfy the demands of the application, developing  numerical methods for MV-SDEs is necessary and indispensable.  Despite its analytical usefulness, numerical computation of the invariant measure remains a central challenge for MV-SDEs in practice \cite{Lin, Bram}.  Therefore, the main objective of this paper is to propose an explicit scheme for MV-SDEs, to  approximate the dynamical behaviors in finite  and  infinite horizons, and to establish the numerical ergodicity property.

  The well-posedness theory is significant and  fundamental for MV-SDE research. Given that the dynamics are not only determined by their trajectories but also their global distributions, the solutions of MV-SDEs lack the strong Markov property. 
Thus,  the classical localization techniques and the Yamada-Watanabe principle for SDEs are not feasible for MV-SDEs  \cite{WFY2018}. As far as we know, three techniques, including the distribution iteration (see, e.g. \cite{WFY2018, Mishura, HuangX, Chaudru}),  the fixed point theorem (see, e.g. \cite{Martin,Ren,SAS1991}) and the Euler-type approximation  \cite{Hongw} are often used to analyze the strong well-posedness of MV-SDEs. For extensive results on the well-posedness of  weak solutions, refer to \cite{MR4260494,CR2018}. 

 Recently, the long-time dynamical behaviors have attracted much attention. 
Ding and Qiao \cite{DXJ2020} delineated sufficient conditions for the mean-square exponential stability and the almost-sure asymptotic stability.  Wu-Hu-Gao-Yuan \cite{MR4489658} gave the stabilization  principles of MV-SDEs via feedback control based on discrete-time state observation. The ergodicity of MV-SDEs has been meticulously analyzed under a spectrum of dissipative conditions. For instance, Wang \cite{WFY2018} derived the exponential ergodicity in the $L^2$-Wasserstein distance under the uniform dissipative condition. \cite{MR4291453,Vere,MR4312362,MR1970276} also yielded the ergodicity  under dissipative conditions at long distances, given that the dependence of drift coefficient on $\mu$ is sufficiently {weak}. Wang \cite{MR4550214} went a  step further to examine the exponential ergodicity in the  Wasserstein quasi-distance  for a class of non-dissipative MV-SDEs. 

Although the dynamical properties of the exact solutions of MV-SDEs have been well investigated, obtaining their closed forms  is almost impossible. Thus, it is essential and necessary to establish reliable numerical methods for MV-SDEs, especially with superlinear structures. 
When numerically solving MV-SDEs, the propagation of chaos plays a key role in  discretizing the distribution in the coefficients.  For clarity, we introduce an interacting particle system (IPS) as follows. For any integer $M\geq 1$,
\begin{equation*}
\left\{
\begin{array}{l}
\mathrm{d}X^{i,M}_{t}=f\left(X^{i,M}_{t},\mathcal{L}_{t}^{X,M}\right)\mathrm{d}t+g\left(X^{i,M}_{t},\mathcal{L}_{t}^{X,M}\right)\mathrm{d}B^{i}_{t},~~~t\geq0,\\
X^{i,M}_{t}|_{t=0}=X^{i}_{0},~~~1\leq i\leq M,
\end{array}
\right.
\end{equation*}
where  $\left(\{B^{i}_{t}\}_{t\geq0},X^{i}_0\right), ~ i=1,\cdots, M$,  are mutually  independent copies of $\left(\{B_{t}\}_{t\geq0},X_{0}\right)$ on the same probability space $(\Omega,\mathcal{F},\mathbb{P}, \{\mathcal{F}_{t}\}_{t\geq0})$,  and $\mathcal{L}^{X,M}_{t}=\frac{1}{M}\sum_{i=1}^{M}\boldsymbol{\delta}_{X^{i,M}_{t}}$ called the empirical distribution of particles $X^{i,M}_{t},~i=1,2,\cdot\cdot\cdot,M$. The  propagation of chaos reveals that  as $M \rightarrow \infty$, the solution of  each particle in IPS converges to that of  the original MV-SDE.  For a more comprehensive understanding of the propagation of chaos, see \cite{MHP1966, arxiv2301, MR4722362, MHP1967,MR4312362,arXivBJ, MR4291453,MR4702608, MR4163850,SAS1991} and the references therein.  

The propagation of chaos provides a theoretical basis for approximating MV-SDEs via the IPSs and has motivated the development of related  numerical methods for MV-SDEs.
The Euler-Maruyama (EM) scheme is proposed for the approximation of   MV-SDEs with linear growth coefficients (see, e.g. \cite{Yun, MR4289889,  Vere}). For  example, \cite{Vere} used the EM scheme to approximate the ergodic measure of MV-SDEs with  additive noise and bounded linear drift.  However,
   as pointed out by \cite{MR4367675},  EM numerical solutions diverge for  the IPS with superlinear structure  due to  the ``particle corruption'' phenomenon, namely,  the divergence of a single particle causes the  divergence of the entire IPS.  For example, consider  the scalar MV-SDE
\begin{align}\label{Ne4}
\mathrm{d}X_t=\big(X_t(-2-|X_t|)+\mathbb{E}X_t\big)\mathrm{d}t+\frac{1}{2}|X_t|^{\frac{3}{2}}\mathrm{d}B_t.
\end{align}
Here $
f(x,\mu)=x(-2-|x|)+\int_{\mathbb{R}}x\mu(\mathrm{d}x),~~~~g(x,\mu)=\frac{1}{2}|x|^{\frac{3}{2}}.
$
One notices that these coefficients satisfy Assumptions \ref{ass5} and \ref{ass6} given in Section \ref{sec55}. According to Corollary \ref{th6.2},  the exact solutions of both MV-SDE \eqref{Ne4} and the corresponding IPS \eqref{Ne1} are exponentially  stable in the mean square. However,  as pointed out in \cite{MR4367675}, superlinear coefficient structure may give rise to the ``particle corruption'' effect during the simulation of  the IPS \eqref{Ne1} by the EM scheme.  Figure \ref{FN1} predicts the sample paths of the EM solutions of the IPS \eqref{Ne1}  diverging to the exact ones.  Moreover,  Lemma \ref{AL2}  verifies  that the EM solutions are never exponentially stable in the mean square, regardless of how small the step size $\Delta$  is. Therefore, the EM scheme is no longer suitable for the approximation of MV-SDEs with superlinear growth coefficients. 
\begin{figure}[H]
  \centering
\includegraphics[width=10cm,height=5cm]{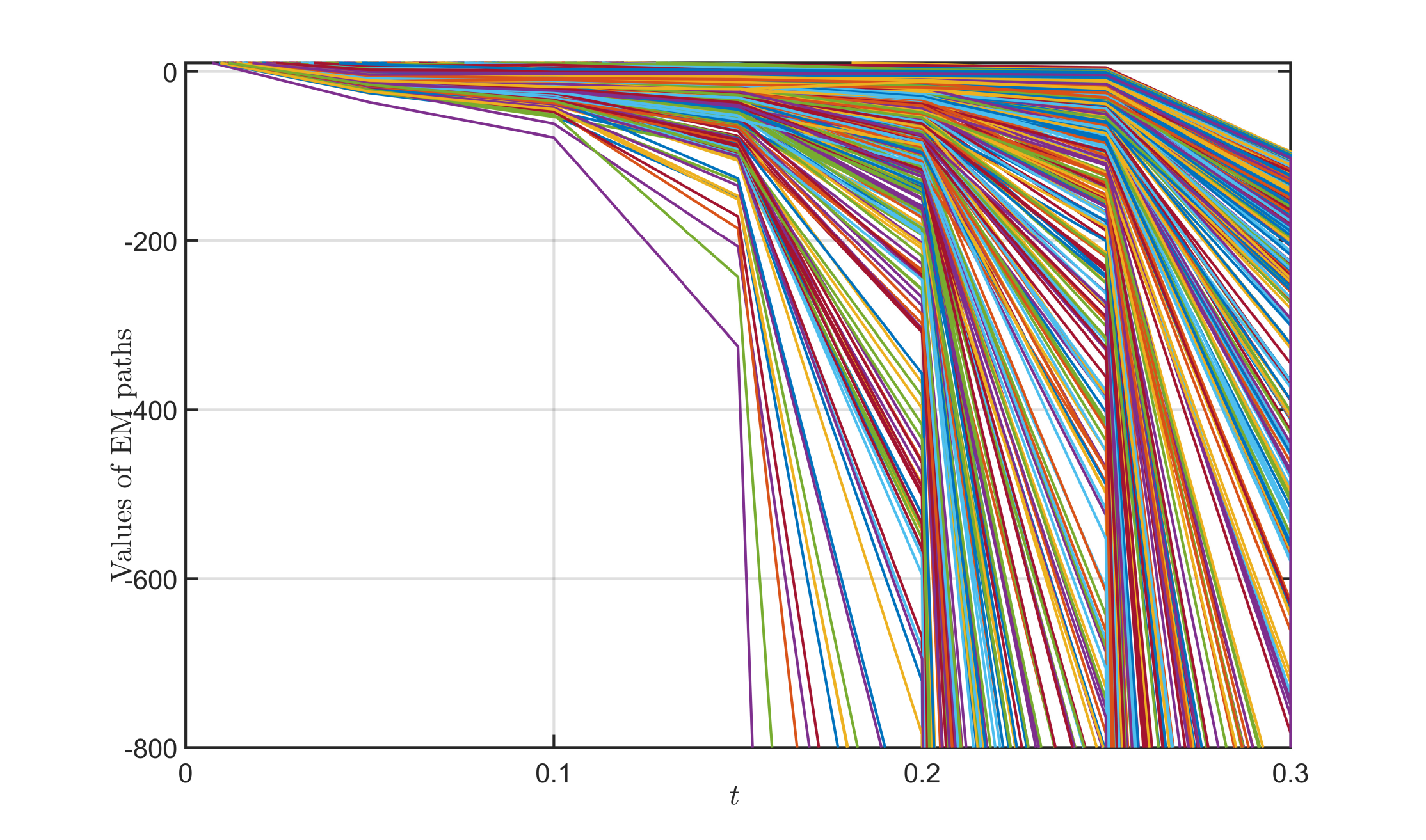}
\captionsetup{font=footnotesize}
    \caption{The sample paths of the  numerical solutions for the IPS \eqref{Ne1} by the EM scheme  for the initial value $X_0=18$, $\Delta=0.05$ and $M=2000$.}
\label{FN1}
\end{figure}

Owing to their simple algebraic structure and ease of implementation, constructing appropriate explicit schemes to approximate the dynamical behaviors of MV-SDEs with superlinear growth coefficients and to estimate their strong convergence rates is both a significant and desirable endeavor. For standard SDEs, a variety of explicit schemes have been developed, including tamed Euler-Maruyama (EM) and Milstein schemes \cite{MR2985171,MR3070913}, adaptive EM schemes \cite{MR4362994,MR4108115}, balanced schemes \cite{MR3129758}, and truncated EM schemes \cite{MR3370415,LXY2019}.

Inspired by these excellent works, several explicit schemes have been adapted to MV-SDEs with superlinear growth in the state variable. For MV-SDEs with a superlinear drift and a globally Lipschitz diffusion, tamed EM schemes with a strong convergence rate of order $1/2$ were introduced in \cite{MR4367675, Zhang}, while tamed Milstein schemes attaining first-order strong convergence were developed in \cite{MR4302574, MR4212406}; these schemes were subsequently extended to MV-SDEs with common noise in \cite{MR4497846}. Under a similar coefficient setting, adaptive EM and Milstein schemes with strong convergence rates of order $1/2$ and $1$, respectively, were proposed in \cite{MR4293705}. Moreover, a tamed EM scheme was proposed in \cite{Soni} for MV-SDEs with superlinear coefficients in both the state and measure variables through convolution-type interactions. Notably, \cite{Soni} also provides a systematic review of existing numerical methods for MV-SDEs, together with a detailed comparison of their applicable conditions and convergence rates in Table 1.

However, a common limitation of the aforementioned results is their requirement that the coefficients be Lipschitz continuous with respect to the measure. Consequently, research on numerical methods for MV-SDEs that are superlinear with respect to the measure variable is less. Recently, to address this gap, Soni-Neelima-Kumar-dos Reis \cite{Soni} proposed an explicit tamed scheme for a specific class of MV-SDEs with a superlinear convolution kernel, achieving a strong convergence rate of order 1/2. Zhang-Song-Zhu \cite{Zhang}  used the tamed EM scheme to verify the dimension-independent  propagation of chaos.
 
 While the aforementioned explicit methods have made significant progress in approximating  MV-SDEs over a finite horizon, research on their long-time dynamic approximations is rare.  In the linear framework, Ye-Zhou \cite{IMA} utilized the Euler-Maruyama (EM) scheme combined with the random batch method to derive an asymptotic error bound in the $L^2$
 -Wasserstein distance between the numerical solution’s distribution and the invariant measure. Furthermore,  Chen-dos Reis-Stockinger-Wilde \cite{MR4908784} constructed a non-Markovian Euler-type scheme for a one-dimensional MV-SDE, providing weak error bounds for both finite and infinite time horizons.

For MV-SDEs with superlinear growth, recent studies have also begun to address long-time behavior. For example, Bao-Hao \cite{arXivBJ} proposed a tamed EM scheme for Langevin-type MV-SDEs and established a uniform-in-time error bound. To ensure numerical stability over long horizons, other works have focused on contractivity. Specifically, Chen-dos Reis \cite{MR4413221} proposed an implicit split-step Euler (ISSE) method for MV-SDEs with superlinear drift in the state variable and proved the mean-square contractivity of the numerical solutions. This approach was later extended by \cite{ChenX} to the more challenging case of superlinear growth in both the state and measure variables.

The study of numerical ergodicity is particularly important, as it implies that the time average of a trajectory converges to its space average, a property with broad applications in optimization and sampling \cite{MR4059375,Ringh}. This has attracted significant attention to the numerical approximation of ergodic behavior. Recently, in a notable contribution, Soni-Neelima-Kumar-dos Reis \cite{Soni} obtained the exponential ergodicity of the tamed EM numerical solutions for the MV-SDEs with superlinear measure variables.  

The primary objective of this paper is to develop an explicit, easily implementable scheme for superlinear MV-SDEs that preserves ergodicity. To this end, we derive comprehensive error bounds for the numerical solution against the exact solution over both finite and infinite time horizons, and we also quantify the discrepancy between their respective invariant measures in the $L^{q}$-Wasserstein distance.
 
Due to the dependence on the law of the solutions, \cite{MR4260494} revealed that the  classical localization procedure does not carry over directly for MV-SDEs,  since the corresponding stopped process  satisfies a different equation compared with the original solution process. 
By  the propagation of chaos, the IPS is a $d\times M$-dimensional SDE with the interaction term. Using any existing numerical method of  SDEs to deal with the IPS leads to  the convergence behavior depending on the number of particles $M.$ As $M$ increases, the divergence may occur dramatically. This brings a huge challenge to the research of numerical method for MV-SDEs \cite{MR4163850}.  By the interacting  particle system to establish the  numerical theory the main obstacle is how to get the uniform convergence of the numerical solutions with respect to $M$.

To avoid the improper use of stopping-time techniques for the nonlinear distribution-dependent SDE, motivated by the truncation technique in \cite{LXY2019}, 
 we propose an explicit easily implementable scheme by means of the associated IPS, which reduces to the classical EM scheme in the linear case. To avoid dimension-dependent estimates, we employ particle-wise stopping-time arguments and exploit the symmetry of the IPS, which enables us to establish strong convergence together with the optimal strong convergence rate of order $1/2$, uniformly with respect to $M$.  Moreover, we show that the numerical solutions faithfully capture the long-time dynamical behavior of the original MV-SDE. In summary, the main contributions of this paper are as follows.

 \begin{itemize}
 \item[$\bullet$] To avoid the  ``particle corruption'' effect of the EM solutions approximating the IPS with superlinear coefficients, we set a truncated device  based on the growth rates of the drift and diffusion coefficients.  With the help of the truncation device, we modify the grid points of EM iterations  point by point.   We construct   an auxiliary piecewise continuous  particle process, which  connects the numerical nodes before and after truncation. By the analysis of the auxiliary process, we obtain the moment boundedness of the numerical solutions.

 \item[$\bullet$] 
  Employing the more precise localization  techniques for each particle, we establish the strong convergence of the numerical solutions with a  random initial value, and obtain the optimal convergence rate of order 1/2 under the polynomial growth condition.

\item[$\bullet$] To overcome the difficulty arising from the non-Markovian nature of the IPS, we construct an appropriate semigroup for the distributions of the numerical solutions. By establishing the Cauchy property of the associated distribution sequence, we prove the existence and uniqueness of the numerical invariant measure for the IPS, and derive a non-asymptotic error estimate between the distribution of the one-particle numerical solution and the marginal distribution of the IPS's invariant measure. Moreover, we establish the uniform-in-time moment boundedness and exponential stability of the numerical solutions. 
    
    \item[$\bullet$] By integrating continuous stochastic analysis with discrete iterative schemes, we not only establish the uniform-in-time error between the one particle numerical and MV-SDE's exact solutions, but also derive the asymptotic error between the one-particle marginal distribution of the numerical invariant measure and that of the exact invariant measure of the MV-SDE.
    
\end{itemize}

The structure of this paper is as follows: Section \ref{n-p} introduces some necessary {notation} and preliminaries.  Section \ref{s-c} proposes an easily implementable explicit scheme for the IPS corresponding to the original MV-SDE and proves the  strong convergence  in the finite horizon. Section \ref{s-cr} analyzes the strong convergence rate of the proposed  scheme.  Section \ref{sec55} explores the asymptotic properties of the numerical solution in the infinite horizon. Section \ref{num} gives numerical examples to validate our theoretical findings.

\section{Notations and Preliminaries}\label{n-p}
Let~$(\Omega,~\mathcal{F},~\mathbb{P}, \{\mathcal{F}_{t}\}_{t\geq0})$~be a complete probability space with a filtration~$\{\mathcal{F}_{t}\}_{t\geq0}$~satisfying the usual conditions (that is, it is right continuous and increasing while~${\mathcal{F}}_{0}$~contains all $\mathbb{P}$-null sets). $\{B_{t}\}_{t\geq0}$ is {an}  $m$-dimensional Brownian motion on probability space $( \Omega,~\mathcal{F},~\mathbb{P}, \{\mathcal{F}_{t}\}_{t\geq0})$. 
 Denote both the Euclidean norm in $\RR^{d}$ and the Frobenius norm in $\RR^{d\times m}$ by $|\cdot|$. For matrix~$A\in \RR^{d\times m}$, denote its transpose by $A^{T}$. For any $a,b\in \RR$,  let~$a\wedge b=\min\{a,b\}$~and~$a\vee b=\max\{a,b\}$. Let C denote a generic positive constant whose value may change in different appearances. 
For a set $G$, let~$I_{G}(x)=1$~if~$x\in G$~and~$0$~otherwise.
 Let $\mathcal{P}(\RR^d)$ denote the set of all   probability measures on ~$\RR^d$. For $q>0$,  define 
$$\mathcal{P}_{q}(\RR^d):=\left\{\mu \in \mathcal{P}(\RR^d) : \mu(|\cdot|^q)=\int_{\RR^d}|x|^q\mu(\mathrm{d}x)<\infty\right\} $$
and   the $L^q$-Wasserstein distance 
$$
\mathbb{W}_{q}(\mu,\nu)=\inf_{\pi \in \mathcal{C}(\mu,\nu)}\bigg(\int_{\RR^d\times \RR^d}|x-y|^{q}\pi(\mathrm{d}x,\mathrm{d}y)\bigg)^{\frac{1}{q}},~~\forall \mu, \nu \in \mathcal{P}_{q}(\RR^d),
$$
where~$\mathcal{C}(\mu,\nu)$~is the set of all couplings for~$\mu$~and~$\nu$, namely, $\pi\in \mathcal{C}(\mu,\nu)$ if and only if $\pi(\cdot,\RR^d)=\mu(\cdot)$ and $\pi(\RR^d, \cdot)=\nu(\cdot)$.  It is well known that $\mathcal{P}_{q}(\RR^d)$ is a Polish space under the $L^q$-Wasserstein distance $\mathbb{W}_{q}(\cdot ,\cdot)$. Especially, for any $\mu\in \mathcal{P}_{q}(\RR^d)$, $\mathbb{W}_q(\mu,\boldsymbol{\delta}_0)=\mu^{\frac{1}{q}}(|\cdot|^q)$, where $\boldsymbol{\delta}_{x}$ represents the Dirac measure concentrated at the point~$x\in\RR^d$.
For any  $\RR^d$-valued stochastic process~$Z=\{Z_{t}\}_{t\geq0}$,  let $\mathcal{L}_{t}^{Z}\in \mathcal{P}(\RR^d)$~denote the marginal distribution  of $Z$ at time~$t$. For~$p\in[2,\infty)$,
$L^{p}_{0}$ denotes the family of $\mathcal{F}_{0}$-measurable  $\RR^{d}$-valued random variables $\xi$
 satisfying $\mathbb{E}|\xi|^{p}<\infty$.

 In this paper,  we focus on  the MV-SDE given by \eqref{eq3.1} with initial data $X_0$, where
$ 
f:\RR^{d}\times\mathcal{P}_{2}(\RR^{d})\rightarrow\RR^{d},~ g:\RR^{d}\times\mathcal{P}_{2}(\RR^{d})\rightarrow\RR^{d\times m}. $ 
 {We now give several assumptions on $f$ and $g$.}
\begin{assp}\label{ass1}
 For any $R>0$, there exists a constant $K_{R}>0$ such that
\begin{align}\label{LL2.5}
|f(x_1,\mu)-f(x_2,\mu)|\leq K_{R}|x_1-x_2|,
\end{align}
for any $x_1,x_2\in \RR^{d}$ with $|x_1|\vee|x_2|\leq R$ and $\mu\in \mathcal{P}_{2}(\RR^{d})$. There exists a constant $K>0$ such that
\begin{align*}
|f(x,\mu_1)-f(x,\mu_2)|\leq K\mathbb{W}_{2}(\mu_{1},\mu_2)
\end{align*}
for any $x\in \RR^{d}$ and  $\mu_1,\mu_2\in \mathcal{P}_{2}(\RR^{d})$. 
\end{assp}
\begin{assp}\label{ass2}
There exists a pair of constants  $p_0\geq2$  and $L_1>0$ such that
\begin{align*}
2(x_1-x_2)^{T}\big(f(x_1,\mu_1)-f(x_2,\mu_2)\big)+(p_0-1)\left|g(x_1,\mu_1)-g(x_2,\mu_2)\right|^2\leq L_1\big(|x_1-x_2|^2+\mathbb{W}^2_{2}(\mu_1,\mu_2)\big)
\end{align*}
 for any $x_1,x_2\in \RR^{d}$ and $\mu_1, \mu_2\in \mathcal{P}_{2}(\RR^{d})$.
\end{assp}
\begin{assp}\label{ass3}
There exists a pair of constants $p\geq2$ and $L_2>0$ such that
\begin{align*}
2x^{T}f(x,\mu)+(p-1)|g(x,\mu)|^2\leq L_2(1+|x|^{2}+\mu(|\cdot|^2))
\end{align*}
for any $x\in \RR^{d}$ and $\mu\in \mathcal{P}_{2}(\RR^{d})$.
\end{assp}

\begin{rem}\label{r1}  One observes that Assumption \ref{ass3}  
 follows from  Assumption \ref{ass2} for any $2\leq p<p_0$. But for the possibly wider range of  $p$  it is still stated.
On the other hand,  Assumptions \ref{ass1} and \ref{ass2} imply that for any $R>0$,
\begin{align}\label{LL2.7}
|g(x_1,\mu)-g(x_2,\mu)|^2\leq C_R|x_1-x_2|^2,~~
|g(x,\mu_1)-g(x,\mu_2)|^2\leq  {L}_3\mathbb{W}^2_{2}(\mu_1,\mu_2)
\end{align}
for any $ x_1,x_2\in \RR^{d}$ with $|x_1|\vee|x_2|\leq R$, $x\in \RR^{d}$ ,~$\mu, \mu_1, \mu_2\in \mathcal{P}_{2}(\RR^{d})$, where $C_R=(L_1+2K_R)/(p_0-1)$ and  $ {L}_3=L_1/(p_0-1)$.
\end{rem}

{\begin{rem} 
Although Assumptions 1-3 do not require specific functional forms for $f(x,\mu)$ and $g(x,\mu)$, they are in fact compatible with various types of linear interactions between the state variable and the measure variable, such as a linear type kernel of the form 
$$f(x,\mu)=\hat{f}\Big(x,\int_{\RR^{d}}h_1(x,z)\mu(\mathrm{d}z)\Big),~~g(x,\mu)=\hat{g}\Big(x,\int_{\RR^{d}}h_2(x,z)\mu(\mathrm{d}z)\Big).$$ 
where functions $\hat{f}:\RR^{d}\times \RR^{d}\rightarrow\RR^{d}$, $\hat{g}:\RR^{d}\times \RR^{d}\rightarrow\RR^{d\times m}$ and $h_1, h_2:\RR^{d}\times \RR^{d}\rightarrow \RR^{d}$ satisfy the following assumptions.  
\begin{itemize}
\item[(H1)] For any $R>0$, there exists a constant $\bar{K}_{R}>0$ and $\bar{L}>0$ such that for any $x_1, x_2 \in \RR^{d}$ with $|x_1|\vee|x_2|\leq R$ and $y_1, y_2\in \RR^{d}$,
\begin{align*}
&\big|\hat{f}(x_1,y_1)-\hat{f}(x_2,y_2)\big|\leq \bar{K}_{R}|x_1-x_2|+\bar{L}|y_1-y_2|,\\
&\big|\hat{g}(x_1,y_1)-\hat{g}(x_2,y_2)\big|^2\leq \bar{K}_{R}|x_1-x_2|^2+\bar{L}|y_1-y_2|^2.
\end{align*}
Furthermore, there exists a constant $\bar{K}>0$ such that for any $x_1,x_2,y_2,y_2\in \RR^{d}$,
\begin{align*}
|h_1(x_1,y_1)-h_1(x_2,y_2)|\vee|h_2(x_1,y_1)-h_2(x_2,y_2)|\leq \bar{K}(|x_1-x_2|+|y_1-y_2|).
\end{align*}
\item[(H2)] There exists a pair of constants  $p_0\geq2$  and $\bar{L}_1>0$ such that  for any   $x_1,x_2,y_2,y_2\in \RR^{d}$,
\begin{align*}
2(x_1-x_2)^{T}\big(\hat{f}(x_1,y_1)&-\hat{f}(x_2,y_2)\big)+(p_0-1)\left|\hat{g}(x_1,y_1)-\hat{g}(x_2,y_2)\right|^2\nn\
\\&\leq \bar{L}_1\big(|x_1-x_2|^2+|y_1-y_2|^2\big).
\end{align*}
\item[(H3)] 
There exists a pair of constants $p\geq2$ and $\bar{L}_2>0$ such that for any $x,y\in \RR^{d}$,
\begin{align*}
2x^{T}\hat{f}(x,y)+(p-1)|\hat{g}(x,y)|^2\leq \bar{L}_2(1+|x|^{2}+|y|^2).
\end{align*}

\end{itemize}
\end{rem}}

{Drawing upon the proof methodologies of the well-posedness of the solutions for MV-SDEs in \cite[Theorems 3.1, 3.2]{Ren} and MV-SDEs with common noise in \cite[Theorem 2.1]{MR4497846}, we obtain the well-posedness  and the moment {boundedness} of the solutions to MV-SDE \eqref{eq3.1}. To avoid duplication we omit the proof.}
\begin{lemma}\label{r3.3} 
Let Assumptions \ref{ass1}-\ref{ass3} and  $X_0\in L^{p}_{0}$ hold.  \eqref{eq3.1} has a unique strong solution $X_t$
on $[0,\infty)$
such that
\begin{align*}
\sup_{t\in[0,T]}\mathbb{E}|X_{t}|^{p}\leq C_{T},~~~\forall T\geq0.
\end{align*}
\end{lemma}

Compared with SDEs,  the key step in discretizing the MV-SDEs is to  approximate the marginal distribution law $\mathcal{L}_{t}^{X}$. Inspired by the propagation of chaos \cite{SAS1991,LD2018}  the empirical law  of  the interacting particles  
is a good candidate to approximate the  original  distribution. 
For any integer  $M\geq 1$, let  $(\{B^{i}_{t}\}_{t\geq0},~X^{i}_0)~ ( i=1,\cdots, M)$  be mutually  independent copies of $\left(\{B_{t}\}_{t\geq0},X_{0}\right)$ on the same probability space $(\Omega,\mathcal{F},\mathbb{P}, \{\mathcal{F}_{t}\}_{t\geq0})$.   Define the  interacting particle  system (IPS)
\begin{equation}\label{eq2}
\left\{
\begin{array}{l}
\mathrm{d}X^{i,M}_{t}=f\left(X^{i,M}_{t},\mathcal{L}_{t}^{X,M}\right)\mathrm{d}t+g\left(X^{i,M}_{t},\mathcal{L}_{t}^{X,M}\right)
\mathrm{d}B^{i}_{t},~~~t\geq0,\\
X^{i,M}_{t}|_{t=0}=X^{i}_{0},~~~i=1,2,\cdot\cdot\cdot,M,
\end{array}
\right.
\end{equation}
where
$
\mathcal{L}^{X,M}_{t}=\frac{1}{M}\sum_{i=1}^{M}\boldsymbol{\delta}_{X^{i,M}_{t}}. 
$
Owing to the symmetric structure of  the IPS 
all particles are identically distributed.  Given that the IPS is  an $(\RR^{d})^{M}$-dimensional  SDE,  we know that \eqref{eq2} is strongly well-posed under Assumptions \ref{ass1} and \ref{ass3} for  $X_0\in {  L^{p}_{0}}$,
referring to \cite[p.58, Theorem 3.5]{MR2380366}. Due to the symmetry,  using the H$\ddot{\hbox{o}}$lder inequality one observes that 
 \begin{align*}
 \mathbb{E}\Big(\mathcal{L}^{X,M}_{t }(|\cdot|^2) \Big)^{\frac{p}{2}}\leq  \mathbb{E} |X^{i,M}_{t }|^p,~~t\geq 0 ,~~i=1,\cdots, M.
 \end{align*}  Moreover, we can prove that the unique strong solution is bounded in the $p$th moment and estimate the  finite-time escape probability from a ball.  The routine proof is omitted.
\begin{lemma}\label{le2.3}
Let Assumptions \ref{ass1}-\ref{ass3} and $X_0\in L^{p}_{0}$  hold. Then, for any $T>0$, ~~ 
$
\sup_{M\geq 1}\sup_{1\leq i\leq M}\sup_{0\leq t\leq T}\mathbb{E}|X^{i,M}_{t}|^{p}\leq C_{T}. 
$
Furthermore, for any $M\geq 1$, $1\leq i\leq M$ and $R>0$, define the stopping time
\begin{align}\label{eql2.9}
\tau^{i,M}_{R}=\inf\left\{t\geq0: |X^{i,M}_{t}|\geq R\right\}.
\end{align}
Then for any $T>0$,~~
$\displaystyle 
\mathbb{P}\left(\tau^{i,M}_{R}\leq T\right)\leq \frac{C_{T}}{R^{p}}.
$
\end{lemma}

 Also, consider the non-interacting particle system {(N-IPS)}
\begin{equation}\label{eq3}
\left\{
\begin{array}{l}
\mathrm{d}X^{i}_{t}=f(X^{i}_{t},\mathcal{L}^{X^{i}}_{t})\mathrm{d}t+g(X^{i}_{t},\mathcal{L}^{X^{i}}_{t})\mathrm{d}B^{i}_{t},~~~\forall t\geq0,\\
X^{i}_{t}|_{t=0}=X_0^i,~~~i=1,2,\cdot\cdot\cdot,M.
\end{array}
\right.
\end{equation}
 Obviously, $ \mathcal{L}^{X^i}_{t}=\mathcal{L}^{X}_{t},$   $i=1,2,\cdot\cdot\cdot,M$, for any $ t\geq0$. The propagation of chaos indicates that the single particle  in IPS (\ref{eq2}) tends to be independent and  converges to the  solution of  the MV-SDE, when the original MV-SDE (\ref{eq3.1}) is viewed as a single particle in the {N-IPS}. It is the basis for constructing an appropriate numerical scheme for the IPS and establishing the convergence via  the stochastic particle method. We next  recall a useful result from \cite{Fournier}, followed by a presentation of the propagation of chaos.
 
\begin{lemma}[{{\cite[Theorem 1]{Fournier}}}]\label{Lem3.1}
Let   $\left\{Z_m\right\}_{m\geq 1}$ be an independent and identically distributed (i.i.d.) sequence of random variables with common distribution $\mu\in \mathcal{P}_{\tilde{q}}(\RR^d)$ with $\tilde{q}>2$. Then for  any $q\in[2,
  \tilde{q}) $, there exists a constant $C_{q, \tilde{q}, d}$ such that for all $M \geq 1$, $$\mathbb{E}\left(\mathbb{W}_q^q\left(\mathcal{L}^{Z,M}, \mu\right)\right)\leq C_{q, \tilde{q}, d}  \mu(|\cdot|^{\tilde{q}})^{\frac{q}{\tilde{q}}} \Upsilon_{M,q,\tilde{q},d},$$ where $\mathcal{L}^{Z,M}=\frac{1}{M} \sum_{m=1}^M \boldsymbol{\delta}_{Z_m},$ and
$$
\Upsilon_{M,q,\tilde{q},d}:=\begin{cases}M^{-1 / 2}+M^{-(\tilde{q}-q) / \tilde{q}}, &\text { if }  q>d / 2 \text { and } \tilde{q} \neq 2 q, \\ M^{-1 / 2} \log (1+M)+M^{-(\tilde{q}-q) / \tilde{q}}, & \text { if } q=d / 2 \text { and } \tilde{q} \neq 2 q, \\ M^{-q / d}+M^{-(\tilde{q}-q) / \tilde{q}}, &\text { if } 2 \leq q<d / 2~\text{and}~\tilde{q}\neq\frac{d}{d-q}.\end{cases}
$$
\end{lemma}

{Classical propagation of chaos results over finite time horizons under superlinear growth conditions have been well established in  the  literatures, such as \cite[Proposition 3.1]{Biswas}, \cite[Proposition 2.5]{MR4727110} and \cite[Proposition 1]{MR4497846}. These results  
primarily hold within a second-moment framework. To establish  the strong convergence of the numerical method  we need  a propagation of chaos   in the  higher-order moment sense. 
By the similar techniques as \cite[Theorem 3.2]{MR4667613},  together with  Lemma \ref{Lem3.1}, we yield the error between the IPS and the N-IPS in the $L^q$ sense as follows. The proof employs standard techniques and is thus omitted for brevity.}
\begin{lemma}[{{\cite[Theorem 3.2]{MR4667613}}}]
\label{lec3.4}
Let Assumptions \ref{ass1}-\ref{ass3} hold with $p>2$ and  $X_0\in L^{p}_{0}$. Then for any $q\in [2,p)\wedge[2,p_0]$ and $T>0$,  there exists a constant $C(=C_{q, p, d, T})$ such that
\begin{align}\label{eqcc3.3}
\Big(\sup_{1\leq i\leq M}\sup_{0\leq t\leq T}\mathbb{E}\mathbb{W}^{q}_{q}\big(\mathcal{L}^{X^i}_{t},\mathcal{L}^{X,M}_{t}\big)\Big)\vee\Big(\sup_{1\leq i\leq M}\sup_{0\leq t\leq T}\mathbb{E}\big|X^{i}_{t}-X^{i,M}_{t}\big|^{q}\Big)\leq C \Upsilon_{M,q,p,d},
\end{align}
\end{lemma}

\section{ TEM Scheme and Strong Convergence }\la{s-c}

This section aims to construct an easily implementable explicit scheme for the IPS and to establish the theory on the strong convergence between the numerical solution and the exact solution of IPS. Then, the desired convergence between the numerical and the exact  solutions of the original MV-SDEs, as one single particle of the corresponding {N-IPS}, follows from the propagation of chaos Lemma \ref{lec3.4} directly.

 To adapt the superlinear drift and diffusion, we introduce a truncation device. Precisely,
owing to \eqref{LL2.5} and \eqref{LL2.7}, we can choose a strictly increasing continuous function~$\varphi: \RR_{+}\rightarrow\RR_{+}$ satisfying $\varphi(u)\rightarrow\infty$ as $u\rightarrow\infty$ such that
\begin{align*}
\sup_{\substack{
|x_1|\vee|x_2|\leq u\\
x_1\neq x_2}}\frac{|f(x_1,\mu)-f(x_2,{ \mu})|}{|x_1-x_2|}\vee\frac{|g(x_1,\mu)-g(x_2,{ \mu})|^2}{|x_1-x_2|^2}\leq \varphi(u),~~~ \forall u>0 
\end{align*} for $x_1,x_2\in \RR^{d}$ and $\mu \in \mathcal{P}_{2}(\RR^{d})$.
 For any given~$\Delta\in(0,1]$, define the truncation  mapping
\begin{align}\label{eqq4.1}
\pi_{\Delta}(x)=\displaystyle\left\{\begin{array}{lcl} x,~~&~ |x|\leq \varphi^{-1}(H\Delta^{-\kappa}),\\
\displaystyle \varphi^{-1}(H\Delta^{-\kappa}) \frac{x}{|x|},~~&~ |x|> \varphi^{-1}(H\Delta^{-\kappa}) 
\end{array}\right.
\end{align}
for $x\in \RR^d$, where $\varphi^{-1} $ 
is the inverse function of $\varphi$,  the constants $\kappa\in (0, 1/3]$ and $H\geq 1 \vee  K\vee L_3\vee |f(0,\boldsymbol{\delta}_0)|\vee|g(0,\boldsymbol{\delta}_0)|^2$.

\begin{rem} \label{rm3.1}One observes that for any  $x_1,x_2\in \RR^{d}$  and $\mu \in \mathcal{P}_{2}(\RR^{d})$,
\begin{align*}
\big|f(\pi_{\Delta}(x_1),\mu)-f(\pi_{\Delta}(x_2),\mu)\big|\leq  H\Delta^{-\kappa}|\pi_{\Delta}(x_1)-\pi_{\Delta}(x_2)|\leq  H\Delta^{-\kappa}|x_1-x_2|,
\end{align*}
and
\begin{align*}
\big|g(\pi_{\Delta}(x_1),\mu)-g(\pi_{\Delta}(x_2),\mu)\big|^2\leq H\Delta^{-\kappa} |\pi_{\Delta}(x_1)-\pi_{\Delta}(x_2)|^2\leq H\Delta^{-\kappa} |x_1-x_2|^2.
\end{align*}
Moreover, for any $x\in \RR^{d}$ and $\mu\in \mathcal{P}_2(\RR^{d})$, 
\begin{align}\label{eqc4.4}
\big|f(\pi_{\Delta}(x),\mu)\big|&\leq \big|f(\pi_{\Delta}(x),\mu)-f(0, \mu)\big|+\big|f(0, \mu)-f(0,\boldsymbol{\delta}_0)\big|+ { \big|f(0,\boldsymbol{\delta}_0)\big|} \nn\
\\&\leq H\big(1+\Delta^{-\kappa} |\pi_{\Delta}(x)|+ \mathbb{W}_{2}(\mu,\boldsymbol{\delta}_0)\big)  
 \leq H\big(1+\Delta^{-\kappa}|\pi_{\Delta}(x)|+\mu^{\frac{1}{2}}(|\cdot|^2) \big),
\end{align}
and
\begin{align}\label{eqc4.5}
\big|g(\pi_{\Delta}(x),\mu)\big|^2\leq 3H \big(1+\Delta^{-\kappa}|\pi_{\Delta}(x)|^2+\mu(|\cdot|^2) \big).
\end{align}
\end{rem}

Now, we propose the TEM scheme as follows.
For   $\Delta\in(0,1]$, 
\begin{equation}\label{eq2.1}
\begin{cases}\bar Y_{0}^{i,M}= X_{0}^{i}, \quad i=1,\cdots ,M,\\
	Y_{t_{k }}^{i,M}=\pi _{\Delta}(\bar{Y}_{t_{k }}^{i,M}),~ ~ 
	\mathcal{L} _{t_k}^{Y,M}=\displaystyle \frac{1}{M}\sum_{i=1}^M{\boldsymbol{\delta} _{Y_{t_k}^{i,M}}},~~ k=0,1,\cdots,\\
	\bar{Y}_{t_{k+1}}^{i,M}=Y_{t_k}^{i,M}+f(Y_{t_k}^{i,M},\mathcal{L} _{t_k}^{Y,M})\Delta +g(Y_{t_k}^{i,M},\mathcal{L} _{t_k}^{Y,M})\Delta B_{t_k}^{i}, 
\end{cases}
\end{equation}
where~$t_{k}=k\Delta $,  $\Delta B^{i}_{t_k}=B^{i}_{t_{k+1}}-B^{i}_{t_{k}}$.
   By virtue of  \eqref{eqc4.4} and \eqref{eqc4.5} this scheme has the properties
\begin{align}\label{eq2.2}
\big|f(Y^{i,M}_{t_k},\mathcal{L}_{t_k}^{Y,M})\big|\leq H\Big(1+\Delta^{-\kappa}|Y^{i,M}_{t_k}|+
\big(\mathcal{L}_{t_k}^{Y,M}(|\cdot|^2)\big)^{\frac{1}{2}} \Big),
\end{align}
\begin{align}\label{eq2.3}
\big|g(Y^{i,M}_{t_k},\mathcal{L}^{Y,M}_{t_k})\big|^{2}\leq 3H \left(1+\Delta^{-\kappa}|Y^{i,M}_{t_k}|^2+\mathcal{L}^{Y,M}_{t_k}(|\cdot|^2) \right).
\end{align}
Furthermore, define the  TEM  numerical solutions as
\begin{align}\label{3.9}
 \bar{Y}^{i,M}_{t} =\bar{Y}^{i,M}_{t_k} ,~~~~ {Y}^{i,M}_{t} = {Y}^{i,M}_{t_k},~~\forall ~t\in[t_{k},t_{k+1}), ~~   i=1,\cdots, M.
\end{align}
Due to the symmetry, the  distributions of $Y_{t_k}^{i, M}, ~i=1,\cdots, M$, are identical.
Using the H$\ddot{\hbox{o}}$lder inequality one observes that for any fixed $p\geq 2,$
 \begin{align}\label{eq*Y}
 \mathbb{E}\Big(\mathcal{L}^{Y,M}_{t_k}(|\cdot|^2) \Big)^{\frac{p}{2}}\leq \E \Big(\frac{1}{M}\sum_{i=1}^{M} |Y^{i,M}_{t_k}|^{2}\Big)^{\frac{p}{2}}\leq \mathbb{E} |Y^{i,M}_{t_k}|^p,~~t\geq 0 ,~~i=1,\cdots, M.
 \end{align}    
   
   {To illustrate how to choose appropriate constants $H,~\kappa$ and function $\varphi$ and how to construct  the TEM scheme, we recall  the MV-SDE \eqref{Ne4} with
$$f(x,\mu)=x(-2-|x|)+\int x\mu(\mathrm{d}x), ~~~g(x,\mu)=|x|^{3/2}/2.$$ 
Obviously, Assumption \ref{ass1} is satisfied with $K=1$. For any $  x_1, x_2\in \RR$ and $\mu \in \mathcal{P}_{2}(\RR)$, we compute  
\begin{align*}
&\sup_{\substack{
|x_1|\vee|x_2|\leq u\\
x_1\neq x_2}}\frac{|f(x_1,\mu)-f(x_2,{ \mu})|}{|x_1-x_2|}\leq \frac{2|  x_1-x_2 |+\Big|x_1|x_1|-x_2|x_2|\Big|}{|x_1-x_2|} \leq 2(1+u) ,~~~ \forall u>0, 
\\&\sup_{\substack{
|x_1|\vee|x_2|\leq u\\
x_1\neq x_2}}  \frac{|g(x_1,\mu)-g(x_2,{ \mu})|^2}{|x_1-x_2|^2}\leq \frac{\Big( |x_1|^{\frac{1}{2}}- |x_2|^{\frac{1}{2}}\Big)^2 \Big( |x_1|^{\frac{1}{2}}+ |x_2|^{\frac{1}{2}}\Big)^4}{4|x_1-x_2|^2} \leq u ,~~~ \forall u>0. 
\end{align*}  Thus, we may choose a function $\varphi(u)=4(1+u)> 2(1+u)\vee u,~u>0.$ 
One notes that Assumptions \ref{ass2} and \ref{ass3} hold with  $ p_0= 25/9$, $p= 9$, and $L_1=L_2=1$, $L_3=9/16$, respectively. $f(0,\boldsymbol{\delta}_0)=g(0,\boldsymbol{\delta}_0)=0$. 
 Then we may choose constants   $  \kappa= {1}/{3}\in (0,  {1}/{3}],~~H=50\geq 1,  $ and compute 
\begin{align*}
\pi_{\Delta}(x)=\displaystyle\left\{\begin{array}{lcl} x,~~&~ |x|\leq 12.5\Delta^{-\frac{1}{3}} -1,\\
\displaystyle \big(12.5\Delta^{-\frac{1}{3}} -1\big)\frac{x}{|x|},~~&~ |x|>  12.5\Delta^{-\frac{1}{3}} -1.
\end{array}\right.
\end{align*}
Thus, the TEM scheme for the MV-SDE \eqref{Ne4} is described by
\begin{equation} \label{eq7.3}
\begin{aligned}
\left\{
\begin{array}{rl}
\bar{Y}^{i,M}_{0}&=X^{i}_0, ~~i=1,2,\cdots,M,\\
Y^{i,M}_{t_{k}}&=\pi_{\Delta}(\bar{Y}^{i,M}_{t_{k}}), ~~\mathcal{L} _{t_k}^{Y,M}=\displaystyle \frac{1}{M}\sum_{i=1}^M{\boldsymbol{\delta} _{Y_{t_k}^{i,M}}},~~ k=0,1,\cdots,\\
\bar{Y}^{i,M}_{t_{k+1}}&=Y^{i,M}_{t_k}+Y^{i,M}_{t_k}\Big(-2-\big|Y^{i,M}_{t_k}\big|+
\frac{1}{M}\sum\limits_{j=1}^{M}Y^{j,M}_{t_k}\Big)\Delta+\frac{1}{2}\big|Y^{i,M}_{t_k}\big|^{\frac{3}{2}} \Delta B^{i}_{t_k}.
\end{array}
\right.
\end{aligned}
\end{equation} }

\begin{rem} 
If
 the coefficients $f(x, \mu )$ and $g(x, \mu )$ are globally Lipschitz continuous with respect to  $x$, then one may take $\varphi(u)\equiv L$ (the global Lipschitz constant) and  $ \varphi^{-1}(u)=+\infty$. Thus, $\pi _{\Delta}(x)=x$ for any $\Delta\in (0, 1]$, $x\in \RR^{d}$, which implies that the TEM scheme \eqref{eq2.1} becomes  the standard EM scheme. So, the EM scheme is a special case of the TEM.
\end{rem}

\begin{rem}
Compared with the tamed numerical methods \cite{arXivBJ,MR4302574,MR4497846} which directly modify the coefficient terms, our scheme  only adjusts the inappropriate grid point values of the EM iteration by the truncation mapping. By this scheme the numerical solutions preserve the geometric structure of the coefficients  (Assumptions \ref{ass5} and \ref{ass6}) perfectly.  
However, the pointwise nature of this correction method precludes a continuity-based proof of our main results such as using the  It\^o formula directly, compelling us to introduce a new  analysis method based on piecewise continuity. Successive applications of the  It\^o formula yields a recursive difference inequality; solving it implies that the numerical solutions rigorously preserve the  same dynamic properties as the exact ones.  
\end{rem}

 Our main aim is  to establish the convergence theory of the  TEM scheme in both finite  and infinite horizons. We first establish the boundedness of moments of the TEM numerical solutions for the IPS over a finite time horizon. In fact,  the IPS is $M\times d$ dimensional and  depends on the empirical distribution of particles.  These  bring  us essential difficulties in the moment estimation of the TEM numerical solutions.  
To analyze the moment boundedness of $Y^{i,M}_{t_k}$ itself, we {introduce  \eqref{eq4.35}
as the continuous time scheme}
\begin{align}\label{eq4.35}
\tilde{Y}^{i,M}_{t}&=Y^{i,M}_{t_k}+\int_{t_{k}}^{t}f(Y^{i,M}_{t_k},\mathcal{L}_{t_k}^{Y,M})\mathrm{d}s+\int_{t_{k}}^{t}g(Y^{i,M}_{t_k},\mathcal{L}_{t_k}^{Y,M})\mathrm{}\mathrm{d}B^{i}_{s},~~\forall t\in[t_k, t_{k+1}),
\end{align} which is right continuous and has the left limit at each grid point, that is, 
\begin{align}\label{eqcc3.37}
\lim_{t\rightarrow t^{+}_{k+1}}\tilde{Y}^{i,M}_{t}=\tilde{Y}^{i,M}_{t_{k+1}}=Y^{i,M}_{t_{k+1}},~~~~\lim_{t\rightarrow t^{-}_{k+1}}\tilde{Y}^{i,M}_{t}=\bar{Y}^{i,M}_{t_{k+1}},~~k=0,1,2, \cdots~~\mathrm{a.s.}
\end{align} 
 Thus, { $\tilde{Y}^{i,M}_t$}  becomes a bridge between   $Y^{i,M}_{t_{k }}$ and $\bar{Y}^{i,M}_{t_{k }}$.
 { By the continuity of the auxiliary process $\tilde{Y}^{i,M}_{t}$ in each interval $[t_{k},t_{k+1})$,  we can  
  obtain the bounds of   $\mathbb{E}|Y^{i,M}_{t_{k}}|^{p}$ and $\mathbb{E}|\bar{Y}^{i,M}_{t_{k}}|^{p}$.} 

\begin{theorem}\label{le3.4}
Let Assumptions \ref{ass1}-\ref{ass3} and  $X_0\in L^{p}_{0}$ hold. Then, 
\begin{align*}
\sup_{M\geq 1}\sup_{1\leq i\leq M}\sup_{0\leq \Delta\leq 1}\sup_{0\leq t \leq T}\Big(\mathbb{E}|Y^{i,M}_{t }|^{p}\vee\mathbb{E}|\bar{Y}^{i,M}_{t }|^{p}\vee\mathbb{E}| {\tilde{Y}}^{i,M}_{t }|^{p}\Big)\leq C_T,~~ \forall T>0.
\end{align*}
\end{theorem}
\begin{proof} Fix any $T>0$. For any $ ~M\geq 1,~ 1\leq i\leq M$, $ 0\leq \Delta\leq 1$ and $0\leq t_k\leq T$, 
by the It\^o formula, we derive from \eqref{eq4.35} and Assumption \ref{ass3} that for any $t\in[t_{k},t_{k+1})$,
\begin{equation*}
\begin{aligned}
\mathbb{E}|\tilde{Y}^{i,M}_{t}|^{p}
 &\leq \mathbb{E}|Y^{i,M}_{t_{k}}|^{p}+\frac{p}{2}\mathbb{E}\int_{t_k}^{t}|\tilde{Y}^{i,M}_{s}|^{p-2}
\Big[2(Y^{i,M}_{t_{k}})^{T}f(Y^{i,M}_{t_{k}},\mathcal{L}^{Y,M}_{t_{k}})+(p-1)\big|g(Y^{i,M}_{t_{k}},\mathcal{L}^{Y,M}_{t_{k}})
\big|^{2}\Big]\mathrm{d}s
\\&~~~+p\mathbb{E}\int_{t_{k}}^{t}|\tilde{Y}^{i,M}_{s}|^{p-2}|\tilde{Y}^{i,M}_{s}-Y^{i,M}_{t_{k}}|
\big|f(Y^{i,M}_{t_{k}},\mathcal{L}^{Y,M}_{t_{k}})\big|\mathrm{d}s \label{eq*+1}
\\&\leq \mathbb{E}|Y^{i,M}_{t_{k}}|^{p}+\frac{pL_2}{2}\mathbb{E}\int_{t_{k}}^{t}|\tilde{Y}^{i,M}_{s}|^{p-2}
\big(1+|Y^{i,M}_{t_{k}}|^{2}+\mathcal{L}^{Y,M}_{t_{k}}(|\cdot|^2)\big)\mathrm{d}s
\\&~~~+p\mathbb{E}\int_{t_{k}}^{t}|\tilde{Y}^{i,M}_{s}|^{p-2}|
\tilde{Y}^{i,M}_{s}-Y^{i,M}_{t_{k}}|\big|f(Y^{i,M}_{t_{k}},\mathcal{L}^{Y,M}_{t_{k}})\big|\mathrm{d}s.
\end{aligned}
\end{equation*}
Applying \eqref{eq*Y}, the Young inequality and the elementary inequality 
we derive that
\begin{align}\label{eq3.36}
\mathbb{E}|\tilde{Y}^{i,M}_{t}|^{p}
 &\leq  C\Delta+(1+C\Delta)\mathbb{E}|Y^{i,M}_{t_{k}}|^{p}+C\int_{t_{k}}^{t}\mathbb{E}|\tilde{Y}^{i,M}_{s}-Y^{i,M}_{t_k}|^{p}\mathrm{d}s \nn\\& ~~+C\int^{t}_{t_{k}}\mathbb{E}\Big(|\tilde{Y}^{i,M}_{s}-Y^{i,M}_{t_{k}}|^{\frac{p}{2}}\big|f(Y^{i,M}_{t_{k}},\mathcal{L}^{Y,M}_{t_{k}})
\big|^{\frac{p}{2}}\Big)\mathrm{d}s.
\end{align}
For any  $s\in [t_{k},t ]$,  applying \eqref{eq4.35}, \eqref{eq2.2}, \eqref{eq2.3} and \eqref{eq*Y}
 yields that
\begin{align}\label{eq3.32}
 \mathbb{E}|\tilde{Y}^{i,M}_{s}-Y^{i,M}_{t_k}|^{p}&\leq2^{p-1}\left(\mathbb{E}\big|f(Y^{i,M}_{t_{k}},\mathcal{L}^{Y,M}_{t_{k}})\big|^{p}\Delta^{p}+\mathbb{E}\big|g(Y^{i,M}_{t_{k}},\mathcal{L}^{Y,M}_{t_{k}})\big|^{p}\Delta^{\frac{p}{2}}\right)\nn\
\\& \leq C\Delta^{p}\Big( \Delta^{-\kappa p} \mathbb{E}|Y^{i,M}_{t_{k}}|^{p}+\mathbb{E}|Y^{i,M}_{t_{k}}|^{p}+1\Big)
+C\Delta^{\frac{p}{2}}  \Big( \Delta^{-\frac{\kappa p} {2}} \mathbb{E}|Y^{i,M}_{t_{k}}|^{p}+\mathbb{E}|Y^{i,M}_{t_{k}}|^{p}+1\Big) \nn\
\\&\leq C\big(\Delta^{\frac{p}{2}}+\Delta^{\frac{p(1-\kappa)}{2}}\mathbb{E}|Y^{i,M}_{t_{k}}|^{p}\big)\leq C\Delta^{\frac{p(1-\kappa)}{2}}\big(1+\mathbb{E}|Y^{i,M}_{t_{k}}|^{p}\big).
\end{align}
Combining this with \eqref{eq2.2} and \eqref{eq*Y}, and  using the H\"older inequality and the Young inequality,  we derive that for any $s\in [t_{k},t ]$,
\begin{align}\label{eq3.35}
&\mathbb{E}\Big(|\tilde{Y}^{i,M}_{s}-Y^{i,M}_{t_k}|^{\frac{p}{2}}
\big|f(Y^{i,M}_{t_{k}},\mathcal{L}^{Y,M}_{t_{k}})\big|^{\frac{p}{2}}\Big)\leq \Big(\mathbb{E}|\tilde{Y}^{i,M}_{s}-Y^{i,M}_{t_{k}}|^{p}\Big)^{\frac{1}{2}}\Big(\mathbb{E}\big|f(Y^{i,M}_{t_{k}},\mathcal{L}^{Y,M}_{t_{k}})\big|^{p}\Big)^{\frac{1}{2}}\nn\
\\ \leq &  C\Delta^{\frac{p(1- \kappa)}{4}}\big(1+\mathbb{E}|Y^{i,M}_{t_{k}}|^{p}\big)^{\frac{1}{2}}
\left(\Delta^{-\kappa p} \mathbb{E}|Y^{i,M}_{t_{k}}|^{p}+\mathbb{E}|Y^{i,M}_{t_{k}}|^{p}+1\right)^{\frac{1}{2}} 
\leq C\Delta^{\frac{p(1-3\kappa)}{4}}\big(1+\mathbb{E}|Y^{i,M}_{t_{k}}|^{p}\big).
\end{align}
Then inserting \eqref{eq3.32} and \eqref{eq3.35} into \eqref{eq3.36} and using $\kappa\in (0, 1/3]$ yields that
\begin{align*}
\mathbb{E}|\tilde{Y}^{i,M}_{t}|^{p}&\leq (1+C\Delta)\mathbb{E}|Y^{i,M}_{t_{k}}|^{p}+C\Delta\big(1+\mathbb{E}|Y^{i,M}_{t_{k}}|^{p}\big),
\end{align*}
which implies that
\begin{align}\label{eqY-t}
1+\mathbb{E}|\tilde{Y}^{i,M}_{t}|^{p}
\leq (1+C\Delta)\left(1+\mathbb{E}|Y^{i,M}_{t_{k}}|^{p}\right). 
\end{align}
By virtue of the Fatou Lemma, we derive from \eqref{eqcc3.37} and \eqref{eqY-t} that
\begin{align*}
 1+\mathbb{E}|\bar{Y}^{i,M}_{t_{k+1}}|^{p}&=1+ \mathbb{E}\Big(\liminf_{t\rightarrow t^{-}_{k+1}}|\tilde{Y}^{i,M}_{t}|^{p}\Big)\leq 1+\liminf_{t\rightarrow t^{-}_{k+1}}\mathbb{E}|\tilde{Y}^{i,M}_{t}|^{p}\\
&\leq (1+C\Delta)\left(1+\mathbb{E}|Y^{i,M}_{t_{k}}|^{p}\right)\leq (1+C\Delta)\left(1+\mathbb{E}|\bar{Y}^{i,M}_{t_{k}}|^{p}\right).
\end{align*}
Due to the inequality $1+x\leq e^{x}$ for $  x\geq 0$, solving the above difference inequality leads to
\begin{align*}
\mathbb{E}|Y^{i,M}_{t_{k}}|^{p}\leq\mathbb{E}|\bar{Y}^{i,M}_{t_{k}}|^{p}&\leq (1+C\Delta)(1+\mathbb{E}|\bar{Y}^{i,M}_{t_{k-1}}|^{p})\leq \cdots\leq (1+C\Delta)^k \big(1+\mathbb{E}|\bar{Y}^{i,M}_{0}|^{p}\big)\nn\
\\&\leq e^{Ck\Delta}(1+\mathbb{E}|X^{i}_0|^{p})\leq e^{CT}\big(1+\mathbb{E}|X_{0}|^{p}\big)\leq C_T.
\end{align*}
Therefore, the desired results follow from the arbitrariness of $M$ and \eqref{eqY-t}.
\end{proof}

{ By virtue of Theorem \ref{le3.4} we can obtain the upper bounds of  $\mathbb{E}|\tilde{Y}^{i,M}_{t}-Y^{i,M}_{t}|^{p}$  from \eqref{eq3.32} and 
the related term from \eqref{eq3.35}. }
\begin{cor}\label{cor2}
Let Assumptions \ref{ass1}-\ref{ass3} and  $X_0\in L^{p}_{0}$ hold. For any  $  T>0$,  we have   
\begin{align*}
&\sup_{M\geq 1}\sup_{1\leq i\leq M}\sup_{0< \Delta\leq 1}\sup_{0\leq t\leq T} \mathbb{E}|\tilde{Y}^{i,M}_{t}-Y^{i,M}_{t}|^{p} \leq C_{T}\Delta^{\frac{p}{3}},  
 \\&\sup_{M\geq 1}\sup_{1\leq i\leq M}\sup_{0< \Delta\leq1}\sup_{0\leq t\leq T}\mathbb{E}\Big(|\tilde{Y}^{i,M}_{t}-Y^{i,M}_{t}|^{\frac{p}{2}}\big|f(Y^{i,M}_{t},\mathcal{L}^{Y,M}_{t})\big|^{\frac{p}{2}}\Big)\leq C_{T}.
\end{align*}
\end{cor}
\begin{lemma}\label{L6.14}
Let $\tilde{r}> 0$ and $\zeta\in L^{\tilde{r}}$.  Then for any $\Delta\in (0,1]$  and $ {r}\in (0, \tilde{r})$,
\begin{align*}
\mathbb{E}|\zeta-\pi_{\Delta}(\zeta)|^{ {r}}\leq  \frac{ C\mathbb{E}|\zeta|^{\tilde{r}}}{\big[\varphi^{-1}(H\Delta^{-\kappa})\big]^{\tilde{r}- {r}}}.
\end{align*}
\end{lemma}
\begin{proof}
For any $\Delta\in (0,1]$, define a set $
\Upsilon_{\Delta}:=\big\{\omega:|\zeta(\omega)|\geq \varphi^{-1}(H\Delta^{-\kappa})\big\}.
$
Note that $\pi_{\Delta}(\zeta)=\zeta$ for any $\omega\in \Upsilon^{c}_{\Delta}$. Then using the H\"older inequality one derives that
\begin{align*}
&\mathbb{E}|\zeta-\pi_{\Delta}(\zeta)|^{ {r}}=\mathbb{E}\big(|\zeta-\pi_{\Delta}(\zeta)|^{ {r}}I_{\Upsilon_{\Delta}}\big)+\mathbb{E}\big(|\zeta-\pi_{\Delta}(\zeta)|^{ {r}}I_{\Upsilon^{c}_{\Delta}}\big)\nn\
\\=&\mathbb{E}\big(|\zeta-\pi_{\Delta}(\zeta)|^{ {r}}I_{\Upsilon_{\Delta}}\big)
\leq \big(\mathbb{E}\big(|\zeta-\pi_{\Delta}(\zeta)|^{\tilde{r}}\big)^{\frac{^{ {r}}}{\tilde{r}}}\big(\mathbb{P}(\Upsilon_{\Delta})\big)^{\frac{\tilde{r}- {r}}{\tilde{r}}}\leq 2^{ {r}}\big(\mathbb{E}|\zeta|^{\tilde{r}}\big)^{\frac{ {r}}{\tilde{r}}}\big(\mathbb{P}(\Upsilon_{\Delta})\big)^{\frac{\tilde{r}- {r}}{\tilde{r}}}
\end{align*}
Furthermore, using the Chebyshev inequality yields that
$\displaystyle
 \mathbb{P}(\Upsilon_{\Delta}) \leq \frac{\mathbb{E}|\zeta|^{\tilde{r}}}{\big[\varphi^{-1}(H\Delta^{-\kappa})\big]^{\tilde{r}}}
 .
$
Therefore, the desired assertion follows. 
The proof is complete.
\end{proof}

It follows from \eqref{3.9} and \eqref{eq4.35} that $\tilde{Y}^{i,M}_{t}$  is formed by continuously concatenating $Y^{i,M}_{t_k}$ and $\bar{Y}^{i,M}_{t_{k+1}}$  for $t\in [t_k, t_{k+1})$.  Whenever $\tilde{Y}^{i,M}_{t}$ lies within the ball of the  radius $\varphi^{-1}(H\Delta_1^{-\kappa})$, 
 $\bar{Y}^{i,M}_{t}$ also resides within the ball, rendering the truncation mapping inactive. Consequently, we consider the first exit time of $\tilde{Y}^{i,M}_{t}$ from the ball; up to this time, $\bar{Y}^{i,M}_{t}=Y^{i,M}_{t}$.

\begin{lemma}\label{le3.5}
Let Assumptions \ref{ass1}-\ref{ass3} and  $X_0\in L^{p}_{0}$ hold. 
For any $\Delta_1\in (0,1]$ and $\Delta \in (0,\Delta_1]$, define 
\begin{align}\label{eqpl3.43}
\eta^{i,M}_{\Delta,\Delta_1}=\inf\left\{t\geq0: |\tilde{Y}^{i,M}_{t}|\geq \varphi^{-1}(H\Delta_1^{-\kappa})\right\} 
\end{align}
for $\kappa\in (0, 1/3]$. Then for any $X_0\in L^{p}_{0}$,~~
$ 
\mathbb{P}\left(\eta^{i,M}_{\Delta,\Delta_1}\leq T\right)\leq \displaystyle {C_{T}}{\left[\varphi^{-1}(H\Delta_1^{-\kappa})\right]^{-p}},~~ \forall T>0.
$ 
\end{lemma}
\begin{proof} Fix the constant $\Delta_1\in (0,1]$.
The increase of $\varphi^{-1}$ implies that  $\varphi^{-1} (H\Delta^{-\kappa})\geq\varphi^{-1} (H\Delta_1^{-\kappa})$ for any $\Delta\in(0,\Delta_1]$. Thus, for any $0\leq t\leq T\wedge \eta^{i,M}_{\Delta,\Delta_1}$,
\begin{align*}
\tilde{Y}^{i,M}_{t}=Y^{i,M}_{0}+\int_{0}^{t}f(Y^{i,M}_{s},\mathcal{L}^{Y,M}_{s})\mathrm{d}s+\int_{0}^{t}g(Y^{i,M}_{s},\mathcal{L}^{Y,M}_{s})\mathrm{d}B^{i}_{s},
\end{align*}
where  $\mathcal{L}^{Y,M}_{s}=\frac{1}{M}\sum_{i=1}^{M}\boldsymbol{\delta}_{Y^{i,M}_{s}}.$
Employing techniques as in the proofs of Theorem \ref{le3.4}, we arrive at
\begin{align*}
\mathbb{E}\big|\tilde{Y}^{i,M}_{T\wedge\eta^{i,M}_{\Delta,\Delta_1}}\big|^{p}&\leq \mathbb{E}|Y^{i,M}_{0}|^{p}+\frac{pL_2}{2}\mathbb{E}\int_{0}^{T\wedge\eta^{i,M}_{\Delta,\Delta_1}}|\tilde{Y}^{i,M}_{s}|^{p-2}\left(1+|Y^{i,M}_{s}|^{2}+\mathcal{L}^{Y,M}_{s}(|\cdot|^2)\right)\mathrm{d}s\nn\
\\&~~~+p\mathbb{E}\int_{0}^{T\wedge\eta^{i,M}_{\Delta,\Delta_1}}|\tilde{Y}^{i,M}_{s}|^{p-2}|\tilde{Y}^{i,M}_{s}-Y^{i,M}_{s}|\big|f(Y^{i,M}_{s},\mathcal{L}^{Y,M}_{s})\big|\mathrm{d}s\nn\
\\&\leq \mathbb{E}|Y^{i,M}_{0}|^{p}+CT+C\int_{0}^{T}\mathbb{E}\big|\tilde{Y}^{i,M}_{s}\big|^{p}\mathrm{d}s+C\int_{0}^{T}\mathbb{E}|Y^{i,M}_{s}|^{p}\mathrm{d}s \nn\
\\&~~~+C\int_{0}^{T}\mathbb{E}\Big(|\tilde{Y}^{i,M}_{s}-Y^{i,M}_{s}|^{\frac{p}{2}}\big|f(Y^{i,M}_{s},\mathcal{L}^{Y,M}_{s})\big|^{\frac{p}{2}}\Big)\mathrm{d}s.
\end{align*}
Thus, it follows from  Theorem \ref{le3.4} and Corollary \ref{cor2}   that 
\begin{align*}
\mathbb{E}\big|\tilde{Y}^{i,M}_{T\wedge\eta^{i,M}_{\Delta,\Delta_1}}\big|^{p}&\leq \mathbb{E}|Y^{i,M}_{0}|^{p}+C_{T}\leq \mathbb{E}|X_{0}|^{p}+C_{T}\leq C_{T},
\end{align*}
where the last inequality used the fact  $X_0\in L^{p}_{0}$.
Therefore, for any $T>0$,
\begin{align*}
\big[\varphi^{-1}(H\Delta_1^{-\kappa})\big]^{p}\mathbb{P}\left(\eta^{i,M}_{\Delta,\Delta_1}\leq T\right)\leq \mathbb{E}\big|\tilde{Y}^{i,M}_{T\wedge\eta^{i,M}_{\Delta,\Delta_1}}\big|^{p}\leq C_T.
\end{align*}
Thus, the desired result follows directly.
\end{proof}
 
 For any $q\in (0, p)$, we establish the convergence  of $\mathbb{E}\big|X^{i,M}_{t}-\tilde{Y}^{i,M}_{t}\big|^{q}$, and then by this bridge we obtain the convergence of $\mathbb{E}\big|X^{i,M}_{t}-Y^{i,M}_{t}\big|^{q}$ and  $\mathbb{E}\big|X^{i,M}_{t}-\bar{Y}^{i,M}_{t}\big|^{q}$.
\begin{theorem}\label{Lemc3.6}
Let Assumptions \ref{ass1}-\ref{ass3}{ hold with $p,p_0>2$} and  $X_0\in L^{p}_{0}$. Then for any $q\in (0,  p)$,
\begin{align*}
\lim\limits_{\Delta\rightarrow0}\sup_{ M\geq 1}\sup_{1\leq i\leq M}\sup_{0\leq t\leq T}\Big(\mathbb{E}\big|X^{i,M}_{t}-Y^{i,M}_{t}\big|^{q}\vee \mathbb{E}\big|X^{i,M}_{t}-\bar{Y}^{i,M}_{t}\big|^{q}\Big)=0,~~~\forall~T\geq 0.
\end{align*}
\end{theorem}
\begin{proof} Fix $T\geq 0$ and $q\in [2,p_0\wedge  p)$.
 For any $M\geq 1$, $\Delta_1\in (0,1]$ and $\Delta\in (0,\Delta_1]$, define $$\theta^{i,M}_{\Delta,\Delta_1}=\tau^{i,M}_{\varphi^{-1}(H\Delta_1^{-\kappa})}\wedge\eta^{i,M}_{\Delta,\Delta_1},~~~~ \Phi^{i,M}_t =X^{i,M}_{t}-\tilde{Y}^{i,M}_{t},~~~\forall 1\leq i\leq M, ~~~\forall t\in [0,T].$$
 For any   $0\leq t\leq T $, one notices that
\begin{align*}
\mathbb{E}|\Phi^{i,M}_t|^{q}&=\mathbb{E}\Big(I_{\{\theta^{i,M}_{\Delta,\Delta_1}>T\}}
|\Phi^{i,M}_t|^{q}\Big)+\mathbb{E}\Big(I_{\{\theta^{i,M}_{\Delta,\Delta_1}\leq T\}}|\Phi^{i,M}_t|^{q}\Big)
\\&\leq \mathbb{E}\Big|\Phi^{i,M}_{t\wedge \theta^{i,M}_{\Delta,\Delta_1}}\Big|^{q}+\mathbb{E}\Big(I_{\{\theta^{i,M}_{\Delta,\Delta_1}\leq T\}}|\Phi^{i,M}_t|^{q}\Big).
\end{align*}
For any $\delta>0$, using the Young inequality yields
\begin{align*}\mathbb{E}\Big(I_{\{\theta^{i,M}_{\Delta,\Delta_1}\leq T\}}|\Phi^{i,M}_t|^{q}\Big)
 \leq  
\frac{q\delta }{p}\sup_{0\leq t\leq T}\mathbb{E}|\Phi^{i,M}_t|^{p}+\frac{(p-q)}{p\delta^{\frac{q}{p-q}}}\mathbb{P}\left(\theta^{i,M}_{\Delta,\Delta_1}\leq T\right).
\end{align*} 
By virtue of Lemma  \ref{le2.3} and  Theorem \ref{le3.4}, it follows that
\begin{align*}
\frac{q\delta }{p}\sup_{0\leq t\leq T}\mathbb{E}|\Phi^{i,M}_t|^{p}\leq \frac{q2^{p-1} \delta}{p}\Big(\sup_{0\leq t\leq T}\mathbb{E}|X^{i,M}_{t}|^{p}+\sup_{0\leq t\leq T}\mathbb{E}|\tilde{Y}^{i,M}_{t}|^{p}\Big)\leq \frac{q2^{p-1} \delta C_{T}}{p}.
\end{align*}
By Lemma \ref{le2.3} and Lemma \ref{le3.5}, we derive that for any $\Delta\in (0,\Delta_1]$,
\begin{align*} 
\frac{(p-q)}{p\delta^{\frac{q}{p-q}}}\mathbb{P}\left(\theta^{i,M}_{\Delta,\Delta_1}
\leq T\right)&\leq \frac{(p-q)}{p\delta^{\frac{q}{p-q}}}\left[\mathbb{P}\left(\tau^{i,M}_{\varphi^{-1}(H\Delta_1^{-\kappa})}\leq T\right)+\mathbb{P}\left(\eta^{i,M}_{\Delta,\Delta_1}\leq T\right)\right]\leq \frac{2(p-q)C_{T}}{p\delta^{\frac{q}{p-q}}[\varphi^{-1}(H\Delta_1^{-\kappa})]^{p}}.
\end{align*}
Therefore, for any $0\leq t\leq T $, we have 
\begin{align*}
\mathbb{E}\Big(I_{\{\theta^{i,M}_{\Delta,\Delta_1}\leq T\}}|\Phi^{i,M}_t|^{q}\Big)
 \leq \frac{q2^{p-1} \delta C_{T}}{p}+\frac{2(p-q)C_{T}}{p\delta^{\frac{q}{p-q}}[\varphi^{-1}(H\Delta_1^{-\kappa})]^{p}},
\end{align*}  which implies \begin{align} \label{eqc4.47}
\mathbb{E}|\Phi^{i,M}_t|^{q}
&\leq \mathbb{E}\Big|\Phi^{i,M}_{t\wedge \theta^{i,M}_{\Delta,\Delta_1}}\Big|^{q}+ \frac{q2^{p-1} \delta C_{T}}{p}+\frac{2(p-q)C_{T}}{p\delta^{\frac{q}{p-q}}[\varphi^{-1}(H\Delta_1^{-\kappa})]^{p}}.
\end{align} 
Noticing that  $ \varphi^{-1}(H\Delta_1^{-\kappa}) \leq \varphi^{-1}(H\Delta^{-\kappa})$ for any $\Delta\in (0,\Delta_1]$, utilizing the It\^o formula, we compute 
\begin{align}\label{LCY3.41}
\mathbb{E}\Big|\Phi^{i,M}_{t\wedge\theta^{i,M}_{\Delta,\Delta_1}}\Big|^{q}\nn\
 = &\mathbb{E} \big|\Phi^{i,M}_{0 }\big|^{q} + \mathbb{E}\Big[ \frac{q}{2}\int_{0}^{t\wedge\theta^{i,M}_{\Delta,\Delta_1}}\big|\Phi^{i,M}_{s }\big|^{q-2}\Big(2\big(\Phi^{i,M}_{s}\big)^{T}\big(f(X^{i,M}_{s},\mathcal{L}^{X,M}_{s})-f(Y^{i,M}_{s},\mathcal{L}^{Y,M}_{s})\big)\nn\
\\~~~&~~~~~~~~~~~~~~~~+(q-1)\big|g(X^{i,M}_{s},\mathcal{L}^{X,M}_{s})-g(Y^{i,M}_{s},\mathcal{L}^{Y,M}_{s})
\big|^2\Big)\mathrm{d}s\Big].
\end{align}
By the Young inequality and Assumption \ref{ass2}, we deduce  that
\begin{align*}
 &2\big(\Phi^{i,M}_{s}\big)^{T} \big(f(X^{i,M}_{s},\mathcal{L}^{X,M}_{s})-f(Y^{i,M}_{s},\mathcal{L}^{Y,M}_{s})\big)+(q-1)\big|g(X^{i,M}_{s},\mathcal{L}^{X,M}_{s})-g(Y^{i,M}_{s},\mathcal{L}^{Y,M}_{s})\big|^{2}\nn\
\\ \leq& 2\big(\Phi^{i,M}_{s}\big)^{T}\big(f(X^{i,M}_{s},\mathcal{L}^{X,M}_{s})-f(\tilde{Y}^{i,M}_{s},\mathcal{L}^{\tilde{Y},M}_{s})\big)+(p_0-1)\big|g(X^{i,M}_{s},\mathcal{L}^{X,M}_{s})-g(\tilde{Y}^{i,M}_{s},\mathcal{L}^{\tilde{Y},M}_{s})\big|^2\nn\
\\& +2(\Phi^{i,M}_{s})^{T}\big(f(\tilde{Y}^{i,M}_{s},\mathcal{L}^{\tilde{Y},M}_{s})
-f(Y^{i,M}_{s},\mathcal{L}^{Y,M}_{s})\big) +\Big(q+\frac{1}{p_0-q}\Big)
\big|g(\tilde{Y}^{i,M}_{s},\mathcal{L}^{\tilde{Y},M}_{s})-g(Y^{i,M}_{s},\mathcal{L}^{Y,M}_{s})\big|^{2} 
\\ \leq & L_1\big(|\Phi^{i,M}_{s}|^2+\mathbb{W}^{2}_{2}(\mathcal{L}^{X,M}_{s},
\mathcal{L}^{\tilde{Y},M}_{s})\big)+2\big|\Phi^{i,M}_{s}\big|
\big|f(\tilde{Y}^{i,M}_{s},\mathcal{L}^{\tilde{Y},M}_{s})-f(Y^{i,M}_{s},\mathcal{L}^{Y,M}_{s})\big|\nn\
\\& +\Big(q+\frac{1}{p_0-q}\Big) \big|g(\tilde{Y}^{i,M}_{s},\mathcal{L}^{\tilde{Y},M}_{s})-g(Y^{i,M}_{s},\mathcal{L}^{Y,M}_{s})\big|^{2}.
\end{align*}
Inserting the above inequality into \eqref{LCY3.41} and then employing the Young inequality we arrive at
\begin{align} \label{eqc342}
\mathbb{E} \Big|\Phi^{i,M}_{t\wedge\theta^{i,M}_{\Delta,\Delta_1}}\Big|^{q}
\leq& \mathbb{E}\big|\Phi^{i,M}_{0}\big|^{q}  +C\int_{0}^{t}\mathbb{E}\Big 
|\Phi^{i,M}_{s\wedge\theta^{i,M}_{\Delta,\Delta_1}}
\Big|^{q}\mathrm{d}s
+C\mathbb{E} \int_{0}^{t}  \mathbb{W}^{q}_{2}\big(\mathcal{L}^{X,M}_{s},{ \mathcal{L}^{\tilde{Y},M}_{s}}\big) \mathrm{d}s+C\mathcal{J}^{i,M},
\end{align}
where
\begin{align}\label{eq*J}
\mathcal{J}^{i,M} =\mathbb{E}\int_{0}^{t\wedge\theta^{i,M}_{\Delta,\Delta_1}}\Big(\big|f(\tilde{Y}^{i,M}_{s},\mathcal{L}^{\tilde{Y},M}_{s})-f(Y^{i,M}_{s},\mathcal{L}^{Y,M}_{s})\big|^q+\big|g(\tilde{Y}^{i,M}_{s},\mathcal{L}^{\tilde{Y},M}_{s})-g(Y^{i,M}_{s},\mathcal{L}^{Y,M}_{s})\big|^{q}\Big)\mathrm{d}s.
\end{align}
 Using  the elementary inequality, and \eqref{eqc4.47}, we compute  
{ \begin{align}\label{eqc343}
 \mathbb{E}  \mathbb{W}^{q}_{2}(\mathcal{L}^{X,M}_{s}, \mathcal{L}^{\tilde{Y},M}_{s})  
 \leq  
    \frac{1}{M}\sum_{i=1}^{M}  \mathbb{E} |\Phi^{i,M}_{s}|^{q} 
 & \leq  \frac{1}{M}\sum_{i=1}^{M}  \mathbb{E} \Big|\Phi^{i,M}_{s\wedge \theta^{i,M}_{\Delta,\Delta_1}}\Big|^{q} + \frac{q2^{p-1} \delta C_{T} }{p}+\frac{2(p-q)C_{T} }{p\delta^{\frac{q}{p-q}}[\varphi^{-1}(H\Delta_1^{-\kappa})]^{p}} . 
\end{align}}
  One observes  that  for any $t\in [0, \theta^{i,M}_{\Delta,\Delta_1}]$,
 $|\tilde{Y}^{i,M}_{t}|\vee |Y^{i,M}_{t}|\leq \varphi^{-1}(H\Delta_1^{-\kappa}).$ 
Then it follows from  \eqref{LL2.5}-\eqref{LL2.7} and Corollary \ref{cor2} that
\begin{align}\label{eqc3.44}
\mathcal{J}^{i,M} &\leq C_{\varphi^{-1}(H\Delta_1^{-\kappa})}\mathbb{E}\int_{0}^{t\wedge\theta^{i,M}_{\Delta,\Delta_1}}\left(|\tilde{Y}^{i,M}_{s}-Y^{i,M}_{s}|^q+\mathbb{W}^{q}_{2}(\mathcal{L}^{\tilde{Y},M}_{s},\mathcal{L}^{Y,M}_{s})\right)\mathrm{d}s\nn\
\\&\leq C_{\varphi^{-1}(H\Delta_1^{-\kappa})}\int_{0}^{t}\Big(
\mathbb{E}|\tilde{Y}^{i,M}_{s}-Y^{i,M}_{s}|^{q}+{ \frac{1}{M}\sum_{i=1}^{M} \mathbb{E}|\tilde{Y}^{i,M}_{s}-Y^{i,M}_{s}|^{q}}\Big)\mathrm{d}s\nn\\
&\leq C_T C_{\varphi^{-1}(H\Delta_1^{-\kappa}) }\Delta^{\frac{q}{3}},
\end{align}
where the last second inequality uses  the identical distribution property of $ |\tilde{Y}^{i,M}_{s}-Y^{i,M}_{s}|^q $, $i=1, \cdots,
M$.
Thanks to $X_0\in L^{p}_{0}$, it follows from \eqref{eqc4.47} that
\begin{align}\label{LCY3.43}
\mathbb{E} \big|\Phi^{i,M}_{0 }\big|^{q}&= \mathbb{E}\Big(I_{\{\theta^{i,M}_{\Delta,\Delta_1}=0\}}
\big|\Phi^{i,M}_{0}\big|^{q}\Big)
 \leq \frac{q2^{p-1} \delta C_{T}}{p}+\frac{2(p-q)C_T}{p\delta^{\frac{q}{p-q}}\big(\varphi^{-1}(H\Delta_1^{-\kappa})\big)^{p}}.
\end{align}
Substituting \eqref{eqc343}, \eqref{eqc3.44} and \eqref{LCY3.43} into  \eqref{eqc342}  yields that for each $i$, 
\begin{align}\label{eq++6}
\mathbb{E}\Big|\Phi^{i,M}_{t\wedge\theta^{i,M}_{\Delta,\Delta_1}}\Big|^{q}
&\leq  C \int_{0}^{t}{ \frac{1}{M}\sum_{i=1}^{M}\mathbb{E} \Big|\Phi^{i,M}_{s\wedge\theta^{i,M}_{\Delta,\Delta_1}}\Big|^{q}}\mathrm{d}s+C_T\Big(  \delta + { \delta^{-\frac{q}{p-q}}[\varphi^{-1}(H\Delta_1^{-\kappa})]^{-p}}+C_{\varphi^{-1}(H\Delta_1^{-\kappa}) }\Delta^{\frac{q}{3}}\Big) . \end{align} 
Taking the particle average on both sides of \eqref{eq++6} and  using the Gronwall inequality  arrives at
\begin{align*}
{ \frac{1}{M}\sum_{i=1}^{M} \mathbb{E}\Big|\Phi^{i,M}_{t\wedge\theta^{i,M}_{\Delta,\Delta_1}}\Big|^{q}}
  &\leq    C_T\Big(  \delta + { \delta^{-\frac{q}{p-q}}[\varphi^{-1}(H\Delta_1^{-\kappa})]^{-p}}+C_{\varphi^{-1}(H\Delta_1^{-\kappa}) }\Delta^{\frac{q}{3}}\Big). 
\end{align*}
Inserting the above inequality into \eqref{eq++6} we derive that 
\begin{align*} 
\mathbb{E}\Big|\Phi^{i,M}_{t\wedge\theta^{i,M}_{\Delta,\Delta_1}}\Big|^{q}
&\leq C_T\Big(  \delta + { \delta^{-\frac{q}{p-q}}[\varphi^{-1}(H\Delta_1^{-\kappa})]^{-p}}+C_{\varphi^{-1}(H\Delta_1^{-\kappa}) }\Delta^{\frac{q}{3}}\Big). \end{align*}
Inserting the above inequality into \eqref{eqc4.47}, we obtain  that for any $0\leq t\leq T$,
\begin{align*}
  \mathbb{E}|\Phi^{i,M}_t|^{q}\leq C_{T}\Big(  \delta + { \delta^{-\frac{q}{p-q}}[\varphi^{-1}(H\Delta_1^{-\kappa})]^{-p}}
  +C_{\varphi^{-1}(H\Delta_1^{-\kappa})}\Delta^{\frac{q}{3}}\Big). 
\end{align*}
Using the elementary inequality  and Corollary \ref{cor2} yields that for any $t\in [0, T]$,
\begin{align*}
 \mathbb{E}|X^{i,M}_{t }-Y^{i,M}_{t }|^{q}&\leq 2^{q-1}\Big( \mathbb{E}|\Phi^{i,M}_t|^{q}+ \mathbb{E}|\tilde{Y}^{i,M}_{t_k}-Y^{i,M}_{t}|^{q}\Big) 
 \nn\
 \\&\leq C_{T}\Big(  \delta + { \delta^{-\frac{q}{p-q}}[\varphi^{-1}(H\Delta_1^{-\kappa})]^{-p}}
  +C_{\varphi^{-1}(H\Delta_1^{-\kappa})}\Delta^{\frac{q}{3}}\Big).
  \end{align*}
For any $\varepsilon>0$, choose a   $\delta>0$  such that $C_{T}\delta  <\varepsilon/3$, and then choose a  $\Delta_1\in(0,1]$ sufficiently  small such that
$  {  C_{T}}{ \delta^{ -{q}/{(p-q)}}[\varphi^{-1}(H\Delta_1^{-\kappa})]^{-p}} < {\varepsilon}/{3}.$ 
For the fixed $\Delta_1$, choose a  $\Delta_2\in (0,\Delta_1]$ small enough such that for any $\Delta\in (0,\Delta_2]$, 
$  C_{T}C_{\varphi^{-1}(H\Delta_1^{-\kappa})}\Delta^{ \frac{q}{3}}< {\varepsilon}/{3}.$ 
 As a result,  for any $q\in [2,p_0\wedge p)$, 
 \begin{align}\label{eqf3.26}
\lim\limits_{\Delta\rightarrow0}\sup_{ M\geq 1}\sup_{1\leq i\leq M}\sup_{0\leq t\leq T}\mathbb{E}\big|X^{i,M}_{t}-Y^{i,M}_{t}\big|^{q}=0.
 \end{align}
 Furthermore, if $p_0<p$, utilizing the H\"older inequality  and  Lemma \ref{le2.3} and Theorem \ref{le3.4} yields that \eqref{eqf3.26} holds for  $ p_0\leq q<p$.  Similarly, \eqref{eqf3.26}  holds  for $0< q<2$ directly by virtue of  the H\"older inequality again.  This implies that  the desired result \eqref{eqf3.26} holds for any $q\in (0,p)$. On the other hand, by virtue of  Lemma  \ref{L6.14} and Theorem \ref{le3.4} it follows that for any $q\in (0,p)$ and any $t\in [0, T]$,
  \begin{align*}
  \mathbb{E}|\bar{Y}^{i,M}_{t }-Y^{i,M}_{t }|^{q}=\mathbb{E}|\bar{Y}^{i,M}_{t }-\pi_{\Delta}(\bar{Y}^{i,M}_{t })|^{q}\leq  { C_T}{\big[\varphi^{-1}(H\Delta^{-\kappa})\big]^{-(p-q)}}\rightarrow 0,~~~ \Delta\rightarrow 0.
\end{align*}
Combining the above inequality with \eqref{eqf3.26} implies that the  desired assertion holds.  
\end{proof}  

 By virtue of the propagation of chaos (Lemma  \ref{lec3.4}),  we further derive the convergence of the  TEM numerical solution to the exact solution of   \eqref{eq3}.  
\begin{theorem}\label{thm4.8}
Let Assumptions \ref{ass1}-\ref{ass3}  and $X_0\in L^{p}_{0}$ hold with $p,p_0>2$. Then for any $q\in [0, p)$,
\begin{align*}
\lim_{\substack{\Delta\rightarrow0\\ M\rightarrow\infty}}\sup_{1\leq i\leq M}\sup_{0\leq t\leq T}\Big(\mathbb{E}\big|X^{i}_{t}-Y^{i,M}_{t}\big|^{q}\vee \mathbb{E}\big|X^{i}_{t}-\bar{Y}^{i,M}_{t}\big|^{q}\Big)=0,  ~~\forall T>0.
\end{align*} 
\end{theorem}
 
 \begin{rem}
Theorem \ref{thm4.8} implies that the limit processes $\Delta\rightarrow0$ and $M\rightarrow\infty$ are not required to be taken sequentially. This is because, as established in Theorem \ref{Lemc3.6}, the convergence of the TEM numerical solution to the exact solution for the IPS holds uniformly with respect to $M$. 
\end{rem}  
\section{Strong Convergence Rate}\label{s-cr}

{This section is dedicated to establishing the convergence rate for the TEM scheme. We impose an assumption to describe the  polynomial growth rate of the coefficients, which is slightly stronger than what is required for the conclusion of convergence.}
\begin{assp}\label{ass4}
There exists a pair of positive constants $\alpha$ and $L_{f}$ such that
\begin{align}\label{eql4.1}
\big|f(x_1,\mu_1)-f(x_2,\mu_2)\big|\leq L_{f}\big(|x_1-x_2|(1+|x_1|^{\alpha}+|x_2|^{\alpha})+\mathbb{W}_{2}(\mu_1,\mu_2)\big)
\end{align}
for any $x_1, x_2\in \RR^{d}$ and $\mu_1,\mu_2\in \mathcal{P}_{2}(\RR^{d})$.
\end{assp}
\begin{rem}\label{rm5.1}
{ One observes  that  Assumption \ref{ass1} follows directly from  Assumption \ref{ass4}  with $K_{R}= L_{f}(1+2R^{\alpha})$ for any $R>0$ and $K=L_{f}$.}
Furthermore, by Assumptions \ref{ass2} and \ref{ass4}, one knows that there exists a constant $L_g$ such that 
\begin{align}\label{eql4.2}
\big|g(x_1,\mu_1)-g(x_2,\mu_2)\big|^2&\leq L_g\big(|x_1-x_2|^{2}(1+|x_1|^{\alpha}+|x_2|^{\alpha})+\mathbb{W}^{2}_{2}(\mu_1,\mu_2)\big)\\
  \big|g(x,\mu)\big|^2&\leq L_g \left(1+|x|^{\alpha+2}+\mu(|\cdot|^2)\right).\label{eql4.3}
\end{align}
for any $  x_1,x_2, x\in \RR^{d}$ and $ \mu_1,\mu_2, \mu \in \mathcal{P}_{2}(\RR^{d})$.
\end{rem}
\begin{rem}\label{rem5.2}
Due to \eqref{eql4.1} and \eqref{eql4.2}, choose $\varphi(u)=2(L_f \vee L_g)(1+u^{\alpha})$ for any $u>0$ while the inverse function $\varphi^{-1}(u)=\big(u/2(L_f \vee L_g)-1\big)^{1/\alpha}$ for any $u>2(L_f \vee L_g)$. Thus,  for any given~$\Delta\in(0,1]$, define the truncation mapping  
\begin{align}\label{trun}
\pi_{\Delta}(x)=\displaystyle\left\{\begin{array}{lcl} x,~~&~ |x|\leq \Big(\frac{H\Delta^{-\kappa}}{2(L_f \vee L_g)}-1\Big)^{\frac{1}{\alpha}} ,\\
\displaystyle \Big(\frac{H\Delta^{-\kappa}}{2(L_f \vee L_g)}-1\Big)^{\frac{1}{\alpha}} \frac{x}{|x|},~~&~ |x|> \Big(\frac{H\Delta^{-\kappa}}{2(L_f \vee L_g)}-1\Big)^{\frac{1}{\alpha}} 
\end{array}\right.
\end{align}
for $x\in \RR^d$,  where  $H\geq 1+ [1\vee K\vee L_3\vee |f(0,\boldsymbol{\delta}_0)|\vee|g(0,\boldsymbol{\delta}_0)|^2)\vee (2 (L_f \vee L_g)]$ and $\kappa\in (0, 1/3]$ will be specified in Theorem \ref{th5.5} below. 
\end{rem}

 With the help of the explicit form of $\pi_{\Delta}$, we will  establish   the optimal  convergence rate of $\mathbb{E}|X^{i,M}_{t}-Y^{i,M}_{t}|^q$. Owing to Assumption \ref{ass4} we can obtain  the more precise bound  of $\mathbb{E}\big|\tilde{Y}^{i,M}_{t}-Y^{i,M}_{t}\big|^{q}$ than that in Corollary \ref{cor2}. 

\begin{lemma}\label{lem3.6}
Let Assumptions \ref{ass2}-\ref{ass4} and $X_0\in L^{p}_{0}$ hold with $p\geq \alpha+2$. Then for any  $q\in [2,2p/(\alpha+2)]$, we have
\begin{align*}
\sup_{M\geq 1} \sup_{1\leq i\leq M}\sup_{0\leq t\leq T}\mathbb{E}\big|\tilde{Y}^{i,M}_{t}-Y^{i,M}_{t}\big|^{q}\leq C_{T}\Delta^{\frac{q}{2}},~~~\forall T>0.
\end{align*}
\end{lemma}
\begin{proof}~~Fix $T>0$.
For any $t\in [t_{k},t_{k+1})\cap [0,T]$,   it follows from \eqref{eq4.35},  \eqref{eq2.2} and \eqref{eql4.3} that
\begin{align*}
\mathbb{E}\big|\tilde{Y}^{i,M}_{t}-Y^{i,M}_{t}\big|^{q}&\leq 2^{q-1}\Big(\mathbb{E}\big|f(Y^{i,M}_{t_{k}},\mathcal{L}^{Y,M}_{t_{k}})\big|^{q}\Delta^{q}+\mathbb{E}\big|g(Y^{i,M}_{t_{k}},\mathcal{L}^{Y,M}_{t_{k}})\big|^{q}\Delta^{\frac{q}{2}}\Big).
\nn\\
&\leq C\Delta^{\frac{2q}{3}}\Big({  1+} \mathbb{E}|Y^{i,M}_{t_{k}}|^{q}+\mathbb{E}\Big(\frac{1}{M}\sum_{i=1}^{M}|Y^{i,M}_{t_{k}}|^{2}\Big)^{\frac{q}{2}}\Big)\nn\
\\&~~~+C\Delta^{\frac{q}{2}}\Big(1+\mathbb{E}|Y^{i,M}_{t_{k}}|^{\frac{q(\alpha+2)}{2}}+\mathbb{E}\Big(\frac{1}{M}\sum_{i=1}^{M}|Y^{i,M}_{t_{k}}|^{2}\Big)^{\frac{q}{2}}\Big).\nn\
\end{align*}
Thanks to  $q(\alpha+2)/2\leq p$,  using \eqref{eq*Y}  and  Theorem \ref{le3.4}  implies  that
\begin{align*}
\mathbb{E}|\tilde{Y}^{i,M}_{t}-Y^{i,M}_{t}|^{q}\leq  & C\Delta^{\frac{q}{2}}\Big(1+\mathbb{E}|Y^{i,M}_{t_{k}}|^{q}+
\mathbb{E}|Y^{i,M}_{t_{k}}|^{\frac{q(\alpha+2)}{2}}\Big) 
\nn\\
\leq & C\Delta^{\frac{q}{2}}\Big(1+\big(\mathbb{E}|Y^{i,M}_{t_{k}}|^{p}\big)^{\frac{q}{p}}+\big(\mathbb{E}|Y^{i,M}_{t_{k}}|^{p}\big)^{\frac{q(\alpha+2)}{2p}}\Big)\leq C_T\Delta^{\frac{q}{2}}.
\end{align*}
\end{proof} 

By the similar techniques as in Theorem \ref{Lemc3.6},  we yield   the optimal rate of the TEM numerical solutions approximating the exact solutions of the IPS corresponding to MV-SDE \eqref{eq3.1}. In order for completeness, we give the outline of the proof emphasizing the  difference from Theorem \ref{Lemc3.6}.

\begin{theorem}\label{th5.5}
Let Assumptions \ref{ass2}-\ref{ass4} and $X_0\in L^{p}_{0}$ { hold with $p\geq2(\alpha+2)\vee (4\alpha) $ and $p_0>2$}. Then for any   $q\in [2,p_0)\cap[2,p/((\alpha+2)\vee(2\alpha))]$, the TEM numerical solutions $ {Y}^{i,M}_{t}$ and $ \bar{Y}^{i,M}_{t}$ given by \eqref{eq2.1} with $\kappa\in [q\alpha/2(p-q),1/3]$ satisfy
\begin{align*}
\sup_{  M\geq 1}\sup_{1\leq i\leq M}\sup_{0\leq t\leq T}\Big(\mathbb{E}\big|X^{i,M}_{t}-Y^{i,M}_{t}\big|^{q}\vee
\mathbb{E}\big|X^{i,M}_{t}-\bar{Y}^{i,M}_{t}\big|^{q}\Big)\leq C_T\Delta^{\frac{q}{2}},~~~\forall T>0.
\end{align*}
\end{theorem}
\begin{proof}
  $p\geq2(\alpha+2)\vee (4\alpha) $ implies that $p\geq 3\alpha+2$. For any  
 given $ q\in[2,p_0)\cap [2,p/((\alpha+2)\vee(2\alpha))]$, we know $q\in [2,2p/(3\alpha+2)]$, which implies 
  $q\alpha/2(p-q)\in (0, 1/3]$. Now consider the TEM scheme \eqref{eq2.1} with the explicit form $\pi_{\Delta}$ defined by \eqref{trun} for $\kappa\in [q\alpha/2(p-q),1/3]$. For any $\Delta\in (0,1]$, define $\theta^{i,M}_{\Delta}=\tau^{i,M}_{\varphi^{-1}(H\Delta^{-\kappa})}\wedge\eta^{i,M}_{\Delta,\Delta}$, where $\tau^{i,M}_{\varphi^{-1}(H\Delta^{-\kappa})}$ and $\eta^{i,M}_{\Delta,\Delta}$ are defined by \eqref{eql2.9} and \eqref{eqpl3.43}, respectively.  
   By virtue of the explicit form of $\varphi$ in Remark \ref{rem5.2}, letting $\delta=\Delta^{\frac{q}{2}}$, one derives that 
   \begin{align}\label{eq*+32} \frac{q2^{p-1} \delta C_{T}}{p}+\frac{2(p-q)C_{T}}{p\delta^{\frac{q}{p-q}}[\varphi^{-1}(H\Delta^{-\kappa})]^{p}}\leq C_{T}\Delta^{{\frac{q}{2}}},
   \end{align}
   which together with \eqref{eqc4.47} in Theorem \ref{Lemc3.6} implies that for any $0\leq t\leq  T$, 
\begin{align}\label{4.4}
\mathbb{E}|\Phi^{i,M}_t|^{q} 
 \leq \mathbb{E}\Big|\Phi^{i,M}_{t\wedge \theta^{i,M}_{\Delta }}\Big|^{q}+ C_{T}\Delta^{{\frac{q}{2}}}.
\end{align}
  In \eqref{eq*J},  by \eqref{eql4.1} and \eqref{eql4.2},   utilizing the Young inequality  and the H\"older inequality gives 
\begin{align} \label{eq*+16}
 \mathcal{J}^{i,M}\leq& C\mathbb{E}\int_{0}^{t}\left(\big|\tilde{Y}^{i,M}_{s}-Y^{i,M}_{s}\big|^q\big(1+
 |\tilde{Y}^{i,M}_{s}|^{q\alpha}+|Y^{i,M}_{s}|^{q\alpha}\big)+\mathbb{W}^{q}_{2}
 \big(\mathcal{L}^{\tilde{Y},M}_{s},\mathcal{L}^{Y,M}_{s}\big)\right)\mathrm{d}s\nn\
\\ 
 \leq& C\int_{0}^{t}\Big[\Big(\mathbb{E}|\tilde{Y}^{i,M}_{s}-Y^{i,M}_{s}|^{2q}\Big)^{ {\frac{1}{2 }}}\left(\mathbb{E}\big(1+|\tilde{Y}^{i,M}_{s}|^{q\alpha}+|Y^{i,M}_{s}|^{q\alpha}
 \big)^{2}\right)^{\frac{1}{2}}+{ \frac{1}{M}\sum_{i=1}^{M}\mathbb{E}
 |\tilde{Y}^{i,M}_{s}-Y^{i,M}_{s}|^{q}}\Big]\mathrm{d}s .
\end{align}
Thanks to $2\leq q\leq p/((2\alpha)\vee(\alpha+2))$, utilizing the H\"older inequality,  by virtue of  Lemma \ref{lem3.6} 
we deduce that $ { \sum_{i=1}^{M}\mathbb{E}
 |\tilde{Y}^{i,M}_{s}-Y^{i,M}_{s}|^{q}}\leq C_{T}\Delta^{\frac{q}{2}},  $ and 
 \begin{align} \label{eq++11}
&\Big(\mathbb{E}|\tilde{Y}^{i,M}_{s}-Y^{i,M}_{s}|^{2q}\Big)^{\frac{1}{2}} \leq\Big(\mathbb{E}|\tilde{Y}^{i,M}_{s}-Y^{i,M}_{s}|^{\frac{2p}{\alpha+2}}\Big)^{\frac{q(\alpha+2)}{ 2p}}\leq C_{T}\Delta^{\frac{q}{2}}, \\
&\mathbb{E}\big(1+|\tilde{Y}^{i,M}_{s}|^{q\alpha}+|Y^{i,M}_{s}|^{q\alpha}\big)^{2} \leq 3\Big( 1+\big(\E|\tilde{Y}^{i,M}_{s}|^{p}\big)^{\frac{2q\alpha}{p}}+\big(\E| {Y}^{i,M}_{s}|^{p}\big)^{\frac{2q\alpha}{p}}\Big) \leq C_{T}.  \label{eq*+37} 
\end{align}
Consequently, we have $
\mathcal{J}^{i,M} \leq C_{T}\Delta^{\frac{q}{2}}. $
Inserting  this inequality, \eqref{eqc343} and \eqref{LCY3.43}  into \eqref{eqc342}, replacing $\theta^{i,M}_{\Delta,\Delta_1}$ by $\theta^{i,M}_{\Delta}$ and using \eqref{eq*+32} arrives at
\begin{align}\label{eq++7}
\mathbb{E}\Big|\Phi^{i,M}_{t\wedge\theta^{i,M}_{\Delta}}\Big|^{q}
\leq  C\int_{0}^{t}{ \frac{1}{M}\sum_{i=1}^{M}\mathbb{E} \Big|\Phi^{i,M}_{s\wedge\theta^{i,M}_{\Delta}}\Big|^{q}}\mathrm{d}s+C_T \Delta^{\frac{q}{2}} .
\end{align} 
 Taking the particle average on both sides of \eqref{eq++7} and then using the Gronwall inequality yields 
 \begin{align*}
{ \frac{1}{M}\sum_{i=1}^{M}\mathbb{E}\Big|\Phi^{i,M}_{t\wedge\theta^{i,M}_{\Delta}}}\Big|^{q}
\leq  C\int_{0}^{t}{ \frac{1}{M}\sum_{i=1}^{M}\mathbb{E} \Big|\Phi^{i,M}_{s\wedge\theta^{i,M}_{\Delta}}\Big|^{q}}\mathrm{d}s+C_T \Delta^{\frac{q}{2}} \leq C_T \Delta^{\frac{q}{2}} . 
\end{align*}  Inserting the above inequality into \eqref{eq++7} derives $\mathbb{E}\Big|\Phi^{i,M}_{t\wedge\theta^{i,M}_{\Delta}}\Big|^{q}
\leq C_T \Delta^{\frac{q}{2}} $. 
Combining this inequality with \eqref{4.4} yields $
\mathbb{E}|\Phi^{i,M}_t|^{q}\leq C_T\Delta^{\frac{q}{2}}.$   This together with  Lemma \ref{lem3.6} implies
\begin{align}\label{4.10*}
\sup_{  M\geq 1}\sup_{1\leq i\leq M}\sup_{0\leq t\leq T}\mathbb{E}\big|X^{i,M}_{t}-Y^{i,M}_{t}\big|^{q}\leq C_T\Delta^{\frac{q}{2}},~~~\forall ~T>0.
\end{align}
Recalling the definition   $\varphi^{-1}$ in Remark \ref{rem5.2}, by virtue of Lemma \ref{L6.14}, we obtain
\begin{align}\label{eqf6.40}
\sup_{  M\geq 1}\sup_{1\leq i\leq M}\sup_{0\leq t\leq T}\mathbb{E}\big|\bar{Y}^{i,M}_{t}-Y^{i,M}_{t}\big|^{q}
 & \leq C\Delta^{\frac{\kappa (p-q)}{\alpha }}\leq C\Delta^{\frac{q}{2}},
\end{align} 
where the last inequality uses the fact $2\leq q\leq 2p/(3\alpha+2)$. Combining \eqref{4.10*} and \eqref{eqf6.40} yields the desired assertion. The proof is complete.
\end{proof}

 By virtue of the propagation of chaos in   Lemma  \ref{lec3.4},   we obtain the convergence rate of the TEM numerical solutions to the exact solutions of the {N-IPS}.
\begin{theorem}[Numerical propagation of chaos]\label{Th5.5}
Let Assumptions \ref{ass2}-\ref{ass4} and $X_0\in L^{p}_{0}$ hold with $p\geq2(\alpha+2)\vee (4\alpha)$ and  $p_0>2$. Then for any   $q\in [2,p_0)\cap[2,p/((\alpha+2)\vee(2\alpha))]$,  the numerical solutions $ {Y}^{i,M}_{t}$ and $ \bar{Y}^{i,M}_{t}$ given by \eqref{eq2.1} with $\kappa\in [q\alpha/2(p-q),1/3]$ satisfy
$$
\sup_{1\leq i\leq M}\sup_{0\leq t\leq T}\Big(\mathbb{E}|X^{i}_{t}-Y^{i,M}_{t}|^{q}\vee \mathbb{E}|X^{i}_{t}-\bar{Y}^{i,M}_{t}|^{q}\Big)\leq  C_{q, p, d,T} {\Upsilon_{M,q,p,d}+C_{T}\Delta^{\frac{q}{2}},~~~\forall T>0,}
$$
where  $\Upsilon_{M,q,p,d}$ is given by Lemma \ref{Lem3.1}.
\end{theorem}
\section{Approximation of Invariant Measure}\label{sec55}
This section investigates the approximation of the invariant probability measure of the MV-SDE \eqref{eq3.1} by the TEM scheme \eqref{eq2.1}. Subsection \ref{sb6.1} derives a non-asymptotic error bound between the distribution of the one-particle numerical solution and the marginal distribution of the IPS's invariant probability measure, under the assumption that the latter is exponentially ergodic. Subsection 5.2 is devoted to establishing the exponential ergodicity of the numerical IPS, and Subsection 5.3 proves a uniform-in-time convergence rate of the one-particle numerical solution towards the exact MV-SDE solution. Combining these results, we further obtain an asymptotic error bound between the one-particle marginal of the IPS’s numerical invariant measure and the exact invariant measure of the MV-SDE.

To characterize the exponential ergodicity of the exact and numerical solutions, we first introduce the underlying semigroups. For  the exact solution of MV-SDE \eqref{eq3.1}, we define a
nonlinear semigroup $\mathrm{P}^{*}_{t}: \mathcal{P}(\RR^{d})\rightarrow \mathcal{P}(\RR^{d})$  as
$\mathrm{P}^{*}_{t}\mu_0=\mathcal{L}^{X}_{t}$ with  $\mathcal{L}^{X}_{0}=\mu_0\in \mathcal{P}(\RR^{d})$. Following \cite[p.598, (1.10)]{WFY2018}, the family of operators $\{\mathrm{P}^{*}_{t}\}_{t\geq 0}$ satisfies  the semigroup property
$\mathrm{P}^*_{s+t}=\mathrm{P}^*_{t}\mathrm{P}^*_{s}$ for all~$s,t\geq0$. To prepare for the subsequent analysis,  we consider the IPS \eqref{eq2} on   $(\RR^{d})^M$ for any fixed $M\geq 1$.
Denote its exact solution  by $X^{M}_{t}=(X^{1,M}_{t},X^{2,M}_{t},\cdots,X^{M,M}_t)\in (\RR^{d})^{M},~t\geq 0$. For any $\mu_0^M\in \mathcal{P}((\RR^{d})^M)$, let $X_0^M$ obey the law $ \mu_0^M$ ($X_0^M\sim\mu_0^M$)
and define the operator $\mathbf{P}^{M,*}_{t}: \mathcal{P}((\RR^{d})^M)\rightarrow \mathcal{P}((\RR^{d})^M)$ by 
$$\mathbf{P}^{M,*}_{t}\mu_0^M=\mathcal{L}^{X^M}_{t},~~t\geq 0.$$ 
Since $X_t^M$ is  an $(\RR^{d})^{M}$-valued Markovian  process,
the family  $\{\mathbf{P}^{M,*}_{t}\}_{t\geq 0}$ satisfies the semigroup property  $\mathbf{P}^{M,*}_{t+s}=\mathbf{P}^{M,*}_{t}\mathbf{P}^{M,*}_{s}$ for any $s,t \geq 0$.
Let $\bar{Y}^{M}_{t_k}=(\bar{Y}^{1,M}_{t_{k}},\bar{Y}^{2,M}_{t_k},\cdots, \bar{Y}^{M,M}_{t_k})$ denote the numerical solution of the IPS given by the TEM scheme \eqref{eq2.1}. For any $\mu_0^M\in \mathcal{P}((\RR^{d})^{M})$, let $Y^{M}_{0}\sim \mu_0^M$ and define $\mathbf{P}^{\Delta,M,*}_{t_{k}}:~ \mathcal{P}\big((\RR^{d})^{M}\big)\rightarrow \mathcal{P}\big((\RR^{d})^{M}\big)$ by 
$
\mathbf{P}^{\Delta,M,*}_{t_{k}} \mu_0^{ M} = \mathcal{L}^{\bar{Y}^{M}}_{t_k}, ~~k\geq 0.
$
Next we show that the family $\{\mathbf{P}^{\Delta,M,*}_{t_{k}}\}_{k\geq 0}$ forms a discrete-time semigroup. In fact, 
let $\bar{U}^{M}_0\sim  \mathcal{L}^{\bar{Y}^{M}}_{t_k}$  and  $\{\Delta \bar{B}^{i}_{t_1}\}_{i=1}^{M}$ be independent Brownian increments which are independent of $\bar{U}^{M}_0$. Then the one-step TEM iteration starting from $\bar{U}^{M}_{0} $ is given by 
 \begin{equation*}
 \begin{cases} 
& U^{i,M}_{ 0}=\pi_{\Delta}(\bar{U}^{i,M}_{ 0}) ,~~~~i=1,\cdots, M,\\
&\bar{U}^{i,M}_{t_1}=U^{i,M}_{0}+f(U^{i,M}_{0},\mathcal{L}^{U,M}_{0})\Delta+g(U^{i,M}_{0},\mathcal{L}^{U,M}_{0})\Delta \bar{B}^{i}_{t_1}.
 \end{cases}
 \end{equation*}
Since $\Delta B^{i}_{t_k}$ and $\Delta \bar{B}^{i}_{t_1}$ are~identically distributed, we obtain that $\mathcal{L}^{\bar{U}^{M}}_{t_1}=\mathcal{L}^{\bar{Y}^{M}}_{t_{k+1}},$
 which implies   
\begin{align*}
\mathbf{P}^{\Delta,M,*}_{t_{k+1}}\mu_0^{M}&= \mathcal{L}^{\bar{Y}^{M}}_{t_{k+1}}= \mathcal{L}^{\bar{U}^{M}}_{t_{1}} =\mathbf{P}^{\Delta,M,*}_{t_{1}}\mathbf{P}^{\Delta,M,*}_{t_{k}}\mu_0^{M},~~k\geq 0.
\end{align*}
Thus, $\{\mathbf{P}^{\Delta,M,*}_{t_{k}}\}_{k\geq 0}$ has the semigroup property
$
\mathbf{P}^{\Delta,M,*}_{t_{k+l}}\mu_0^M=\mathbf{P}^{\Delta,M,*}_{t_{k}}
\mathbf{P}^{\Delta,M,*}_{t_{l}}\mu_0^M,~ \forall~ k,l\geq 0.
$
Let $ \mathcal{P}^*((\RR^{d})^{M})$ and $ \mathcal{P}^*_q((\RR^{d})^{M})$ denote the family of  all
   exchangeable   $\mu_0^{M}\in \mathcal{P}((\RR^{d})^{M})$  and $\mu_0^{M}\in \mathcal{P}_q((\RR^{d})^{M})$ with respect to the  indices $\{1,2,\cdots, M\}$ for some $q>0$, respectively. For any ~$\nu^{M}\in \mathcal{P}((\RR^{d})^{M})$, define the marginal distribution  with respect to the $i$th component ($1\leq i \leq M$) 
\begin{align*}
\Gamma_i(\nu^{M})(A)=\nu^{M}(  {{\RR^{d}\times  \RR^{d}\times \cdots  { {A}}  \cdots\times \RR^{d}\times  \RR^{d}}} ),  ~~\forall~A \in \mathcal{B}(\RR^{d}) .\end{align*}
Since both the IPS and the  TEM scheme are symmetric for the particles   and the   Brownian motions  are i.i.d., it follows that for any  $  \mu_0^{M}\in \mathcal{P}^*((\RR^{d})^{M})$, 
 $\mathrm{P}^{M,*}_{t}\mu_0^M\in \mathcal{P}^*((\RR^{d})^{M}),~ t\geq 0 $ and $ \mathbf{P}^{\Delta,M,*}_{t_k}\mu_0^M \in \mathcal{P}^*((\RR^{d})^{M}),~k\geq 0.  $ 
This invariance implies $$\Gamma_1(\mathrm{P}^{M,*}_{t}\mu_0^M)=\cdots=\Gamma_M(\mathrm{P}^{M,*}_{t}\mu_0^M),~~~~\Gamma_1(\mathbf{P}^{\Delta,M,*}_{t_k}\mu_0^M) 
=\cdots=\Gamma_M(\mathbf{P}^{\Delta,M,*}_{t_k}\mu_0^M).$$ 
In particular, the above identities hold for  $\mu_0^M= \mu^{\otimes M}_0  $, where $\mu^{\otimes M}_0 $ denotes the $M$-independent product of  $\mu_0\in \mathcal{P}( \RR^{d} )$. Define an operator $\mathrm{P}^{\Delta,M,*}_{t_{k}}: ~ \mathcal{P}(\RR^{d}) \rightarrow \mathcal{P}(\RR^{d}) $ by 
 \begin{align}\label{eq++4}
  \mathrm{P}^{\Delta,M,*}_{t_{k}}(\mu_0)=\Gamma_1(\mathbf{P}^{\Delta,M,*}_{t_k}\mu^{\otimes M}_{0}),~~~~ \forall \mu_0\in \mathcal{P}( \RR^{d} ). \end{align}

To proceed with our analysis, in this section, we require $\kappa\in (0,1/3)$ in the definition of the  truncation  mapping given by \eqref{eqq4.1}, and set \begin{equation}\label{eq*+25}\theta=\frac{1}{3}-\kappa > 0.\end{equation} 

\subsection{The Non-asymptotic  Error Estimate }\label{sb6.1}

In this subsection, we derive a non-asymptotic error bound in the Wasserstein distance between the one-particle marginal distribution of the TEM numerical solution and the corresponding marginal of the IPS invariant probability measure. Our analysis combines the finite-time convergence of the TEM scheme with the exponential ergodicity of the one-particle marginal of the IPS. We begin by introducing an assumption on the coefficients of the MV-SDE \eqref{eq3.1}.

\begin{assp}\label{ass5}
There exist  a  constant~$p\geq 2$ and a pair of constants  $\lambda_1>\lambda_2\geq0$ such that
\begin{align*}
 2x^{T}f(x,\mu)+(p-1)|g(x,\mu)|^{2}\leq-\lambda_1|x|^{2}+\lambda_2\mu(|\cdot|^2)+C
\end{align*}
for any $x\in \RR^{d}$ and $\mu\in \mathcal{P}_{2}(\RR^{d})$.
\end{assp}

Under Assumption \ref{ass5},
using the similar techniques as in   \cite[Theorem 3.1]{WFY2018},  we can  obtain the uniform-in-time  moment boundedness of the exact solutions to MV-SDE \eqref{eq3.1} and to the IPS \eqref{eq2}, respectively. We omit the proof to avoid duplication.
\begin{lemma}\label{Le5.1}
Let Assumptions  \ref{ass1}, \ref{ass2}, \ref{ass5}  and $X_0\in {  L^{p}_{0}}$ hold. Then 
$$
\sup_{t\geq 0}\mathbb{E}|X_{t}|^{p}\leq C,~~~~
\sup_{M\geq 1}\sup_{1\leq i\leq M}\sup_{t\geq 0}\mathbb{E}|X^{i,M}_{t}|^{p}\leq C.
$$ 
\end{lemma}

 Now we   establish the uniform-in-time moment boundedness of the numerical solution generated  by the TEM scheme \eqref{eq2.1}.  
\begin{theorem}\label{L7.4}
Let Assumptions  \ref{ass1}, \ref{ass2}, \ref{ass5} and $X_0\in { L^{p}_{0}}$ hold. Then there exists a  $\Delta_1^*\in (0,1]$ such that   the TEM numerical solution given by  \eqref{eq2.1} satisfies  
 \begin{align*}
\sup_{M\geq 1} \sup_{1\leq i\leq M}\sup_{\Delta\in(0,\Delta_1^*]}\sup_{ t\geq0}  \Big(\mathbb{E}|Y^{i,M}_{t }|^{p}\vee \mathbb{E}|\bar{Y}^{i,M}_{t }|^{p}\Big)\leq C.
 \end{align*}
 \end{theorem} 
\begin{proof} 
For convenience,  define  $\displaystyle\varrho=\frac{p(\lambda_1-\lambda_2)}{8}  >0$. For any $\Delta\in (0,1]$, using the It\^o formula, we derive from \eqref{eq4.35} and Assumption \ref{ass5} that for any $t\in[t_{k},t_{k+1})$,
\begin{align}\label{eqqc6.6}
 \mathbb{E}\Big(e^{\varrho t }|\tilde{Y}^{i,M}_{t}|^{p}\Big) 
&\leq \mathbb{E}\Big(e^{\varrho t_k }|Y^{i,M}_{t_{k}}|^{p}\Big)
+\varrho\mathbb{E}\int_{t_{k}}^{t}e^{\varrho s }
|\tilde{Y}^{i,M}_{s}|^{p}\mathrm{d}s\nn\\
&~~~+\frac{p}{2}\mathbb{E}\int_{t_{k}}^{t}e^{\varrho s}|\tilde{Y}^{i,M}_{s}|^{p-2}
\Big[2(\tilde{Y}^{i,M}_{s})^{T}f(Y^{i,M}_{t_{k}},\mathcal{L}^{Y,M}_{t_{k}})+(p-1)\big|g(Y^{i,M}_{t_{k}},\mathcal{L}^{Y,M}_{t_{k}})
\big|^{2}\Big]\mathrm{d}s\nn\\
&\leq e^{\varrho t_k}\mathbb{E}|Y^{i,M}_{t_{k}}|^{p}
+\varrho\int_{t_{k}}^{t}e^{\varrho s}
\mathbb{E}|\tilde{Y}^{i,M}_{s}|^{p}\mathrm{d}s\nn\\
&~~~+\frac{p }{2}\mathbb{E}\int_{t_{k}}^{t}e^{\varrho s}|\tilde{Y}^{i,M}_{s}|^{p-2}
\big(-\lambda_1| Y^{i,M}_{t_{k}}|^2 +\lambda_2\mathcal{L}^{Y,M}_{t_{k}}(|\cdot|^2)+C\big)\mathrm{d}s\nn\
\\&~~~+p\mathbb{E}\int_{t_{k}}^{t}e^{\varrho s}|\tilde{Y}^{i,M}_{s}|^{p-2}|
\tilde{Y}^{i,M}_{s}-Y^{i,M}_{t_{k}}| |f(Y^{i,M}_{t_{k}},\mathcal{L}^{Y,M}_{t_{k}}) |\mathrm{d}s. 
\end{align}
One notices from the elementary inequality that 
for any $\delta>0$ (defined later) 
$$
| Y^{i,M}_{t_{k}}|^2 \geq (1-\delta)
|\tilde{Y}^{i,M}_{s}|^2+\Big(1-\frac{1}{\delta}\Big)
|\tilde{Y}^{i,M}_{s}-Y^{i,M}_{t_{k}}|^2\geq (1-\delta)
|\tilde{Y}^{i,M}_{s}|^2 -\frac{1}{\delta} 
|\tilde{Y}^{i,M}_{s}-Y^{i,M}_{t_{k}}|^2. 
$$
Using the Young inequality gives that
\begin{align*}
\frac{p C}{2}\int_{t_{k}}^{t}e^{\varrho s}\mathbb{E}|\tilde{Y}^{i,M}_{s}|^{p-2}\mathrm{d}s\leq \frac{\varrho}{2}\int_{t_{k}}^{t}e^{\varrho s}\mathbb{E}|\tilde{Y}^{i,M}_{s}|^{p}\mathrm{d}s+C\Delta.
\end{align*}
Inserting the above two inequalities into \eqref{eqqc6.6}  and dividing by $e^{\varrho t}$ yield 
\begin{align}\label{eq*2}
\mathbb{E}|\tilde{Y}^{i,M}_{t}|^{p}&\leq e^{\varrho(t_k-t)}\mathbb{E}|Y^{i,M}_{t_{k}}|^{p}
-\Big(\frac{p }{2} (1- \delta)\lambda_1-\frac{3\varrho}{2}\Big)\int_{t_{k}}^{t}
e^{\varrho  (s-t)}
\mathbb{E}|\tilde{Y}^{i,M}_{s}|^{p}\mathrm{d}s+C\Delta\nn\\
&~~+\frac{p  \lambda_1}{2\delta} \int_{t_{k}}^{t}e^{\varrho  (s-t)}\mathbb{E}\Big( |\tilde{Y}^{i,M}_{s}|^{p-2}
|\tilde{Y}^{i,M}_{s}-Y^{i,M}_{t_{k}}|^2\Big) \mathrm{d}s\nn\\
&~~ +\frac{p  \lambda_2}{2 }\int_{t_{k}}^{t} e^{\varrho  (s-t)} \mathbb{E}\Big( |\tilde{Y}^{i,M}_{s}|^{p-2}\mathcal{L}^{Y,M}_{t_{k}}(|\cdot|^2) \Big)\mathrm{d}s \nn\\
&~~+p \int_{t_{k}}^{t}e^{\varrho  (s-t)} \mathbb{E}\Big(|\tilde{Y}^{i,M}_{s}|^{p-2}|
\tilde{Y}^{i,M}_{s}-Y^{i,M}_{t_{k}}| |f(Y^{i,M}_{t_{k}},\mathcal{L}^{Y,M}_{t_{k}})
 |\Big)\mathrm{d}s.\end{align}
Now we estimate the terms on the right side of the above inequality.
Using the H$\ddot{\hbox{o}}$lder inequality, \eqref{eq3.32} and  the Young inequality we obtain that
\begin{align}\label{eq*3}
\mathbb{E}\Big( |\tilde{Y}^{i,M}_{s}|^{p-2}
|\tilde{Y}^{i,M}_{s}-Y^{i,M}_{t_{k}}|^2\Big)\leq & \Big(\mathbb{E}|\tilde{Y}^{i,M}_{s}|^p\Big)^{\frac{p-2}{p}}
\Big(\mathbb{E}|\tilde{Y}^{i,M}_{s}-Y^{i,M}_{t_{k}}|^p \Big)^{\frac{2}{p} }\nn\\
\leq & C \Delta^{\frac{2}{3}}\Big(\mathbb{E}
|\tilde{Y}^{i,M}_{s}|^p\Big)^{\frac{p-2}{p}}\Big(1+ \mathbb{E}|Y^{i,M}_{t_{k}}|^p\Big)^{\frac{2}{p}}\nn\\
 \leq &C \Delta^{\frac{2}{3}}\Big(1 +\mathbb{E}|\tilde{Y}^{i,M}_{s}|^p+\mathbb{E}|Y^{i,M}_{t_{k}}|^p \Big),
\end{align} where $C$ is a positive constant only depending on $p$.
Using the Young inequality we have
\begin{align}\label{eq*4}
 \mathbb{E}\Big( |\tilde{Y}^{i,M}_{s}|^{p-2}\mathcal{L}^{Y,M}_{t_{k}}(|\cdot|^2) \Big)
 \leq  \frac{ p-2  }{p} \mathbb{E} |\tilde{Y}^{i,M}_{s}|^{p} 
+   \frac{  2  }{p} \mathbb{E}\Big(\mathcal{L}^{Y,M}_{t_{k}}(|\cdot|^2) \Big)^{\frac{p}{2}}.\end{align}
Owing to the convexity of function $y=x^q$ for $x>0$ we know that for any $\delta>0$,
 \begin{align}\label{5.8+}
 (a+b)^{q}\leq (1+\delta)a^{q}+C_{\delta}b^{q},~~\forall a, b>0.
 \end{align} 
 Using  the H\"older inequality, \eqref{eq*Y} and \eqref{5.8+} yields that for any $s\in [t_{k},t)$ and $\delta>0$,
 \begin{align*}
\mathbb{E}\Big(\mathcal{L}^{Y,M}_{t_{k}}(|\cdot|^2) \Big)^{\frac{p}{2}}&\leq \mathbb{E}|Y^{i,M}_{t_{k}}|^{p}
\leq (1+\delta)\mathbb{E}|\tilde{Y}^{i,M}_{s}|^{p}+C_{\delta}\mathbb{E}|\tilde{Y}^{i,M}_{s}-Y^{i,M}_{t_{k}}|^{p}\nn\
\\&\leq (1+\delta)\mathbb{E}|\tilde{Y}^{i,M}_{s}|^{p}+C_{\delta}\Delta^{\frac{p}{3}}\big(1+\mathbb{E}|Y^{i,M}_{t_{k}}|^{p}\big).
\end{align*}
Inserting the above inequality into \eqref{eq*4} gives that
\begin{align*}
\mathbb{E}\Big( |\tilde{Y}^{i,M}_{s}|^{p-2}\mathcal{L}^{Y,M}_{t_{k}}(|\cdot|^2) \Big)
 \leq (1+\delta)\mathbb{E}|\tilde{Y}^{i,M}_{s}|^{p}+C_{\delta}\Delta^{\frac{p}{3}}\big(1+\mathbb{E}|Y^{i,M}_{t_{k}}|^{p}\big).
\end{align*}
Furthermore, using the Young inequality, \eqref{eq3.35} and  \eqref{eq*+25} we derive that
\begin{align}\label{eq*5}
 &\mathbb{E}\Big(|\tilde{Y}^{i,M}_{s}|^{p-2}|
\tilde{Y}^{i,M}_{s}-Y^{i,M}_{t_{k}}| |f(Y^{i,M}_{t_{k}},\mathcal{L}^{Y,M}_{t_{k}})
 |\Big)\nn\\
 \leq & 
 \frac{ \lambda_1-\lambda_2}{8}\mathbb{E} |\tilde{Y}^{i,M}_{s}|^{p} 
+  \Big(\frac{8}{\lambda_1-\lambda_2}\Big)^{\frac{p-2}{2}}\mathbb{E}\Big(|\tilde{Y}^{i,M}_{s}-Y^{i,M}_{t_{k}}|^{\frac{p}{2}} |f(Y^{i,M}_{t_{k}},\mathcal{L}^{Y,M}_{t_{k}}) | ^{\frac{p}{2}} \Big)\nn\\
\leq  &
  \frac{ \lambda_1-\lambda_2}{8}\mathbb{E} |\tilde{Y}^{i,M}_{s}|^{p} 
+  C  \Delta^{\frac{3p\theta}{4}  }\Big(1+\mathbb{E}|Y^{i,M}_{t_{k}}|^p \Big). 
\end{align}
Let $\delta=\frac{\lambda_1-\lambda_2}{4(\lambda_1+\lambda_2)}>0$. Inserting inequalities \eqref{eq*3}-\eqref{eq*5} into \eqref{eq*2}  yields that 
 \begin{align*} 
\mathbb{E}|\tilde{Y}^{i,M}_{t}|^{p}&\leq e^{\varrho  (t_k-t)}\mathbb{E}|Y^{i,M}_{t_{k}}|^{p}
-  \Big(\frac{\varrho}{2}-C\Delta^{\frac{2}{3}}\Big)\int_{t_{k}}^{t}
e^{\varrho  (s-t)}
\mathbb{E}|\tilde{Y}^{i,M}_{s}|^{p}\mathrm{d}s\nn\\
&~~+C \int_{t_{k}}^{t}e^{\varrho  (s-t)}\Big(1+\Delta^{\frac{2}{3}\wedge\frac{3p\theta}{4}  }\mathbb{E} 
| Y^{i,M}_{t_{k}}|^p\Big) \mathrm{d}s+C\Delta.\end{align*}
Choose $\Delta_1^*\in (0,1]$ small sufficiently such that   
\begin{align}\label{eq*6}
 C(\Delta_1^*)^{\frac{2}{3}}\leq \frac{\varrho}{2},~~~~  \frac{\varrho^2}{2} \Delta_1^* + C (\Delta_1^*)^{\frac{2}{3}\wedge \frac{3p\theta}{4}}\leq \frac{\varrho}{2 },~~~~\frac{\varrho}{2}\Delta^*< 1.
 \end{align}Thus, for any $\Delta\in (0, \Delta_1^*]$, $t\in [t_k,t_{k+1})$,
\begin{align*} 
\mathbb{E}|\tilde{Y}^{i,M}_{t}|^{p} &\leq   \Big( e^{-\varrho  ( t-t_k)}+C\Delta^{1+\frac{2}{3}\wedge\frac{3p\theta}{4} }\Big)\mathbb{E}|Y^{i,M}_{t_{k}}|^{p}
 +C \Delta.
 \end{align*}
Taking the limit inferior on both sides of the above inequality, by  the Fatou Lemma, we arrive at 
\begin{align*}
\mathbb{E}|\bar{Y}^{i,M}_{t_{k+1}}|^{ {p}}
  \leq \liminf_{t\rightarrow t^{-}_{k+1}}\mathbb{E}|\tilde{Y}^{i,M}_{t}|^{p}&   \leq \Big( e^{-\varrho  \Delta}+C\Delta^{1+\frac{2}{3}\wedge\frac{3p\theta}{4} }\Big)\mathbb{E}|Y^{i,M}_{t_{k}}|^{p}
 +C \Delta.
\end{align*}
Then using the Taylor expansion, \eqref{eq*6}  and the fact   $|Y^{i,M}_{t_{k+1}}|\leq |\bar{Y}^{i,M}_{t_{k+1}}|$  we obtain  
\begin{align*}  
\mathbb{E}|\bar{Y}^{i,M}_{t_{k+1}}|^{ {p}} &\leq    \Big(1- \varrho  \Delta+ \frac{\varrho^2}{ 2}\Delta^2+C\Delta^{1+\frac{2}{3}\wedge \frac{3p\theta}{4} } \Big)\mathbb{E}| {Y}^{i,M}_{t_{k}}|^{p}
 +C \Delta\nn\\
 & \leq   \Big(1- \frac{\varrho}{2}\Delta \Big)\mathbb{E}|\bar{Y}^{i,M}_{t_{k}}|^{p}
 +C \Delta.\end{align*}
Solving the above difference inequality and using the fact   $|Y^{i,M}_{t_{k+1}}|\leq |\bar{Y}^{i,M}_{t_{k+1}}|$  yields   
 \begin{align*} 
\mathbb{E}|Y^{i,M}_{t_{k+1}}|^{p}\leq\mathbb{E}|\bar{Y}^{i,M}_{t_{k+1}}|^{ {p}}  \leq \Big(1- \frac{\varrho}{2}\Delta \Big)^{k+1}\mathbb{E}|X^{i}_{0}|^{p} +C \Delta\sum_{n=0}^{k}\Big(1- \frac{\varrho}{2}\Delta \Big)^n 
\leq C,
\end{align*}
where the positive constant $C$ is independent of  $k$, $M$ and $\Delta$.
Therefore, the desired assertions follow from the arbitrariness of $M$ and \eqref{3.9}.
The proof is complete.
\end{proof}

Owing to Theorem \ref{L7.4}, following the proof line of Lemma \ref{lem3.6}, we can directly derive the uniform error of $\mathbb{E}\big|\tilde{Y}^{i,M}_{t}-Y^{i,M}_{t}\big|^{q}$ with respect to $t\geq 0$.
\begin{cor}\label{cor+1}
Let Assumptions \ref{ass2}, \ref{ass4}, \ref{ass5}  hold and $X_0\in L^{p}_{0}$ with $p\geq \alpha+2$. Then for any  $q\in [2,2p/(\alpha+2)]$, we have
\begin{align*}
\sup_{M\geq 1} \sup_{1\leq i\leq M}\sup_{t\geq 0}\mathbb{E}\big|\tilde{Y}^{i,M}_{t}-Y^{i,M}_{t}\big|^{q}\leq C \Delta^{\frac{q}{2}},
\end{align*} where the positive constant $C$ is independent of $\Delta$.
\end{cor}

In order   for the non-asymptotic error we cite a useful lemma from \cite[Lemma 3.18]{IMA} and impose a hypothesis on the ergodicity of the IPS. 
{\begin{lemma}[{{\cite[Lemma 3.18]{IMA}}}]\label{L6.4}
Given any integer $m\geq 1$, $\varepsilon>0$ and $0<\delta<1$. If a non-negative sequence $\{\alpha_{n}\}_{n\geq 0}$ satisfies 
$
\alpha_{n}\leq \varepsilon+\delta \alpha_{n-m},~ \forall n\geq m,
$
then 
\begin{align*}
\alpha_{n}\leq \frac{\varepsilon}{1-\delta}+K\delta^{\frac{n}{m}-1},~~~\forall n\geq 0,
\end{align*}
where $K=\sup_{0\leq k\leq m-1}\alpha_{k}$.
\end{lemma}}
\begin{assp}\label{ass7}
 For any fixed $M\geq1$, the IPS \eqref{eq2} admits a unique  invariant probability measure $\mu^{M,*}\in\mathcal{P}^*_{\bar{q}}((\RR^d)^M)$ for some $\bar{q}\geq 2$. Moreover, there exist constants $\gamma_{\bar{q}}>0$ and $K_{\bar{q}}\geq1$ such that for any $M\geq 1$ and  $\mu_0^M\in\mathcal{P}_{\bar{q}}^*((\RR^d)^M)$ 
\begin{align*}
\mathbb{W}_{\bar{q}}\big(\Gamma_1(\mathbf{P}^{M,*}_{t}\mu_0^{M}),\Gamma_1(\mu^{M,*})\big)\leq K_{\bar{q}}e^{-\gamma_{\bar{q}} t}\mathbb{W}_{\bar{q}}\big(\Gamma_1(\mu_0^M),\Gamma_1(\mu^{M,*})\big),~~~~t\geq 0.
\end{align*}
\end{assp}

 By combining the finite-time convergence rate of the TEM numerical solution to the IPS established in Theorem \ref{th5.5}, together with Assumption~\ref{ass7} and Lemma~\ref{L6.4}, we obtain the following non-asymptotic error estimate.

\begin{theorem}
Let Assumptions \ref{ass2} and \ref{ass4}-\ref{ass7} hold with  $p\geq \bar{q}(\alpha+2 \vee \alpha)$  and $p_0>\bar{q}$. 
Then for any $\mu _0\in \mathcal{P}_{p}(\RR^{d})$ and $\Delta\in (0,\Delta_1^*]$, the TEM numerical solution  with $\kappa\in [\bar{q}\alpha/2(p-\bar{q}),1/3)$  satisfies
\begin{align*}
\mathbb{W}_{\bar{q}}\big( \mathrm{{P}}^{\Delta,M,*}_{t_{k}} \mu_{0} ,\Gamma_1(\mu^{M,*})\big)&\leq {C}\Delta^{\frac{1}{2}}+  {C} e^{-\lambda k\Delta}.
\end{align*}
where  $\Delta_1^*$ is given in Theorem \ref{L7.4},  $ {C}>0$ is  independent of $k$, $M$ and $\Delta$, and 
$\lambda =\gamma_{\bar{q}}/(1+\gamma_{\bar{q}}+\log K_{\bar{q}}).$
\end{theorem}
\begin{proof} Fix a $\kappa\in [\bar{q}\alpha/2(p-\bar{q}),1/3)$.
By virtue of Theorem \ref{L7.4},   there is a $\Delta_1^*\in (0,1]$ such that for any  $\mu_{0}\in \mathcal{P}_{p}(\RR^{d})$, $\Delta\in (0,\Delta_1^*]$,
\begin{align*}
 \mathbf{P}^{\Delta,M,*}_{t_k}\mu^{\otimes M}_0 (|\cdot|^{p})\leq  \mathbb{E}\Big(\sum_{i=1}^{M}|\bar{Y}^{i,M}_{t_{k}}|^{2}\Big)^{\frac{p}{2}}\leq  M^{\frac{p}{2}} \mathbb{E}|\bar{Y}^{i,M}_{t_k}|^{p}\leq C M^{\frac{p}{2}} ,  
\end{align*}
which implies that $\mathbf{P}^{\Delta,M,*}_{t_{k}}\mu^{\otimes M}_{0}\in \mathcal{P}^*_{p}\big((\RR^{d})^{M}\big)$ and $ \mathrm{{P}}^{\Delta,M,*}_{t_{k}}\mu_{0} \in \mathcal{P}_{p}(\RR^{d})$ for any $k\geq 0$. Thanks to $\mu^{M,*}\in \mathcal{P}_{\bar{q}}^*((\RR^{d})^M)$, we have $\Gamma_1(\mu^{M,*})\in \mathcal{P}_{\bar{q}}(\RR^{d})$.
 Then  using the triangle  inequality and  the semigroup property of  $\mathbf{P}^{\Delta,M,*}_{t_k}$,   for any integers $m> 0$ (defined later) and  $k\geq m$, we obtain that
\begin{align}\label{eqq6.14}
\mathbb{W}_{\bar{q}}\big( {P}^{\Delta,M,*}_{t_{k}}\mu_{0},\Gamma_1(\mu^{M,*})\big) \leq I_1+I_2, \end{align} where
 \begin{align*}I_1 &=\mathbb{W}_{\bar{q}}\big(\Gamma_1(\mathbf{P}^{\Delta,M,*}_{t_{m}}\mathbf{P}^{\Delta,M,*}_{t_{k-m}}\mu^{\otimes M}_{0}),\Gamma_1(\mathbf{P}^{M,*}_{t_m}\mathbf{P}^{\Delta,M,*}_{t_{k-m}}\mu^{\otimes M}_0 )\big),\\
I_2 &=\mathbb{W}_{\bar{q}}\big(\Gamma_1(\mathbf{P}^{M,*}_{t_{m}}\mathbf{P}^{\Delta,M,*}_{t_{k-m}}\mu^{\otimes M}_{0}), \Gamma_1(\mu^{M,*})\big).
\end{align*}
Thanks to $\bar{q}\in [2,p_0 )\cap[2,p/((\alpha+2)\vee(2\alpha))]{\color{blue}}$,  applying Theorem \ref{th5.5} yields  that
\begin{align}\label{eqqc6.15}
I_1\leq \Big(\mathbb{E}|X^{1,M}_{t_m}-Y^{1,M}_{t_m}|^{\bar{q}}\Big)^{\frac{1}{\bar{q}}}\leq C \Delta^{\frac{1}{2}},
\end{align}
where $ {C}  =C_{\bar{q}, d, m}>0$ is independent of $\Delta$ and $M$. 
Due to $\bar{q}\leq p$, 
$\mathbf{P}^{\Delta,M,*}_{t_{k-m}}\mu^{\otimes M}_{0} \in \mathcal{P}^*_{\bar{q}}((\RR^{d})^{M})$. This together with Assumption \ref{ass7}  implies that
\begin{align}\label{eqq6.16}
I_2\leq K_{\bar{q}} e^{-\gamma_{\bar{q}} t_{m}}\mathbb{W}_{\bar{q}}\big( \mathrm{{P}}^{\Delta,M,*}_{t_{k-m}}\mu_{0}, \Gamma_1(\mu^{M,*})\big).
\end{align} 
Then inserting \eqref{eqqc6.15} and \eqref{eqq6.16} into \eqref{eqq6.14} leads to 
\begin{align*}
\mathbb{W}_{\bar{q}}\big( \mathrm{{P}}^{\Delta,M,*}_{t_{k}}\mu_{0},\Gamma_1(\mu^{M,*})\big)&\leq {C}_2\Delta^{\frac{1}{2}}+K_{\bar{q}}e^{-\gamma_{\bar{q}} t_{m}}\mathbb{W}_{\bar{q}}\big( \mathrm{{P}}^{\Delta,M,*}_{t_{k-m}}\mu_0 ,\Gamma_1(\mu^{M,*})\big).
\end{align*}
Take   $m=\Big[\frac{\log \bar{K}_q+1}{\gamma_{\bar{q}} \Delta}\Big]+1$, where $[a]$ is the  integer part of $a\in \RR$. One observes that  \begin{align*}
\frac{\log K_{\bar{q}}+1}{\gamma_{\bar{q}} }\leq t_{m} \leq \Big(\frac{\log K_{\bar{q}}+1}{\gamma_{\bar{q}}\Delta}+1\Big)\Delta\leq \frac{\log K_{\bar{q}}+1}{\gamma_{\bar{q}}}+1,~~~~K_{\bar{q}}e^{-\gamma_{\bar{q}} t_{m}}\leq \frac{1}{e}.
\end{align*} 
Therefore, we obtain that
\begin{align}\label{eqf6.17}
\mathbb{W}_{\bar{q}}\big( \mathrm{{P}}^{\Delta,M,*}_{t_{k}}\mu_{0},\Gamma_1(\mu^{M,*})\big)&\leq {C} \Delta^{\frac{1}{2}}+\frac{1}{e }\mathbb{W}_{\bar{q}}\big( \mathrm{{P}}^{\Delta,M,*}_{t_{k-m}}\mu_0 ,\Gamma_1(\mu^{M,*})\big).
\end{align} 
Then using Lemma \ref{L6.4} 
implies that for any $k\geq m$,
\begin{align}\label{6.17}
\mathbb{W}_{\bar{q}}\big( \mathrm{{P}}^{\Delta,M,*}_{t_{k}}\mu_{0},\Gamma_1(\mu^{M,*})\big) \leq {C}  \Delta^{\frac{1}{2}}+  {C}_1 e^{-\frac{k}{m}},
\end{align}
where  
\begin{align*}
 {C}_1&=e  \sup_{1\leq n\leq m-1}\mathbb{W}_{\bar{q}}\big( \mathrm{{P}}^{\Delta,M,*}_{t_{n}}\mu_{0},\Gamma_1(\mu^{M,*})\big) \nn\
 \\&\leq e \Big[\sup_{\Delta\in (0,\Delta^*_1]}\sup_{t\geq 0}\Big(\mathbb{E}|\bar{Y}^{1,M}_{t}|^{\bar{q}}\Big)^{\frac{1}{q}}+\Big(\Gamma_1(\mu^{M,*})(|\cdot|^{\bar{q}})\Big)^{\frac{1}{\bar{q}}}\Big]<\infty.
\end{align*}
Let
$\displaystyle\lambda=\frac{\gamma_{\bar{q}}}{1+\gamma_{\bar{q}}+\log K_{\bar{q}}}.$ Obviously, 
$ \displaystyle
\frac{k}{m}\geq \displaystyle\frac{\gamma_{\bar{q}} k\Delta}{ 1+\gamma_{\bar{q}} \Delta+\log K_{\bar{q}}} \geq \lambda k\Delta.
$
Then the desired assertion follows from \eqref{6.17} directly.
The proof is complete.
\end{proof}

\subsection{The Exponential Ergodicity of Numerical Solution}\label{sb6.2}
{This subsection explores  the asymptotic properties  of the TEM numerical solutions including the exponential stability and the exponential ergodicity.  
 Now we impose the following assumption.}

\begin{assp}\label{ass6} 
There exist  constants~$p_0\geq 2$, $\bar{\lambda}_{1}>\bar{\lambda}_{2}\geq 0$ such that
$$
2(x_{1}-x_{2})^{T}\big(f(x_{1},\mu_{1})-f(x_{2},\mu_{2})\big)+(p_0-1)|g(x_1,\mu_1)-g(x_{2},\mu_2)|^{2}\leq -\bar{\lambda}_1|x_1-x_{2}|^{2}+\bar{\lambda}_2\mathbb{W}^{2}_{2}(\mu_{1},\mu_{2})
$$
for any~$x_{1},~x_{2}\in\RR^{d}$ and~$\mu_{1},~\mu_{2}\in\mathcal{P}_{2}(\RR^{d})$.
\end{assp}

\begin{rem}\label{rm6.7}
If $f(0,\boldsymbol{\delta}_0)=g(0,\boldsymbol{\delta}_0)=0$, then Assumption \ref{ass6} implies that 
$$
2x^{T}f(x,\mu)+(p_0-1)|g(x,\mu)|^{2}\leq -\bar{\lambda}_1|x|^{2}+\bar{\lambda}_2\mu(|\cdot|^2).
$$
\end{rem}

 Under Assumptions  \ref{ass1}, \ref{ass5} and \ref{ass6} and by a similar argument to \cite[Theorem 3.1]{WFY2018},  MV-SDE \eqref{eq3.1} admits a unique   invariant  probability measure.
\begin{lemma}\label{L7.1}
Let Assumptions \ref{ass1},  \ref{ass5}, \ref{ass6}  hold and $\mu_0\in  \mathcal{P}_{p}(\RR^{d})$. Then for any $q\in [2,p_0\wedge p]$, MV-SDE \eqref{eq3.1} admits a unique invariant probability measure $\mu^*\in \mathcal{P}_{p}(\RR^{d})$ such that 
\begin{align*}
\mathbb{W}_{q}(\mathrm{P}^{*}_{t}\mu_{0},\mu^*)\leq \mathbb{W}_{q}(\mu_0,\mu^*)e^{-\frac{(\bar{\lambda}_1-\bar{\lambda}_2)t}{2}},~~~\forall t\geq0.
\end{align*}
\end{lemma}  

If $f(0,\boldsymbol{\delta}_0)=g(0,\boldsymbol{\delta}_0)=0$,  MV-SDE \eqref{eq3.1} and its corresponding  IPS \eqref{eq2} have the trivial solutions $X_t \equiv 0$ a.s.  and $X_t^{i,M} \equiv 0$ a.s.  $(i=1,\cdots, M)$, respectively. Next  we  give an exponential stability criterion on  their trivial solutions. Since the proof is routine we omit it to avoid redundancy.
 \begin{cor}\label{th6.2}
Let Assumptions \ref{ass1},   \ref{ass6}  and  $X_{0}\in L^{p_0}_{0}$ hold.  If $f(0,\boldsymbol{\delta}_0)=g(0,\boldsymbol{\delta}_0)=0$, then for any $q\in [2,p_0]$, the solution to MV-SDE \eqref{eq3.1} satisfies
\begin{align*}
\mathbb{E}|X_{t}|^{q}\leq \mathbb{ E}|X_0|^{q}e^{-\frac{q(\bar{\lambda}_1-\bar{\lambda}_2)t}{2}},~~~~\sup_{1\leq i\leq M}\mathbb{E}|X^{i,M}_{t}|^q\leq \mathbb{ E}|X_0|^{q}e^{-\frac{q(\bar{\lambda}_1-\bar{\lambda}_2)t}{2}},~~~~\forall t\geq0.
\end{align*}
\end{cor}  
\begin{lemma}\label{L7.3}
Let Assumptions \ref{ass1}, \ref{ass3} and \ref{ass6}  hold and $\mu^{M}_0,\nu^{M}_0\in \mathcal{P}^*_p((\RR^{d})^M)$.
 Then for any $q\in [2,p_0\wedge p]$ and $\varepsilon\in (0,q(\bar{\lambda}_1-\bar{\lambda}_2)/2)$, there exists a  $\Delta_2^*\in(0,1]$ such that for any $\Delta\in (0,\Delta_2^*]$, the semigroup  $\mathbf{P}^{\Delta,M,*}_{\cdot}$ satifies 
 \begin{align*}
 &\mathbb{W}^q_{q}\Big(\Gamma_1(\mathbf{P}^{\Delta,M,*}_{t_{k}}\mu_0^M),
 \Gamma_1(\mathbf{P}^{\Delta,M,*}_{t_k}\nu_0^M)\Big) 
    \leq 
   2^{q-1}\big(\Gamma_1(\mu_0^M)(|\cdot|^q)+\Gamma_1(\nu_0^M)(|\cdot|^q)\big)
   e^{-\big(\frac{q(\bar{\lambda}_1-\bar{\lambda}_2)}{2}-\varepsilon\big)t_{k}},\nn\\
  &\mathbb{W}^q_{q}\Big(\mathbf{P}^{\Delta,M,*}_{t_{k}}\mu_0^M,\mathbf{P}^{\Delta,M,*}_{t_k}\nu_0^M\Big) \leq 
  M^{\frac{q}{2}}2^{q-1}\big(\Gamma_1(\mu_0^M)(|\cdot|^q)+\Gamma_1(\nu_0^M)(|\cdot|^q)\big)
  e^{-\big(\frac{q(\bar{\lambda}_1-\bar{\lambda}_2)}{2}-\varepsilon\big)t_{k}}, ~~ k\geq 0. \end{align*}

 \end{lemma} 
\begin{proof} 
Let $X_0^{M} \sim \mu^{M}_{0}$ and $Z^{M}_{0}  \sim \nu^{M}_{0}$ be the initial data of the TEM numerical solutions $\{Y^{M}_{t_k}\}$ and $\{Z^{M}_{t_k}\}$, respectively.
Since the proof has some similar arguments as Theorem \ref{L7.4}, we only give the outline.
For brevity, define $\Theta^{i,M}_{t_k}=Y^{i,M}_{t_{k}}-Z^{i,M}_{t_{k}}$, $\tilde{\Theta}^{i,M}_{t}=\tilde{Y}^{i,M}_{t}-\tilde{Z}^{i,M}_{t}$,  $\bar{\Theta}^{i,M}_{t}=\bar{Y}^{i,M}_{t}-\bar{Z}^{i,M}_{t}$. Moreover, define
 \begin{align*}
 F\big(Y^{i,M}_{t_{k}},Z^{i,M}_{t_{k}}\big)&=f\big(Y^{i,M}_{t_k},\mathcal{L}^{Y,M}_{t_k}\big)-f\big(Z^{i,M}_{t_k},\mathcal{L}^{Z,M}_{t_k}\big),
 \nn\\
  G\big(Y^{i,M}_{t_{k}},Z^{i,M}_{t_{k}}\big)&=g\big(Y^{i,M}_{t_k},\mathcal{L}^{Y,M}_{t_k}\big)-g\big(Z^{i,M}_{t_k},\mathcal{L}^{Z,M}_{t_k}\big).\end{align*}
By Assumption \ref{ass1} and \ref{ass3}, Remarks \ref{r1} and \ref{rm3.1}, it follows that
\begin{equation}\label{5.29*}
\begin{aligned}
&\big|F\big(Y^{i,M}_{t_{k}},Z^{i,M}_{t_{k}}\big)\big|\leq H\Delta^{-\kappa}\big|\Theta^{i,M}_{t_{k}}\big|+K\mathbb{W}_{2}\big(\mathcal{L}^{Y,M}_{t_{k}},\mathcal{L}^{Z,M}_{t_{k}}\big),
\\&\big|G\big(Y^{i,M}_{t_{k}},Z^{i,M}_{t_{k}}\big)\big|^{2}\leq 2H\Delta^{-\kappa}\big|\Theta^{i,M}_{t_{k}}\big|^2+2L_3\mathbb{W}^{2}_{2}\big(\mathcal{L}^{Y,M}_{t_{k}},\mathcal{L}^{Z,M}_{t_{k}}\big).
\end{aligned}
\end{equation} 
For any $\varepsilon\in (0,q(\bar{\lambda}_1-\bar{\lambda}_2)/2)$,
define $\varrho=\frac{q(\bar{\lambda}_1-\bar{\lambda}_2)}{2}-\frac{3\varepsilon}{4} $.
For any $\Delta\in (0,1]$, using the It\^o formula, we derive from \eqref{eq4.35} and Assumption \ref{ass6}  that for any $t\in[t_{k},t_{k+1})$,
\begin{align*}
 e^{\varrho t}\mathbb{E}\big|\tilde{\Theta}^{i,M}_{t}\big|^{q}
&\leq e^{\varrho t_{k}}\mathbb{E}\big|\Theta^{i,M}_{t_{k}}\big|^{q}
+\varrho\int_{t_{k}}^{t}
e^{\varrho s}\mathbb{E}\big|\tilde{\Theta}^{i,M}_{s}\big|^{q}\mathrm{d}s\nn\\
&~~~+\frac{q}{2}\mathbb{E}\int_{t_{k}}^{t}e^{\varrho s}\big|\tilde{\Theta}^{i,M}_{s}\big|^{q-2}
\Big[2(\tilde{\Theta}^{i,M}_{s})^{T}F\big(Y^{i,M}_{t_{k}},Z^{i,M}_{t_{k}}\big)+(q-1)\big|G\big(Y^{i,M}_{t_{k}},Z^{i,M}_{t_{k}}\big)
\big|^{2}\Big]\mathrm{d}s.
\end{align*}
 By the similar arguments as \eqref{eq*2} as in Theorem \ref{L7.4}, for any $\delta>0$ (defined later),  we obtain 
\begin{align}\label{eq*+2}
\frac{1}{M}\sum_{i=1}^{M}\mathbb{E}\big|\tilde{\Theta}^{i,M}_{t}\big|^{q}&\leq e^{\varrho (t_k-t)}\frac{1}{M}\sum_{i=1}^{M}\mathbb{E}\big|\Theta^{i,M}_{t_{k}}\big|^{q}
-\Big(\frac{q }{2} (1- \delta)\bar{\lambda}_1-\varrho\Big)\int_{t_{k}}^{t}e^{\varrho (s-t)}
\frac{1}{M}\sum_{i=1}^{M}\mathbb{E}\big|\tilde{\Theta}^{i,M}_{s}\big|^{q}\mathrm{d}s\nn\\
&~~+\frac{q  \bar{\lambda}_1}{2\delta} \int_{t_{k}}^{t}e^{\varrho (s-t)}\frac{1}{M}\sum_{i=1}^{M}\mathbb{E}\Big( \big|\tilde{\Theta}^{i,M}_{s}\big|^{q-2}
\big|\tilde{\Theta}^{i,M}_{s}-\Theta^{i,M}_{t_{k}}\big|^2\Big) \mathrm{d}s\nn\\
&~~ +\frac{q  \bar{\lambda}_2}{2 }\int_{t_{k}}^{t} e^{\varrho (s-t)} \frac{1}{M}\sum_{i=1}^{M}\mathbb{E}\Big( \big|\tilde{\Theta}^{i,M}_{s}\big|^{q-2}\mathbb{W}^{2}_{2}\big(\mathcal{L}^{Y,M}_{t_{k}},\mathcal{L}^{Z,M}_{t_{k}}\big)\Big)\mathrm{d}s \nn\\
&~~+q \int_{t_{k}}^{t} e^{\varrho (s-t)}\frac{1}{M}\sum_{i=1}^{M}\mathbb{E}\Big(\big|\tilde{\Theta}^{i,M}_{s}\big|^{q-2}\big|
\tilde{\Theta}^{i,M}_{s}-\Theta^{i,M}_{t_{k}}\big| \big|F\big(Y^{i,M}_{t_{k}},Z^{i,M}_{t_{k}}\big)
 \big|\Big)\mathrm{d}s.
 \end{align}
Now we estimate the terms on the right side of the above inequality. Using the elementary inequality, \eqref{5.29*} and the following inequality
$
\mathbb{E}\mathbb{W}^{q}_{2}\big(\mathcal{L}^{Y,M}_{t_{k}},\mathcal{L}^{Z,M}_{t_{k}}\big)\leq \frac{1}{M}\sum_{i=1}^{M}\mathbb{E}|\Theta^{i,M}_{t_{k}}|^{q},
$
we  derive that  
\begin{align}\label{5.32*}
&\frac{1}{M}\sum_{i=1}^{M}\mathbb{E}\big|\tilde{\Theta}^{i,M}_{s}-\Theta^{i,M}_{t_{k}}\big|^q\leq \frac{1}{M}\sum_{i=1}^{M}\Big(2^ {q}\mathbb{E}\big|F(Y^{i,M}_{t_{k}},Z^{i,M}_{t_{k}})\big|^{q}\Delta^{q}+2^{q}\mathbb{E}\big|G(Y^{i,M}_{t_{k}},Z^{i,M}_{t_{k}})\big|^{q}\Delta^{\frac{q}{2}}\Big)\nn\
\\\leq & \frac{C}{M}\sum_{i=1}^{M}\Big(\mathbb{E}\big|\Theta^{i,M}_{t_{k}}\big|^{q}\Delta^{\frac{q(1-\kappa)}{2}}+\mathbb{E}
\mathbb{W}^{q}_{2}\big(\mathcal{L}^{Y,M}_{t_{k}},\mathcal{L}^{Z,M}_{t_{k}}\big)\Delta^{\frac{q}{2}}\Big)\leq \frac{C}{M}\sum_{i=1}^{M}\mathbb{E}|\Theta^{i,M}_{t_{k}}|^{q}\Delta^{\frac{q(1-\kappa)}{2}}.
\end{align}
Using the H$\ddot{\hbox{o}}$lder inequality twice and \eqref{5.32*}, and the Young inequality  yields that
\begin{align}\label{eq*+3}
\frac{1}{M}\sum_{i=1}^{M}\mathbb{E}\Big( \big|\tilde{\Theta}^{i,M}_{s}\big|^{q-2}
\big|\tilde{\Theta}^{i,M}_{s}-\Theta^{i,M}_{t_{k}}\big|^2\Big)&\leq \frac{1}{M}\sum_{i=1}^{M}\Big(\mathbb{E}\big|\tilde{\Theta}^{i,M}_{s}\big|^q\Big)^{\frac{q-2}{q}}
\Big(\mathbb{E}\big|\tilde{\Theta}^{i,M}_{s}-\Theta^{i,M}_{t_{k}}\big|^q \Big)^{\frac{2}{q} }\nn\\
&\leq\Big(\frac{1}{M}\sum_{i=1}^{M}\mathbb{E}
\big|\tilde{\Theta}^{i,M}_{s}\big|^q\Big)^{\frac{q-2}{q}}\Big( \frac{1}{M}\sum_{i=1}^{M}\mathbb{E}\big|\tilde{\Theta}^{i,M}_{s}-\Theta^{i,M}_{t_{k}}\big|^q\Big)^{\frac{2}{q}}\nn\
\\&\leq C \Delta^{\frac{2}{3}}\Big(\frac{1}{M}\sum_{i=1}^{M}\mathbb{E}
\big|\tilde{\Theta}^{i,M}_{s}\big|^q\Big)^{\frac{q-2}{q}}\Big( \frac{1}{M}\sum_{i=1}^{M}\mathbb{E}\big|\Theta^{i,M}_{t_{k}}\big|^q\Big)^{\frac{2}{q}}\nn\
\\&\leq C \Delta^{\frac{2}{3}}\Big(\frac{1}{M}\sum_{i=1}^{M}\mathbb{E}\big|\tilde{\Theta}^{i,M}_{s}\big|^q+\frac{1}{M}\sum_{i=1}^{M}\mathbb{E}\big|\Theta^{i,M}_{t_k}\big|^q\Big),
\end{align}
where $C$ is a positive constant only depending on $q$.
Using the Young inequality  we have
\begin{align}\label{eq*+4}
 \frac{1}{M}\sum_{i=1}^{M}\mathbb{E}\Big( \big|\tilde{\Theta}^{i,M}_{s}\big|^{q-2}\mathbb{W}^{2}_{2}\big(\mathcal{L}^{Y,M}_{t_{k}},\mathcal{L}^{Z,M}_{t_{k}}\big)\Big)
 \leq  \frac{ q-2  }{q} \frac{1}{M}\sum_{i=1}^{M}\mathbb{E} \big|\tilde{\Theta}^{i,M}_{s}\big|^{q} 
+  \frac{  2  }{q}\E \mathbb{W}^{q}_{2}\big(\mathcal{L}^{Y,M}_{t_{k}},\mathcal{L}^{Z,M}_{t_{k}}\big).
\end{align}
Using \eqref{5.8+} and \eqref{5.32*} implies that  
\begin{align*}
\E \mathbb{W}^{q}_{2}\big(\mathcal{L}^{Y,M}_{t_{k}},\mathcal{L}^{Z,M}_{t_{k}}\big)&\leq \frac{1}{M}\sum_{i=1}^{M}\mathbb{E}\big|\Theta^{i,M}_{t_{k}}\big|^{q}
\leq (1+\delta)\frac{1}{M}\sum_{i=1}^{M}\mathbb{E}\big|\tilde{\Theta}^{i,M}_{s}\big|^{q}+C_{\delta}\frac{1}{M}\sum_{i=1}^{M}\mathbb{E}\big|\tilde{\Theta}^{i,M}_{s}-\Theta^{i,M}_{t_{k}}\big|^{q}\nn\
\\&\leq (1+\delta)\frac{1}{M}\sum_{i=1}^{M}\mathbb{E}\big|\tilde{\Theta}^{i,M}_{s}\big|^{q}+C_{\delta}\frac{1}{M}\sum_{i=1}^{M}\mathbb{E}\big|\Theta^{i,M}_{t_{k}}\big|^{q}\Delta^{\frac{q(1-\kappa)}{2}}.
\end{align*}
Inserting the above inequality into \eqref{eq*+4} yields that
\begin{align}\label{5.35*}
 &\frac{1}{M}\sum_{i=1}^{M}\mathbb{E}\Big( \big|\tilde{\Theta}^{i,M}_{s}\big|^{q-2}\mathbb{W}^{2}_{2}\big(\mathcal{L}^{Y,M}_{t_{k}},\mathcal{L}^{Z,M}_{t_{k}}\big)\Big)\nn\
 \\ \leq & \Big(\frac{q-2}{q}+\frac{2}{q}(1+\delta)\Big)\frac{1}{M}\sum_{i=1}^{M}\mathbb{E}\big|\tilde{\Theta}^{i,M}_{s}\big|^{q}+\frac{2C_{\delta}}{q}\frac{1}{M}\sum_{i=1}^{M}\mathbb{E}\big|\Theta^{i,M}_{t_{k}}\big|^{q}\Delta^{\frac{q(1-\kappa)}{2}}\nn\
 \\ \leq& (1+\delta)\frac{1}{M}\sum_{i=1}^{M}\mathbb{E}\big|\tilde{\Theta}^{i,M}_{s}\big|^{q}+C_{\delta}\frac{1}{M}\sum_{i=1}^{M}\mathbb{E}\big|\Theta^{i,M}_{t_{k}}\big|^{q}\Delta^{\frac{q(1-\kappa)}{2}}.
\end{align}
 Furthermore, using the Young inequality, the H\"older inequality,  \eqref{5.29*}, \eqref{5.32*} and  the fact $\kappa=1/3-\theta \in (0,1/3)$ implies that for any $\varepsilon>0$,
\begin{align}\label{eq*+5}
  &\frac{1}{M}\sum_{i=1}^{M}\mathbb{E}\Big(\big|\tilde{\Theta}^{i,M}_{s}\big|^{q-2}\big|
\tilde{\Theta}^{i,M}_{s}-\Theta^{i,M}_{t_{k}}\big| \big|F(Y^{i,M}_{t_{k}},Z^{i,M}_{t_{k}})
 \big|\Big)\nn\\
 \leq & 
 \frac{\e  }{ 4q}\frac{1}{M}\sum_{i=1}^{M}\mathbb{E} \big|\tilde{\Theta}^{i,M}_{s}\big|^{q} 
+ \frac{C_{\e}}{M}\sum_{i=1}^{M}\mathbb{E}\Big(\big|\tilde{\Theta}^{i,M}_{s}-\Theta^{i,M}_{t_{k}}\big|^{\frac{q}{2}} \big|F\big(Y^{i,M}_{t_{k}},Z^{i,M}_{t_{k}}\big) \big| ^{\frac{q}{2}} \Big) \nn\\
\leq &  \frac{\e  }{4 q } \frac{1}{M}\sum_{i=1}^{M}\mathbb{E} \big|\tilde{\Theta}^{i,M}_{s}\big|^{q} 
+  C_{\e} \Big(\frac{1}{M}\sum_{i=1}^{M}\mathbb{E}\big|\tilde{\Theta}^{i,M}_{s}-\Theta^{i,M}_{t_{k}}\big|^{q} \Big)^{\frac{1}{2}}\Big(\frac{1}{M}\sum_{i=1}^{M}\mathbb{E}\big|F(Y^{i,M}_{t_{k}},Z^{i,M}_{t_{k}})\big|
^{q}\Big)^{\frac{1}{2}} 
 \nn\\
\leq & 
  \frac{\e  }{ 4q }\frac{1}{M}\sum_{i=1}^{M} \mathbb{E} \big|\tilde{\Theta}^{i,M}_{s}\big|^{q} 
+   C_{\e} \Delta^{\frac{3q\theta}{4}} \frac{1}{M}\sum_{i=1}^{M}\mathbb{E}\big|\Theta^{i,M}_{t_{k}}\big|^q. 
\end{align}
Inserting inequalities \eqref{eq*+3}, \eqref{5.35*} and \eqref{eq*+5} into \eqref{eq*+2}, letting $\delta=\frac{\e }{2q(\bar{\lambda}_1+\bar{\lambda}_2)}>0$ and recalling the definition of $\var$ yields that 
 \begin{align*} 
\frac{1}{M}\sum_{i=1}^{M}\mathbb{E}\big|\tilde{\Theta}^{i,M}_{t}\big|^{q}&\leq e^{\varrho(t_k-t)}\frac{1}{M}\sum_{i=1}^{M}\mathbb{E}\big|\Theta^{i,M}_{t_{k}}\big|^{q}
- \Big(\frac{\e}{4}-C \Delta^{\frac{2}{3}}\Big)\int_{t_{k}}^{t}e^{\varrho (s-t)}
\frac{1}{M}\sum_{i=1}^{M}\mathbb{E}|\tilde{\Theta}^{i,M}_{s}|^{q}\mathrm{d}s\nn\
\\&~~~ + C \Delta^{1+ \frac{2}{3}\wedge\frac{3q\theta}{4} } \frac{1}{M}\sum_{i=1}^{M} \mathbb{E} 
|\Theta^{i,M}_{t_{k}}|^q.
\end{align*}
Choose $\Delta_2^*\in (0,1]$ small sufficiently such that   
\begin{align}\label{e5.15}
  C (\Delta_2^*)^{\frac{2}{3}} \leq \frac{\e }{4},~~~ ~ C (\Delta_2^*)^{ \frac{2}{3}\wedge\frac{3q\theta}{4} } \leq \frac{\e }{8},~~~~   \var^2  \Delta_2^*   \leq \frac{\e }{4}.
 \end{align} 
 Therefore, for any $\Delta\in (0,\Delta_2^*]$, we obtain that for any $t\in [t_k,t_{k+1})$,
 \begin{align*} 
\frac{1}{M}\sum_{i=1}^{M}\mathbb{E}\big|\tilde{\Theta}^{i,M}_{t}\big|^{q}&\leq \Big(e^{-\varrho(t-t_k)}+\frac{\e }{8}\Delta\Big)\frac{1}{M}\sum_{i=1}^{M}\mathbb{E}\big|\Theta^{i,M}_{t_{k}}\big|^{q}.
\end{align*} 
 Taking the limit inferior on both sides of the above inequality, using the Fatou Lemma and the Taylor expansion, by \eqref{e5.15} we arrive at
\begin{align*}
&\frac{1}{M}\sum_{i=1}^{M}\mathbb{E}\big|\bar{\Theta}^{i,M}_{t_{k+1}}\big|^{ {q}}\leq \frac{1}{M}\sum_{i=1}^{M}\liminf_{t\rightarrow t^{-}_{k+1}}\mathbb{E}|\tilde{\Theta}^{i,M}_{t}|^{q} \leq  \Big( e^{-\var  \Delta}+ \frac{\e }{8}\Delta\Big)\frac{1}{M}\sum_{i=1}^{M}\mathbb{E}|\Theta^{i,M}_{t_{k}}|^{q}\nn\\
  \leq&   \Big(1-  \var \Delta +\frac{\var^2}{2}\Delta^2+ \frac{\e }{8}\Delta\Big)\frac{1}{M}\sum_{i=1}^{M}\mathbb{E}|\Theta^{i,M}_{t_{k}}|^{q}\leq   \Big(1- \big(\var-\frac{\e }{4} \big)\Delta  \Big)\frac{1}{M}\sum_{i=1}^{M}\mathbb{E}|\Theta^{i,M}_{t_{k}}|^{q}\nn\
  \\ \leq&  \Big(1- \big(\var-\frac{\e }{4} \big)\Delta  \Big)\frac{1}{M}\sum_{i=1}^{M}\mathbb{E}|\bar{\Theta}^{i,M}_{t_{k}}|^{q},
    \end{align*}
 where the last inequality used the inequality $|\pi_{\Delta}(x)-\pi_{\Delta}(y)|\leq |x-y|, ~\forall x,y \in \RR^{d}$. Solving the above difference inequality and recalling the definitions of $\var$, we yield 
 \begin{align}\label{eN5.28}
&\frac{1}{M}\sum_{i=1}^{M}\mathbb{E}\big|\bar{\Theta}^{i,M}_{t_{k+1}}\big|^{ {q}}\leq \frac{1}{M}\sum_{i=1}^{M}\mathbb{E}|X^{i,M}_{0}-Z^{i,M}_{0}|^{q}e^{-\big(\frac{q(\bar{\lambda}_1-\bar{\lambda}_2)}{2}-\varepsilon \big)t_{k+1}}\nn\
\\ \leq& 2^{q-1}\big(\Gamma_1(\mu_0^M)(|\cdot|^q)+\Gamma_1(\nu_0^M)(|\cdot|^q)\big)e^{-\big(\frac{q(\bar{\lambda}_1-\bar{\lambda}_2)}{2}-\varepsilon \big)t_{k+1}}.
 \end{align}
Therefore, \eqref{eN5.28} together with the fact that  $\mu^{M}_0,\nu^{M}_0\in \mathcal{P}^*_p((\RR^{d})^M) \subseteq\mathcal{P}^*_q((\RR^{d})^M)$ implies that 
\begin{align*}
\mathbb{W}^q_{q}&\Big(\Gamma_1(\mathbf{P}^{\Delta,M,*}_{t_{k}}\mu_0^M),
\Gamma_1(\mathbf{P}^{\Delta,M,*}_{t_k}\nu_0^M)\Big) = \frac{1}{M}\sum_{i=1}^{M}\mathbb{W}^q_{q}
\Big(\Gamma_i(\mathbf{P}^{\Delta,M,*}_{t_{k}}\mu_0^M),\Gamma_i(
\mathbf{P}^{\Delta,M,*}_{t_k}\nu_0^M)\Big)\nn\
\\&\leq \frac{1}{M}\sum_{i=1}^{M}\mathbb{E}\big|\bar{\Theta}^{i,M}_{t_{k }}\big|^{ {q}}\leq 2^{q-1}\big(\Gamma_1(\mu_0^M)(|\cdot|^q)+\Gamma_1(\nu_0^M)(|\cdot|^q)\big)e^{-\big(\frac{q(\bar{\lambda}_1-\bar{\lambda}_2)}{2}-\varepsilon \big)t_{k}}.
\end{align*}
Similarly, we also obtain that
\begin{align*}
\mathbb{W}^q_{q}\Big(\mathbf{P}^{\Delta,M,*}_{t_{k}}\mu_0^M, \mathbf{P}^{\Delta,M,*}_{t_{k}}\nu_0^M\Big) &\leq \mathbb{E}\Big(\sum_{i=1}^{M}|\bar{\Theta}^{i,M}_{t_{k}}|^2\Big)^{\frac{q}{2}}
\nn\\
&\leq  M^{\frac{q}{2}}2^{q-1}\big(\Gamma_1(\mu_0^M)(|\cdot|^q)+\Gamma_1(\nu_0^M)(|\cdot|^q)\big)e^{-\big(\frac{q(\bar{\lambda}_1-\bar{\lambda}_2)}{2}-\varepsilon \big)t_{k}}.
\end{align*}
The proof is complete.
\end{proof}
 
 If $f(0,\boldsymbol{\delta}_0)=g(0,\boldsymbol{\delta}_0)=0$, by Lemma \ref{L7.3},   the TEM numerical solution is stable exponentially.   
 \begin{cor}\label{th6.3}
Let Assumptions \ref{ass1},  \ref{ass6} hold and $X_0\in L^{p_0}_{0}$. If $f(0,\boldsymbol{\delta}_0)=g(0,\boldsymbol{\delta}_0)=0$, then for any $q\in [2, p_0]$ and $\varepsilon\in (0,q(\bar{\lambda}_1-\bar{\lambda}_2)/2)$, the TEM numerical solutions  satisfy
\begin{align*}
\mathbb{E}\big|Y^{i,M}_{t_{k}}\big|^q\leq \mathbb{E}\big|\bar{Y}^{i,M}_{t_{k}}\big|^q&\leq \mathbb{E}\big|X_0\big|^qe^{-\big(\frac{q(\bar{\lambda}_{1}-\bar{\lambda}_2)}{2}-\varepsilon\big)t_{k}},~~~ i=1,\cdots, M.
\end{align*}
 \end{cor}

  By utilizing Theorem~\ref{L7.4} and Lemma~\ref{L7.3}, we establish the existence and uniqueness of the numerical invariant measure via constructing a Cauchy sequence of numerical solution distributions in the complete space $\mathcal{P}^{*}_{q}((\RR^{d})^M)$. Furthermore, we demonstrate that the one-particle marginal distribution of the TEM numerical solution converges exponentially to that of the numerical invariant measure.
\begin{theorem}\label{th7.2}
Let  Assumptions \ref{ass1}, \ref{ass5} and \ref{ass6} hold. Then  for any $\Delta\in (0,\Delta_1^*\wedge \Delta^{*}_{2}]$,  $q\in [2,p_0\wedge p]$ and $\varepsilon\in \big(0,\frac{q(\bar{\lambda}_1-\bar{\lambda}_2)}{2}\big)$, 
the TEM numerical solution with any initial distribution $\mu^M\in \mathcal{P}^*_{p}((\RR^{d})^{M})$ admits a unique  invariant probability measure $\mu^{\Delta,M,*}\in \mathcal{P}^*_{p}((\RR^{d})^{M})$ satisfying 
\begin{align}\label{eq++1} 
& \mathbb{W}^{q}_{q}\big(\mathbf{P}^{\Delta,M,*}_{t_{k}}\mu^{ M} ,\mu^{\Delta,M,*}\big)  \leq CM^{\frac{q}{2}}e^{-\big(\frac{q(\bar{\lambda}_1-\bar{\lambda}_2)}{2}-\varepsilon\big)t_{k}},\\
 &   \mathbb{W}^{q}_{q}\big( \Gamma_1( \mathbf{P}^{\Delta,M,*}_{t_{k}}\mu^{ M} ) ,\Gamma_1(\mu^{\Delta,M,*})\big)  \leq C e^{-\big(\frac{q(\bar{\lambda}_1-\bar{\lambda}_2)}{2}-\varepsilon\big)t_{k}}.\label{eq++3}   
\end{align} where $C $   is independent of $M$, $\Delta$ and $k$, and $\Delta_1^*$ and $\Delta_2^*$ are given in Theorem \ref{L7.4} and Lemma \ref{L7.3}, respectively. 
 \end{theorem}
\begin{proof}  
Fix $q\in [2, p_0\wedge p]$. It follows from 
 Theorem \ref{L7.4}  that
 $\mathbf{P}^{\Delta,M,*}_{t_{k}}\boldsymbol{\delta}_0^{\otimes M}\in \mathcal{P}^*_{p}((\RR^{d})^{M}).$ 
Using Lemma  \ref{L7.3} and Theorem \ref{L7.4},  for any $\varepsilon\in \big(0,\frac{q(\bar{\lambda}_1-\bar{\lambda}_2)}{2}\big)$ and  $\Delta\in (0,\Delta_1^*\wedge \Delta_2^*]$, 
 \begin{align}\label{eqf6.36}
 \mathbb{W}^{q}_{q}\left(\mathbf{P}^{\Delta,M,*}_{t_k}\boldsymbol{\delta}^{\otimes M}_0,
\mathbf{P}^{\Delta,M,*}_{t_{k+n}}\boldsymbol{\delta}^{\otimes M}_0\right)
 &=\mathbb{W}^{q}_{q}\Big(\mathbf{P}^{\Delta,M,*}_{t_k}\boldsymbol{\delta}^{\otimes M}_0,
 \mathbf{P}^{\Delta,M,*}_{t_k}(\mathbf{P}^{\Delta,M,*}_{t_n}\boldsymbol{\delta}^{\otimes M}_0)\Big) \nn\\
 &\leq C M^{\frac{q}{2}}e^{-\big(\frac{q(\bar{\lambda}_1-\bar{\lambda}_2)}{2}-\varepsilon\big)t_{k}},
 \end{align} 
which  implies that $\left\{\mathbf{P}^{\Delta,M,*}_{t_k}\boldsymbol{\delta}^{\otimes M}_0\right\}_{k=1}^{\infty}$ is a cauchy sequence in $\mathcal{P}^*_{q}((\RR^{d})^{M})$.  Due to the completeness of space $\mathcal{P}^*_{q}((\RR^{d})^{M})$ under the $L^{q}$-Wasserstein distance,  there exists a unique probability measure $\mu^{\Delta,M,*}\in \mathcal{P}^*_{q}((\RR^{d})^{M})$ such that $\mathbb{W}_{q}\left(\mathbf{P}^{\Delta,M,*}_{t_{k+n}}\boldsymbol{\delta}^{\otimes M}_0,\mu^{\Delta,M,*}\right)\rightarrow 0 $ as $n\rightarrow\infty$. Combining this and \eqref{eqf6.36}, and using the continuity of $L^{q}$-Wasserstein distance (see \cite[p.97, Corollary 6.11]{CV2008})  we yield
\begin{align*} 
\mathbb{W}^q_{q}\left(\mathbf{P}^{\Delta,M,*}_{t_k}\boldsymbol{\delta}^{\otimes M}_0,\mu^{\Delta,M,*}\right)&= \lim_{n\rightarrow\infty}\mathbb{W}^q_{q}\left(\mathbf{P}^{\Delta,M,*}_{t_k}\boldsymbol{\delta}^{\otimes M}_0,\mathbf{P}^{\Delta,M,*}_{t_{k+n}}\boldsymbol{\delta}^{\otimes M}_0\right) \nn\
\\&\leq C M^{\frac{q}{2}}e^{-\big(\frac{q(\bar{\lambda}_1-\bar{\lambda}_2)}{2}-\varepsilon\big)t_{k}}.
\end{align*}
Furthermore, employing the continuity of $L^{q}$-Wasserstein distance  again and \eqref{eqf6.36} leads that
\begin{align*}
\mathbb{W}^q_{q}\left(\mathbf{P}^{\Delta,M,*}_{t_n}\mu^{\Delta,M,*},\mu^{\Delta,M,*}\right)&=\lim_{k\rightarrow\infty}\mathbb{W}^q_{q}\big(\mathbf{P}^{\Delta,M,*}_{t_n}(\mathbf{P}^{\Delta,M,*}_{t_k}\boldsymbol{\delta}^{\otimes M}_0),\mathbf{P}^{\Delta,M,*}_{t_k}\boldsymbol{\delta}^{\otimes M}_0\big)\nn\
\\&=\lim_{k\rightarrow\infty}\mathbb{W}^q_{q}(\mathbf{P}^{\Delta,M,*}_{t_{k+n}}\boldsymbol{\delta}^{\otimes M}_0,\mathbf{P}^{\Delta,M,*}_{t_{k}}\boldsymbol{\delta}^{\otimes M}_0)=0,~~\forall n\geq 0.
\end{align*}
which verifies that $\mu^{\Delta,M,*}\in \mathcal{P}^*_{q}((\RR^{d})^{M})$ is indeed invariant. Obviously,  $\Gamma_1(\mathbf{P}^{\Delta,M,*}_{t_k}\mu^{\Delta,M,*})=\Gamma_1(\mu^{\Delta,M,*}),~~k\geq 0$.
  Hence, the invariance of the joint distribution of the TEM numerical solution is inherited by its one-particle marginal distribution. 
For any $\mu^M\in \mathcal{P}^*_{p}((\RR^{d})^{M})$,  one deduces from   Lemma \ref{L7.3} and Theorem \ref{L7.4}  that for any $\Delta\in (0,\Delta_2^*]$,
\begin{align}\label{ee5.23}
 \mathbb{W}^{q}_{q}\Big(\mathbf{P}^{\Delta,M,*}_{t_k}\mu^{ M} , \mu^{\Delta,M,*}\Big) 
 &= \mathbb{W}^{q}_{q}\Big(\mathbf{P}^{\Delta,M,*}_{t_k}\mu^{ M} , \mathbf{P}^{\Delta,M,*}_{t_k}\mu^{\Delta,M,*}\Big)  \leq  C M^{\frac{q}{2}}  e^{-\big(\frac{q(\bar{\lambda}_1-\bar{\lambda}_2)}{2}
-\varepsilon\big)t_k},
\end{align} where $C$ is independent of $M$, $\Delta$ and $k$. Similarly,  the desired assertion \eqref{eq++3} follows from  Lemma \ref{L7.3} and Theorem \ref{L7.4}.
Letting $k\rightarrow\infty$ on both sides of \eqref{ee5.23} implies  that  $\mu^{\Delta,M,*}$ is the unique  invariant measure. 
 By the Skorohod representation theorem \cite[Theorem 3.30]{MR1876169}, there is a  probability space $ (\tilde{\Omega},~\tilde{\mathcal{F}},~\tilde{\mathbb{P}})$ and an $(\RR^{d})^{M}$-valued  random variable sequence $\{\eta_k\}_{k\geq 0}$, and a random variable 
$\eta$ defined on this probability space  with $\eta_k\sim \mathbf{P}^{\Delta,M,*}_{t_k}\boldsymbol{\delta}^{\otimes M}_0 ,~k=1,2,\cdots$, and $\eta\sim \mu^{\Delta,M,*} $  such that $\eta_k\rightarrow \eta$ a.s. as $k\rightarrow \infty$. Thus, using the Fatou lemma and Theorem \ref{L7.4}   yields  that  
\begin{align*}
 \mu^{\Delta,M,*} (|\cdot|^p)= \tilde{\E} |\eta|^p&\leq \liminf_{k\rightarrow \infty} \tilde{\E}|\eta_k|^p =\liminf_{k\rightarrow \infty} \E \Big(\sum_{i=1}^M |\bar{Y}^{i,M}_{t_{ k}}|^2\Big)^{\frac{p}{2}}\nn\
\\&\leq  M^{\frac{p}{2}} \sup_{1\leq i\leq M}\sup_{\Delta\in(0,\Delta_1^*]}\sup_{ k\geq0}  \mathbb{E}|\bar{Y}^{i,M}_{t_k}|^p\leq  C M^{\frac{p}{2}}<\infty , 
\end{align*}
where $\tilde{\E}(\cdot)  $ denotes the expectation with respect to $\tilde{\mathbb{P}}$, $\{\bar{Y}^{ M}_{t_k}\}$ is the TEM numerical solution with the initial distribution $\boldsymbol{\delta}^{\otimes M}_0$ and $C $ is independent of $M$, $\Delta$ and $k$. This implies that $\mu^{\Delta,M,*}\in \mathcal{P}^*_{p}((\RR^{d})^{M})$.
The proof is complete.
\end{proof}

By Theorem \ref{th7.2} we can obtain the following result on the  operator $\mathrm{P}^{\Delta,M,*}_{t_{k}}$ defined by \eqref{eq++4}.
\begin{cor} Under the conditions of Theorem \ref{th7.2}, for any $\mu_0\in \mathcal{P}_{p}(\RR^{d})$,  
\begin{align}\label{eq++2}  
  \mathbb{W}^{q}_{q}\big( \mathrm{P}^{\Delta,M,*}_{t_{k}}\mu_0 ,\Gamma_1(\mu^{\Delta,M,*})\big)  \leq C e^{-\big(\frac{q(\bar{\lambda}_1-\bar{\lambda}_2)}{2}-\varepsilon\big)t_{k}}.
 \end{align}
 \end{cor}

\subsection{The Uniform-in-time Convergence Rate}\label{sb6.3}

{This subsection focuses on analyzing the uniform-in-time convergence rate of the numerical solution generated by the TEM scheme  \eqref{eq2.1} to the exact solution of the MV-SDE \eqref{eq3.1} in the {\color{blue}$L^q$}-Wasserstein distance. To achieve this goal, we first establish the uniform-in-time propagation of chaos result.  
We then obtain the uniform-in-time convergence error between  the TEM numerical solution generated by   \eqref{eq2.1} and the exact solution of the IPS. Thus, the desired uniform-in-time convergence rate follows.

\begin{theorem}[Uniform-in-time propagation of chaos]\label{thmps}
 Let Assumptions \ref{ass1}, \ref{ass5} and \ref{ass6} hold with $p>2$ and  $X_0\in L^{p}_{0}$. Then for any $q\in [2, p_0]\cap[2,p)$,  there is a constant $C(=C_{q, p, d})$ such that
\begin{align*}
\Big(\sup_{1\leq i\leq M}\sup_{t\geq 0}\mathbb{E}\mathbb{W}^{q}_{q}\big(\mathcal{L}^{X^i}_{t},\mathcal{L}^{X,M}_{t}\big)\Big)\vee\Big(\sup_{1\leq i\leq M}\sup_{t\geq 0}\mathbb{E}\big|X^{i}_{t}-X^{i,M}_{t}\big|^{q}\Big)\leq C \Upsilon_{M,q,p,d},
\end{align*}
where $\Upsilon_{M,q,p,d}$ is defined in Lemma \ref{Lem3.1}.
\end{theorem}
\begin{proof}
 For convenience, define  $
 \displaystyle\varsigma =\frac{q(\bar{\lambda}_1-\bar{\lambda}_2)}{4}>0,~  U^{i,M}_{t}=X^{i}_{t}-X^{i,M}_{t}, ~t\geq 0.$ Using the It\^o formula, one derives from \eqref{eq2} and \eqref{eq3} that for any $t\geq 0$, 
\begin{align*}
\mathrm{d}\Big(e^{\va t}|U^{i,M}_{t}|^{q}\Big)&=\va e^{\va t}|U^{i,M}_{t}|^{q}+\frac{q}{2}e^{\va t} |U^{i,M}_{t}|^{q-2}\Big[2(U^{i,M}_{t})^{T}\big(f(X^{i}_{t},\mathcal{L}^{X^i}_{t})-f(X^{i,M}_{t},\mathcal{L}^{X,M}_{t})\big)
\\&~~~+(q-1)\big|g(X^{i}_{t},\mathcal{L}^{X^i}_{t})-g(X^{i,M}_{t},\mathcal{L}^{X,M}_{t})\big|^2\Big]\mathrm{d}t
 \nn\\
& +q e^{\va t}|U^{i,M}_{t}|^{q-2}(U^{i,M}_{t})^{T}\big(g(X^{i}_{t},\mathcal{L}^{X^{i}}_{t})-g(X^{i,M}_{t},
\mathcal{L}^{X,M}_{t})\big)\mathrm{d}B^{i}_{t}.\end{align*}
 Thanks to $q\leq p_0$, integrating the above equation from $0$ to $t$ and taking the expectations on both sides, by Assumption \ref{ass6}, we obtain   
\begin{align}\label{eqq6.39}
\mathbb{E}\Big(e^{\va t}|U^{i,M}_{t}|^{q}\Big)&\leq -\Big(\frac{q\bar{\lambda}_1}{2}-\va\Big) \int_{0}^{t}  e^{\va s}\mathbb{E}|U^{i,M}_{s}|^{q}\mathrm{d}s + \frac{q\bar{\lambda}_2}{2}\mathbb{E}\int_{0}^{t}e^{\va s} |U^{i,M}_{s}|^{q-2} \mathbb{W}_{2}^{2}(\mathcal{L}^{X^i}_{s},\mathcal{L}^{X,M}_{s}) \mathrm{d}s .
\end{align} 
Define  the empirical measure of the {N-IPS} \eqref{eq3} by $ \mu^{X,M}_{t}:=\frac{1}{M}\sum_{i=1}^{M}\boldsymbol{\delta}_{X^i_t}.$ For any  $\varepsilon >0$ (defined later),  using the elementary inequality yields 
\begin{align*}
\mathbb{W}^{2}_{2}(\mathcal{L}^{X^i}_s,\mathcal{L}^{X,M}_{s})&\leq\Big[ (1+\varepsilon )\mathbb{W}^2_{2}(\mu^{X,M}_{s},\mathcal{L}^{X,M}_{s})
+\Big(1+\frac{1}{\varepsilon }\Big)\mathbb{W}^{2}_{2}(
\mathcal{L}^{X^{i}}_s,\mu^{X,M}_{s}) \Big].
\end{align*}
Inserting the above inequality into \eqref{eqq6.39} gives that 
\begin{align}\label{eq6.34}
\mathbb{E}\Big(e^{\va t}|U^{i,M}_{t}|^{q}\Big)&\leq -\Big(\frac{q\bar{\lambda}_1}{2}-\va\Big) \mathbb{E}\int_{0}^{t}  e^{\va s}|U^{i,M}_{s}|^{q}\mathrm{d}s  +\mathcal{A}_1+\mathcal{A}_2, 
\end{align}
where 
\begin{align*}
\mathcal{A}_1&=\frac{q\bar{\lambda}_2(1+\varepsilon )}{2}\int_{0}^{t}e^{\va s}\mathbb{E}\Big(|U^{i,M}_{s}|^{q-2}\mathbb{W}^2_{2}(\mu^{X,M}_{s},\mathcal{L}^{X,M}_{s})\Big)\mathrm{d}s,
\\
\mathcal{A}_2&=\frac{q\bar{\lambda}_2}{2}\Big(1+\frac{1}{\varepsilon }\Big)\int_{0}^{t}e^{\va s}\mathbb{E}\Big(|U^{i,M}_{s}|^{q-2}\mathbb{W}^{2}_{2}(\mathcal{L}^{X^{i}}_s,\mu^{X,M}_{s})\Big)\mathrm{d}s.
\end{align*}
One notices that $ \mathbb{W}^{q}_{2}(\mu^{X,M}_{s},\mathcal{L}^{X,M}_{s})\leq \frac{1}{M}\sum_{i=1}^{M}\mathbb{E}|U^{i,M}_{s}|^{q}$. 
Then using the Young inequality derives that
\begin{align}\label{eq6.35}
\mathcal{A}_1&\leq \frac{q\bar{\lambda}_2(1+\varepsilon )}{2}\int_{0}^{t}e^{\va s}\Big(\frac{q-2}{q}\mathbb{E}|U^{i,M}_{s}|^{q}+\frac{2}{q}\mathbb{E}\mathbb{W}^q_{2}(\mu^{X,M}_{s},
\mathcal{L}^{X,M}_{s})\Big)\mathrm{d}s \nn\\
& \leq \frac{q\bar{\lambda}_2(1+\varepsilon )}{2}\int_{0}^{t}e^{\va s}\Big(\frac{q-2}{q}\mathbb{E}|U^{i,M}_{s}|^{q}+{ \frac{2}{qM}\sum_{i=1}^{M}
\mathbb{E}|U^{i,M}_{s}|^{q}}\Big)\mathrm{d}s .
\end{align}
Utilizing the Young inequality again gives that for any $\varepsilon >0$,
\begin{align}\label{eq6.36}
\mathcal{A}_2\leq\frac{q\bar{\lambda}_2\varepsilon }{2}\int_{0}^{t}e^{\va s}\mathbb{E}|U^{i,M}_{s}|^{q}\mathrm{d}s+\frac{q\bar{\lambda}_2C_{\varepsilon }}{2}\int_{0}^{t}e^{\va s}\mathbb{E}\mathbb{W}^{q}_{2}\big(\mathcal{L}^{X^i}_s,\mu^{X,M}_{s}\big)\mathrm{d}s,
\end{align}
where $C_{\varepsilon }=\varepsilon^{-q+1}(1+ {\varepsilon })^{\frac{q}{2}}$. Due to $\bar{\lambda}_1-\bar{\lambda}_2>0$,  choose $  \varepsilon =  ({\bar{\lambda}_1-\bar{\lambda}_2})/{(4\bar{\lambda}_2)}.$ 
Substituting \eqref{eq6.35} and \eqref{eq6.36}  into \eqref{eq6.34} leads to  
\begin{align}\label{eq++8}
\mathbb{E}\Big(e^{\va t}|U^{i,M}_{t}|^{q}\Big) &\leq    \frac{ \bar{\lambda}_1+3\bar{\lambda}_2}{4}\int_{0}^{t}e^{\va s}\Big( -\mathbb{E}|U^{i,M}_{s}|^{q}+{ \frac{1}{ M}\sum_{i=1}^{M}
\mathbb{E}|U^{i,M}_{s}|^{q}}\Big)\mathrm{d}s+C\int_{0}^{t}e^{\va s}\mathbb{E}\mathbb{W}^{q}_{2}\big(\mathcal{L}^{X^i}_s,\mu^{X,M}_{s}\big)\mathrm{d}s. 
\end{align}
Taking the particle average on both sides of the above inequality, due to $q< p$, applying Lemma  \ref{Lem3.1}  with $\tilde{q}=p$ and Lemma \ref{Le5.1} we derive that
{ \begin{align*}
\frac{1}{ M}\sum_{i=1}^{M} \mathbb{E}\Big(e^{\va t}|U^{i,M}_{t}|^{q}\Big)&\leq    C_{q, p, d}  \sup_{t\geq 0} \big(\mathbb{E}|X^{i}_{t}|^{p}\big)^{\frac{q}{p}} \Upsilon_{M,q,p,d}e^{\va t}\leq C_{q, p, d} \Upsilon_{M,q,p,d}e^{\va t},  
\end{align*}
  which implies 
$
\sup_{t\geq0}{ \frac{1}{ M}\sum_{i=1}^{M}}\mathbb{E}|U^{i,M}_{t}|^{q} \leq C_{q, p, d}\Upsilon_{M,q,p,d},
$
where $C_{q, p, d}$ is independent of $t$.  Inserting this inequality into \eqref{eq++8}, applying Lemma  \ref{Lem3.1}  with $\tilde{q}=p$ and Lemma \ref{Le5.1} again we obtain  the  desired assertion. The proof is complete.}
\end{proof}

Next, we prove the convergence rate of the numerical solution generated by the TEM scheme \eqref{eq2.1}  to the exact solution of the IPS.

\begin{theorem}\label{L+1}
Let Assumptions \ref{ass1}, \ref{ass4}, \ref{ass5}, \ref{ass6}  hold with $p\geq (6\alpha+2)$, $p_0>2$ and $X_0\in L^{p}_{0}$. Then for any   $q\in [2,p_0\wedge ((2p-6\alpha)/(3\alpha+2)))\cap[2,p/((\alpha+2)\vee2\alpha)]$, the TEM numerical solutions defined by \eqref{eq2.1} with $\kappa\in [\alpha(q+2)/2(p-q),1/3)$ satisfy
\begin{align*}
\sup_{  M\geq 1}\sup_{1\leq i\leq M}\sup_{k\geq 0}\Big(\mathbb{E}\big|Y^{i,M}_{t_{k}}-X^{i,M}_{t_{k}}\big|^{q}\vee \mathbb{E}\big|\bar{Y}^{i,M}_{t_{k}}-X^{i,M}_{t_{k}}\big|^{q}\Big)\leq C\Delta^{\frac{q}{2}},~~\forall \Delta\in (0,\Delta_1^*],
\end{align*}
where $C$ is a positive constant independent of    $k$, $M$ and $\Delta$.
\end{theorem}
\begin{proof} Fix a  $q\in [2,p_0\wedge ((2p-6\alpha)/(3\alpha+2)))\cap[2,p/((\alpha+2)\vee2\alpha)]$ and a $\kappa\in [\alpha(q+2)/2(p-q),1/3)$.  Recalling  $\va=\frac{q(\bar{\lambda}_1-\bar{\lambda}_2)}{4}>0$ in Theorem \ref{thmps}, using the It\^o formula, we obtain  from \eqref{eq2} and \eqref{eq4.35} that for any $t\in [t_{k},t_{k+1})$,
\begin{align*}
e^{\va t} \mathbb{E}&\big|\tilde{Y}^{i,M}_{t}-X^{i,M}_{t}\big|^{q}  =e^{ \va t _{k}}\mathbb{E}\big|Y^{i,M}_{t_{k}}-X^{i,M}_{t_{k}}\big|^q+\va\int_{t_{k}}^{t}e^{ \va s }\mathbb{E}|\tilde{Y}^{i,M}_{s}-X^{i,M}_{s}|^{q}\mathrm{d}s\nn\\
&~~~
+\frac{q}{2}\mathbb{E}\int_{t_{k}}^{t}e^{\va s}\big|\tilde{Y}^{i,M}_{s}-X^{i,M}_{s}\big|^{q-2} 
 \Big[2(\tilde{Y}^{i,M}_{s}-X^{i,M}_{s})^{T}
\big(f(Y^{i,M}_{s},\mathcal{L}^{Y,M}_{s})-f(X^{i,M}_{s},\mathcal{L}^{X,M}_{s})\big)\nn\
\\&~~~~~~~~~~~~~~~~~~~~~~~~~~~~~~~~~~~~~~~~~~~~~+(q-1)\big|g(Y^{i,M}_{s},\mathcal{L}^{Y,M}_{s})-g(X^{i,M}_{s},\mathcal{L}^{X,M}_{s})\big|^2\Big]\mathrm{d}s.
\end{align*}
Then using the Young inequality and Assumption \ref{ass6} yields  
\begin{align}\label{e5.60}
e^{ \va t }\mathbb{E}\big|\tilde{Y}^{i,M}_{t}-X^{i,M}_{t}\big|^{q} 
 &\leq e^{ \va t _{k}}\mathbb{E}\big|Y^{i,M}_{t_{k}}-X^{i,M}_{t_{k}}\big|^q+\Big(\va-\frac{q\bar{\lambda}_1}{2}\Big)\int_{t_{k}}^{t}e^{ \va s }\mathbb{E}\big|\tilde{Y}^{i,M}_{s}-X^{i,M}_{s}\big|^{q}\mathrm{d}s\nn\
\\& +\frac{q\bar{\lambda}_2}{2}\int_{t_{k}}^{t}e^{ \va s }\mathbb{E}\Big(\big|\tilde{Y}^{i,M}_{s}-X^{i,M}_{s}\big|^{q-2}\mathbb{W}^{2}_{2}(\mathcal{L}^{\tilde{Y},M}_{s},\mathcal{L}^{X,M}_{s})\Big)\mathrm{d}s+J_1+J_2.
\end{align}
where 
\begin{align*}
J_1&=q\mathbb{E}\int_{t_{k}}^{t}e^{ \va s }\big|\tilde{Y}^{i,M}_{s}-X^{i,M}_{s}\big|^{q-1}\big|f(Y^{i,M}_{s},\mathcal{L}^{Y,M}_{s})-f(\tilde{Y}^{i,M}_{s},\mathcal{L}^{\tilde{Y},M}_{s})\big|\mathrm{d}s
 \\
J_2&=\frac{q}{2}\Big(q-1+\frac{1}{p_0-q}\Big)\mathbb{E}\int_{t_{k}}^{t}e^{ \va s }\big|\tilde{Y}^{i,M}_{s}-X^{i,M}_{s}\big|^{q-2}\big|g(Y^{i,M}_{s},\mathcal{L}^{Y,M}_{s})-g(\tilde{Y}^{i,M}_{s},\mathcal{L}^{\tilde{Y},M}_{s})\big|^2\mathrm{d}s.
\end{align*}
Using the Young inequality, the H\"older inequality and  the identical distribution property of $\tilde{Y}^{i,M}_{t}-X^{i,M}_{t}$, $i=1,2,\cdots M$, we obtain that
\begin{align}\label{eq*+35}
 \mathbb{E}\Big(\big|\tilde{Y}^{i,M}_{s}-X^{i,M}_{s}\big|^{q-2}\mathbb{W}^{2}_{2}(\mathcal{L}^{\tilde{Y},M}_{s},\mathcal{L}^{X,M}_{s})\Big)\nn\
 &\leq  \frac{q-2}{q}\mathbb{E}\big|\tilde{Y}^{i,M}_{s}-X^{i,M}_{s}\big|^{q}+\frac{2}{q}\mathbb{E}\mathbb{W}^{q}_{2}\big(\mathcal{L}^{\tilde{Y},M}_{s},\mathcal{L}^{X,M}_{s}\big)\nn\
 \nn\\
 &= \frac{q-2}{q}\mathbb{E}\big|\tilde{Y}^{i,M}_{s}-X^{i,M}_{s}\big|^{q}+{ \frac{2}{ qM}\sum_{i=1}^{M}}\mathbb{E}\big|\tilde{Y}^{i,M}_{s}-X^{i,M}_{s}\big|^{q}.
\end{align}
 Using the Young inequality yields that 
\begin{align}\label{eee5.59}
J_1+J_2&\leq \va\int_{t_{k}}^{t}e^{ \va s }\mathbb{E}\big|\tilde{Y}^{i,M}_{s}-X^{i,M}_{s}\big|^q\mathrm{d}s+C\int_{t_{k}}^{t}e^{ \va s }\mathbb{E}\big|f(Y^{i,M}_{s},\mathcal{L}^{Y,M}_{s})-f(\tilde{Y}^{i,M}_{s},\mathcal{L}^{\tilde{Y},M}_{s})\big|^q\mathrm{d}s\nn\
\\&~~~+C\int_{t_{k}}^{t}e^{ \va s }\mathbb{E}\big|g(Y^{i,M}_{s},\mathcal{L}^{Y,M}_{s})-g(\tilde{Y}^{i,M}_{s},\mathcal{L}^{\tilde{Y},M}_{s})\big|^q\mathrm{d}s.
\end{align}
Thanks to $2\leq q\leq p/((\alpha+2)\vee(2\alpha))$,
by the similar arguments as \eqref{eq*+16}-\eqref{eq*+37}, using Assumption \ref{ass4}, Theorem \ref{L7.4} and Corollary \ref{cor+1},  we derive that
\begin{align*}
\mathbb{E}\big|f(Y^{i,M}_{s},\mathcal{L}^{Y,M}_{s})-f(\tilde{Y}^{i,M}_{s},\mathcal{L}^{\tilde{Y},M}_{s})\big|^q\vee\mathbb{E}\big|g(Y^{i,M}_{s},\mathcal{L}^{Y,M}_{s})-g(\tilde{Y}^{i,M}_{s},\mathcal{L}^{\tilde{Y},M}_{s})\big|^q\leq C\Delta^{\frac{q}{2}}.
\end{align*}
Inserting the above inequality into \eqref{eee5.59} yields that
\begin{align}\label{e5.61}
J_1+J_2\leq \va\int_{t_{k}}^{t}e^{ \va s }\mathbb{E}|\tilde{Y}^{i,M}_{s}-X^{i,M}_{s}|^q\mathrm{d}s+
C\Delta^{\frac{q}{2} } \int_{t_{k}}^{t}e^{ \va s }\mathrm{d}s.
\end{align}
Then combining  \eqref{eq*+35} and \eqref{e5.61} with \eqref{e5.60}, recalling the definition of $\va$  yields
\begin{align}
e^{ \va t }\mathbb{E}\big|\tilde{Y}^{i,M}_{t}-X^{i,M}_{t}\big|^{q}& 
 \leq e^{ \va t _{k}}\mathbb{E}\big|Y^{i,M}_{t_{k}}-X^{i,M}_{t_{k}}\big|^q-{
\bar{\lambda}_2 \int_{t_{k}}^{t}e^{ \va s }\mathbb{E}\big|\tilde{Y}^{i,M}_{s}-X^{i,M}_{s}\big|^{q}\mathrm{d}s }\nn\\
 &~~~~ +{ \bar{\lambda}_2\int_{t_{k}}^{t}e^{ \va s }\frac{1}{ M}\sum_{i=1}^{M}}\mathbb{E}\big|\tilde{Y}^{i,M}_{s}-X^{i,M}_{s}\big|^{q}\mathrm{d}s +C\Delta^{\frac{q}{2} } \int_{t_{k}}^{t}e^{ \va s }\mathrm{d}s.\label{eq++9}
\end{align}
Then, dividing both sides of \eqref{eq++9}  by $e^{ \va t }$   and then taking the particle average yields \begin{align}\label{eq++10}{ \frac{1}{ M}\sum_{i=1}^{M}}\mathbb{E}\big|\tilde{Y}^{i,M}_{t}-X^{i,M}_{t}\big|^{q} \leq { \frac{1}{ M}\sum_{i=1}^{M}} \mathbb{E}\big|Y^{i,M}_{t_{k}}-X^{i,M}_{t_{k}}\big|^qe^{-\va (t-t_{k}) }+C\Delta^{\frac{q}{2}+1}.\end{align}
Using the Fatou Lemma and the inequality $|\pi_{\Delta}(x)-\pi_{\Delta}(y)|\leq |x-y|$,  $\forall x,y\in \RR^{d}$, we derive from  \eqref{eqcc3.37} that
\begin{align}\label{eee5.62}
 \mathbb{E}\big|Y^{i,M}_{t_{k+1}}-\pi_{\Delta}(X^{i,M}_{t_{k+1}})\big|^q
   &\leq  \mathbb{E}\big|\bar{Y}^{i,M}_{t_{k+1}}-X^{i,M}_{t_{k+1}}\big|^{q} 
 \leq \liminf_{t\rightarrow t_{k+1}^-} \mathbb{E}\big|\tilde{Y}^{i,M}_{t}-X^{i,M}_{t}\big|^{q}.\end{align}
{ Taking particle average on both sides of \eqref{eee5.62}, and then combining it with  } \eqref{eq++10} yields that \begin{align}\label{eq++12}{ \frac{1}{ M}\sum_{i=1}^{M}}\mathbb{E}\big|Y^{i,M}_{t_{k+1}}-\pi_{\Delta}(X^{i,M}_{t_{k+1}})\big|^q
   \leq{ \frac{1}{ M}\sum_{i=1}^{M}} \mathbb{E}\big|Y^{i,M}_{t_{k}}-X^{i,M}_{t_{k}}\big|^qe^{-\va\Delta} +C\Delta^{\frac{q}{2}+1}.
\end{align}
For any $\Delta\in (0,1]$, define a set $
\Upsilon^{i,M}_{k,\Delta}:=\big\{\omega:|X^{i,M}_{t_{k}}|\geq \varphi^{-1}(H\Delta^{-\kappa})\big\}.
$
Note that $\pi_{\Delta}(X^{i,M}_{t_{k}})=X^{i,M}_{t_{k}}$ for any $\omega\in (\Upsilon^{i,M}_{k,\Delta})^{c}$. Utilizing the Young inequality and the Chebyshev inequality, for any $\delta>0$, we obtain from  Lemma \ref{Le5.1} and Theorem \ref{L7.4} that 
\begin{align}\label{eq5.62*}
\mathbb{E}\big|Y^{i,M}_{t_{k}}-X^{i,M}_{t_{k}}\big|^q&
=\mathbb{E}\Big(\big|Y^{i,M}_{t_{k}}- X^{i,M}_{t_{k}} \big|^qI_{\big(\Upsilon^{i, M}_{k,\Delta}\big)^c}\Big) +\mathbb{E}\Big(\big|Y^{i,M}_{t_{k}}-X^{i,M}_{t_{k}}\big|^qI_{\Upsilon^{i, M}_{k,\Delta}}\Big)\nn\
\\&\leq \mathbb{E}\big|Y^{i,M}_{t_{k}}-\pi_{\Delta}(X^{i,M}_{t_{k}})\big|^q+\mathbb{E}\Big(\big|Y^{i,M}_{t_{k}}-X^{i,M}_{t_{k}}\big|^qI_{\Upsilon^{i, M}_{k,\Delta}}\Big).
\end{align}
Using the H\"older inequality and the Chebyshev inequality,  we obtain that
\begin{align*}
\mathbb{E}\Big(\big|Y^{i,M}_{t_{k}}-X^{i,M}_{t_{k}}\big|^qI_{\Upsilon^{i, M}_{k,\Delta}}\Big)&\leq \big(\mathbb{E}|Y^{i,M}_{t_{k}}-X^{i,M}_{t_{k}}\big|^p\Big)^{\frac{q}{p}}\big(\mathbb{P}(\Upsilon^{i, M}_{k,\Delta})\big)^{\frac{p-q}{p}}\leq \frac{2^q \sup\limits_{k\geq 0}\Big(\E |Y^{i,M}_{t_{k}}|^p \vee \E |X^{i,M}_{t_{k}}|^p\Big)} {\big(\varphi^{-1}(H\Delta^{-\kappa})\big)^{p-q}}.
\end{align*}
Thanks to $p\geq 6\alpha+2$ and $\kappa \in [\alpha(q+2)/(2(p-q)), 1/3)$,  by virtue of Lemma \ref{Le5.1} and Theorem \ref{L7.4},  recalling the definition of $\varphi^{-1}$ in Remark \ref{rem5.2},  we yield 
\begin{align*}
 \mathbb{E}\Big(\big|Y^{i,M}_{t_{k}}-X^{i,M}_{t_{k}}\big|^qI_{\Upsilon^{i, M}_{k,\Delta}}\Big)&\leq {C}{\Big(\frac{H\Delta^{-\kappa}}{2(L_f \vee L_g)}-1\Big)^{-\frac{p-q}{\alpha}}}\leq C\Delta^{\frac{\kappa(p-q)}{\alpha}}\leq C\Delta^{\frac{q}{2}+1}.
\end{align*}
Then inserting the above inequality into \eqref{eq5.62*} yields that
\begin{align} \label{eee5.64*}
\mathbb{E}\big|Y^{i,M}_{t_{k}}-X^{i,M}_{t_{k}}\big|^q&\leq \mathbb{E}\big|Y^{i,M}_{t_{k}}-\pi_{\Delta}(X^{i,M}_{t_{k}})\big|^q+C\Delta^{\frac{q}{2}+1}.
\end{align}
Substituting the above inequality  into \eqref{eq++12}  gives  
\begin{align*}
{ \frac{1}{ M}\sum_{i=1}^{M}}&\mathbb{E}\big|Y^{i,M}_{t_{k+1}}-\pi_{\Delta}(X^{i,M}_{t_{k+1}})\big|^q\leq { \frac{1}{ M}\sum_{i=1}^{M}}\mathbb{E}\big|Y^{i,M}_{t_{k}}-\pi_{\Delta}(X^{i,M}_{t_{k}})\big|^{q}
e^{-\va \Delta} +C\Delta^{\frac{q}{2}+1} \leq \cdots
\\\leq& { \frac{1}{ M}\sum_{i=1}^{M}}\mathbb{E}\big|X^{i}_0-\pi_{\Delta}(X^{i}_{0})\big|^q e^{-\va t_{k+1}} +C\Delta^{\frac{q}{2}+1}\sum_{n=0}^{k}e^{-\va n \Delta} \leq { \frac{1}{ M}\sum_{i=1}^{M}}\mathbb{E}\big|X^{i}_0-\pi_{\Delta}(X^{i}_{0})\big|^q+ C\Delta^{\frac{q}{2} }  .
\end{align*}
Due to $\kappa \geq q\alpha/(2(p-q))$,  $\kappa (p-q)/\alpha\geq q/2$.  It follows from Lemma \ref{L6.14} that
\begin{align*}
{ \frac{1}{ M}\sum_{i=1}^{M}}\mathbb{E}\big|Y^{i,M}_{t_{k+1}}-\pi_{\Delta}(X^{i,M}_{t_{k+1}})\big|^q\leq   C\Delta^{\frac{q}{2}},
\end{align*}
where  $C$ is independent of $k$, $\Delta$ and $M$. Inserting the above inequality into \eqref{eee5.64*} derives \begin{align*}
{ \frac{1}{ M}\sum_{i=1}^{M}}\mathbb{E}\big|Y^{i,M}_{t_{k }}- X^{i,M}_{t_{k }} \big|^q\leq   C\Delta^{\frac{q}{2}}.
\end{align*} 
{  Inserting the above inequality into \eqref{eq++9} and then dividing both sides by $e^{ \va t }$  implies 
 $$ \mathbb{E}\big|\tilde{Y}^{i,M}_{t}-X^{i,M}_{t}\big|^{q} \leq  \mathbb{E}\big|Y^{i,M}_{t_{k}}-X^{i,M}_{t_{k}}\big|^qe^{-\va (t-t_{k}) }+C\Delta^{\frac{q}{2}+1}.$$
Similarly, using \eqref{eee5.62} and \eqref{eee5.64*}, we obtain 
\begin{align*}
& \mathbb{E}\big|Y^{i,M}_{t_{k+1}}-\pi_{\Delta}(X^{i,M}_{t_{k+1}})\big|^q\leq \mathbb{E}\big|Y^{i,M}_{t_{k}}-\pi_{\Delta}(X^{i,M}_{t_{k}})\big|^{q}
e^{-\va \Delta} +C\Delta^{\frac{q}{2}+1}.
\end{align*}
Solving the above difference equation and using \eqref{eee5.64*} we yield 
\begin{align*}
\sup_{  M\geq 1}\sup_{1\leq i\leq M}\sup_{k\geq 0}\mathbb{E}\big|Y^{i,M}_{t_{k}}-X^{i,M}_{t_{k}}\big|^{q}\leq C\Delta^{\frac{q}{2}}.
\end{align*}}
Furthermore, by a similar argument  as in Theorem \ref{th5.5}, using Lemma \ref{L6.14} and Theorem \ref{L7.4} we can obtain another desired assertion.
The proof is complete.
\end{proof}

Combining the results of Theorems \ref{thmps} and \ref{L+1}, we obtain the uniform-in-time strong convergence rate of the one-particle numerical solution of the IPS to the exact solution of the MV-SDE.
\begin{theorem}\label{Th5.16*} Under the conditions of Theorem \ref{L+1},
  for any   $q\in [2,p_0\wedge ((2p-6\alpha)/(3\alpha+2)))\cap[2,p/((\alpha+2)\vee2\alpha)]$ and any $ \Delta\in (0,\Delta_1^*]$,  the TEM numerical solutions with $\kappa\in [\alpha(q+2)/2(p-q),1/3)$ satisfy that
\begin{align*}
\sup_{1\leq i\leq M}\sup_{k\geq0}\Big(\mathbb{E}|X^{i}_{t_{k}}-Y^{i,M}_{t_k}|^q\vee\mathbb{E}|X^{i}_{t_{k}}-\bar{Y}^{i,M}_{t_k}|^q\Big)\leq C \Upsilon_{M,q,p,d}+C\Delta^{\frac{q}{2}},~~~
\end{align*}
where the constant $C$ is independent of $k$, $M$ and $\Delta$.
\end{theorem}

By virtue of Theorem \ref{Th5.16*}, we obtain the convergence rate of the one-particle marginal distribution of the numerical invariant measure to the invariant measure of the MV-SDE in the $L^q$-Wasserstein distance.

\begin{theorem}\label{Th5.21} Under the conditions of Theorem \ref{L+1}, for  any   $q\in [2,p_0\wedge ((2p-6\alpha)/(3\alpha+2)))\cap[2,p/((\alpha+2)\vee2\alpha)]$ and any $ \Delta\in (0,\Delta_1^*\wedge \Delta_2^*]$,  there exists a unique   numerical invariant measure  $\mu^{\Delta,M,*}\in \mathcal{P}^*_{p}((\RR^{d})^{M})$ with  $\kappa\in [\alpha(q+2)/2(p-q),1/3)$ satisfying
$$
\mathbb{W}^{q}_{q}(\Gamma_1(\mu^{\Delta,M,*}),\mu^*)\leq
C \Upsilon_{M,q,p,d}+C\Delta^{\frac{q}{2}},
$$
where the constant $C$ is independent of  $M$ and $\Delta$.
\end{theorem}

\section{Examples}\label{num}
In this section,  we give several examples of superlinear MV-SDEs  and carry out some numerical experiments to check 
the effectiveness of the TEM scheme given by \eqref{eq2.1}.   Thanks to the propagation of chaos  Lemma \ref{lec3.4},  we only predict the convergence error between the numerical solutions of the TEM scheme and the exact solutions of the IPS  corresponding to \eqref{eq3.1}.  For any $T\geq0$, due to the identical distribution property of $X^{i,M}_{T}-\bar{Y}^{i,M}_{T}$ for  $  i=1,\cdots, M$, we have
\begin{align*}
\mathbb{E}\big|X^{i,M}_{T}-\bar{Y}^{i,M}_{T}\big|^{2}=\mathbb{E}\Big(\frac{1}{M}\sum_{i=1}^{M}\big|X^{i,M}_{T}-\bar{Y}^{i,M}_{T}\big|^{2}\Big).
\end{align*}
For the given sample number $N$, define the root mean square error (RMSE) at time  $T$  as
\begin{align*}
RMSE:= \bigg(\mathbb{E}\Big(\frac{1}{M}\sum_{i=1}^{M}\big|X^{i,M}_{T}-\bar{Y}^{i,M}_{T}\big|^{2}\Big) \bigg)^{\frac{1}{2} }\approx \bigg(\frac{1}{MN}\sum_{j=1}^{N}\sum_{i=1}^{M}\big|X^{i,M,(j)}_{T}-\bar{Y}^{i,M,(j)}_{T}\big|^{2}\bigg)^{\frac{1}{2}},
\end{align*}
where $X^{i,M,(j)}_{T}$ and $\bar{Y}^{i,M,(j)}_{T}$ represent  the $j$th independent copies of $X^{i,M}_{T}$ and $\bar{Y}^{i,M}_{T}$, respectively, $  j=1, \cdots, N$,
generated from the same Brownian motion. Since the closed form of the exact solution  of the IPS is unknown, we regard the numerical solution with the smaller step size $\Delta=2^{-16}$ as the exact solution in the estimation of the strong convergence rate.

\begin{expl}\label{expl7.1} {\rm 
{\bf \underline{Case 1 ($d=1$).}~ }~~ Recall the MV-SDE \eqref{Ne4} with
$$f(x,\mu)=x(-2-|x|)+\int x\mu(\mathrm{d}x), ~~~g(x,\mu)=|x|^{3/2}/2.$$ 
 Consider the corresponding IPS   
\begin{align}\label{Ne1}
\mathrm{d}X^{i,M}_{t}=X^{i,M}_{t}\Big(-2-|X^{i,M}_{t}|+\frac{1}{M}\sum_{i=1}^{M}X^{i,M}_{t}\Big)\mathrm{d}t
+\frac{1}{2}|X^{i,M}_{t}|^{\frac{3}{2}}\mathrm{d}B^{i}_{t}, ~~~~i=1,\cdots, M.
\end{align}
One observes that  Assumptions \ref{ass1}-\ref{ass6} hold and $f(0,\boldsymbol{\delta}_0)=g(0,\boldsymbol{\delta}_0)=0$. 
According to Corollary \ref{th6.2},  the exact solutions of MV-SDE and the corresponding IPS  are exponentially stable in  $L^{2}$ sense, namely,
$
\mathbb{E}|X_{t}|^2\leq \mathbb{E}|X_0|^2e^{-2t},~~ 
\mathbb{E}|X^{i,M}_{t}|^2\leq \mathbb{E}|X_0|^2e^{-2t},~ \forall t\geq 0.
$
    Figure \ref{FN1} captures the explosion behavior of the EM numerical solution of  \eqref{Ne1}. Furthermore, we give a  rigorous proof that  the EM numerical solution does not  keep the exponentially stable in mean-square for MV-SDE \eqref{Ne1}, see Lemma \ref{AL2}  for details.

Next, we use the TEM scheme \eqref{eq7.3} to carry out the numerical experiments.  
Let initial values $X^{i}_0$, $i=1,\cdots, M$,  obey a standard normal distribution $\mathcal{N}(0,1)$ independently. Take $T=1$, $N=10$,  $M=5000$, and $\Delta=2^{q}$, where $q\in\{-10,-11,-12,-13,-14\}$. Figure \ref{figure1} plots  the $\log_2 (RMSE)$ 
  as a function of $q$.  
  Figure \ref{figure1} shows that the TEM scheme possesses a   convergence rate of {\color{red}order $1/2$} with respect to time step size $\Delta$, which is consistent with our theoretical result.  Figure \ref{3} plots the  sample paths of the numerical solutions  of the TEM scheme for the initial value $\bar{Y}^{i,M}_0=18$, $\Delta=0.05$ and $M=2000$.   Comparing Figure \ref{3} with  Figure \ref{FN1},   one observes   that  the ``truncation device" in the TEM scheme suppresses the ``particle corruption''  arising in the EM iteration process successfully and enables the TEM numerical solution to realize the underlying exponential stability.
\begin{figure}[H]
  \centering
\includegraphics[width=12cm,height=4cm]{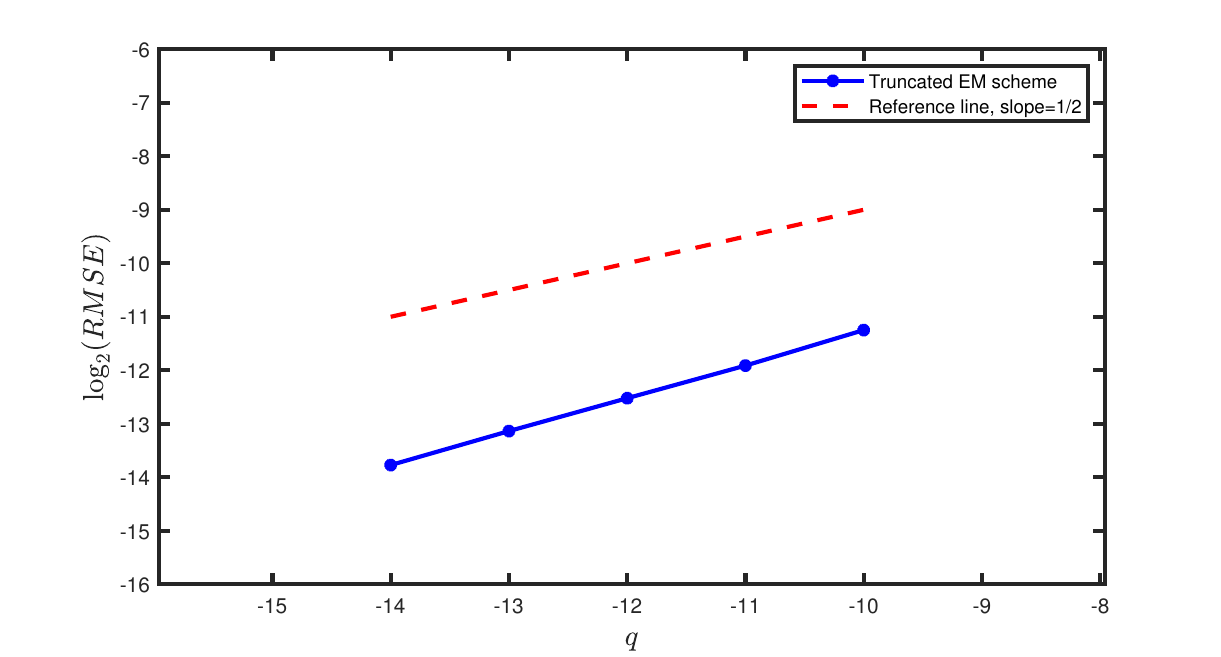}
\captionsetup{font=footnotesize}\vspace{-0.8em}
  \caption{The numerical error v.s. time step size $\Delta$ at $t=1$.}
\label{figure1}
\end{figure}\vspace{-2em}
\begin{figure}[H]
  \centering
\includegraphics[width=12cm,height=4cm]{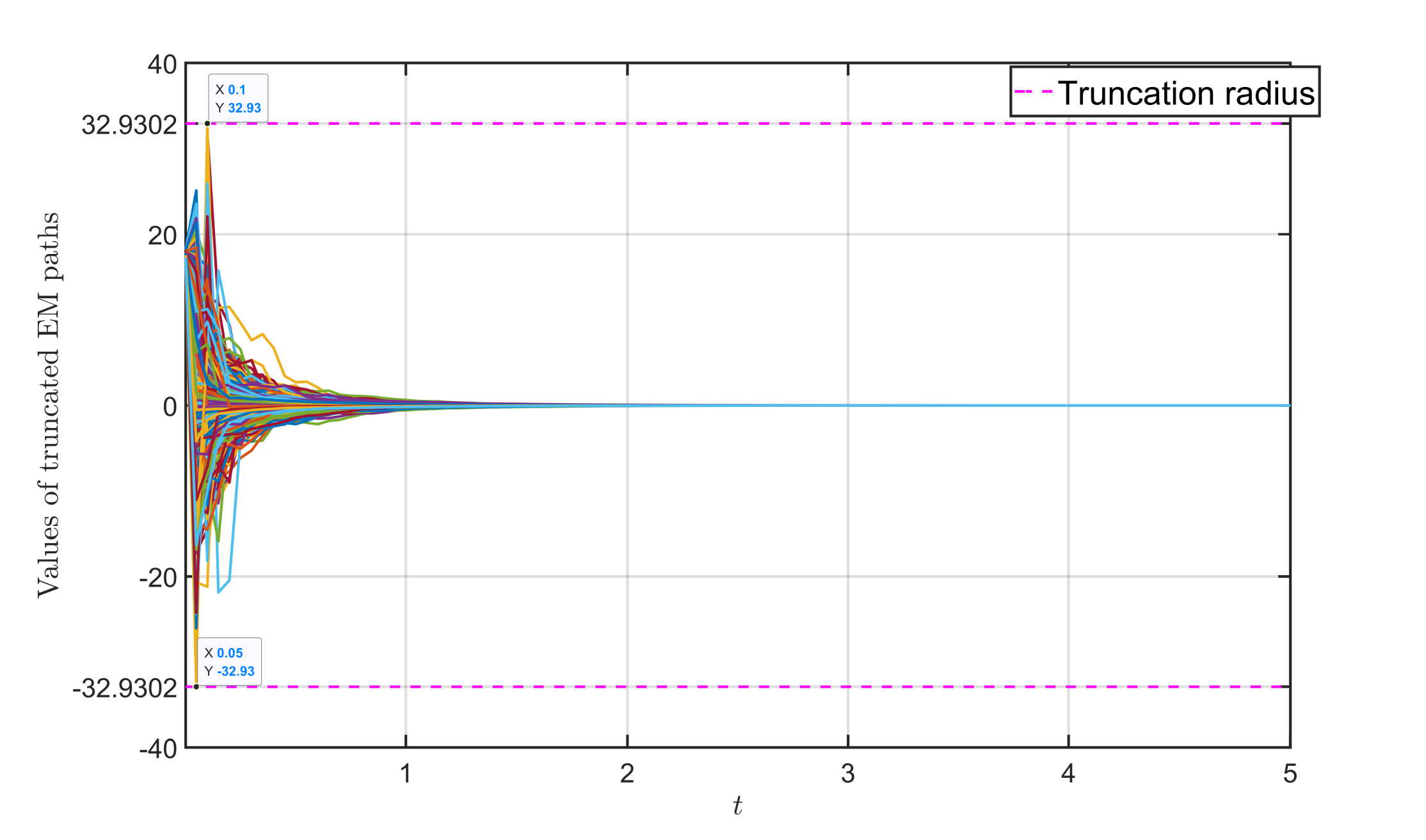}
\captionsetup{font=footnotesize}\vspace{-0.8em}
  \caption{The sample paths of the  numerical solution by the TEM scheme  for the initial value $X_0=18$, $\Delta=0.05$ and $M=2000$.}
\label{3}
\end{figure}\vspace{-0.5em}
{\bf \underline{Case 2 ($d\geq 2$).}~ }
{ Now we check the effect  of dimension $d$ of the MV-SDE on the TEM scheme. Let $x=(x_1,x_2,\cdots, x_d)$, $B_t$ is a $d$-dimensional Brownian motion, the initial condition is a vector whose components are independent $\mathcal{N}(0,1)$-random variables, the coefficients $$f(x,\mu)=f_1(x)+\int_{\mathbb{R}^{d}}y\mu(\mathrm{d}y),~~
g(x)=|x|^{\frac{3}{2}}E_{d},$$ where $$f_1(x)=\Big(-2x_1-x_1\sqrt{x_1^2+\cdots+x_d^2}, \cdots, -2x_d-x_d\sqrt{x_1^2+\cdots+x_d^2} \Big)^{T},$$
and 
  $E_{d}$ is a $d\times d$ identity matrix. We implement the TEM scheme in MATLAB to test the convergence rate.  Let initial values $X^{i}_0$, $i=1,\cdots, M$,  obey a standard normal distribution $\mathcal{N}(0,1)$ independently. Take $T=1$, $N=10$,  $M=5000$, and $\Delta=2^{q}$, where $q\in\{-10,-11,-12,-13,-14\}$. Figure \ref{4} depicts that the convergence rate of the TEM scheme with respect to 
time step size $\Delta$ across different dimensions
$d=1$,$d=2$, $d=4$ and $d=6$, respectively. 
The fact is revealed that the convergence rate is of order $1/2$ and independent of the dimension  $d$. }
\vspace{-1em}
\begin{figure}[H]
  \centering
\includegraphics[width=12cm,height=4cm]{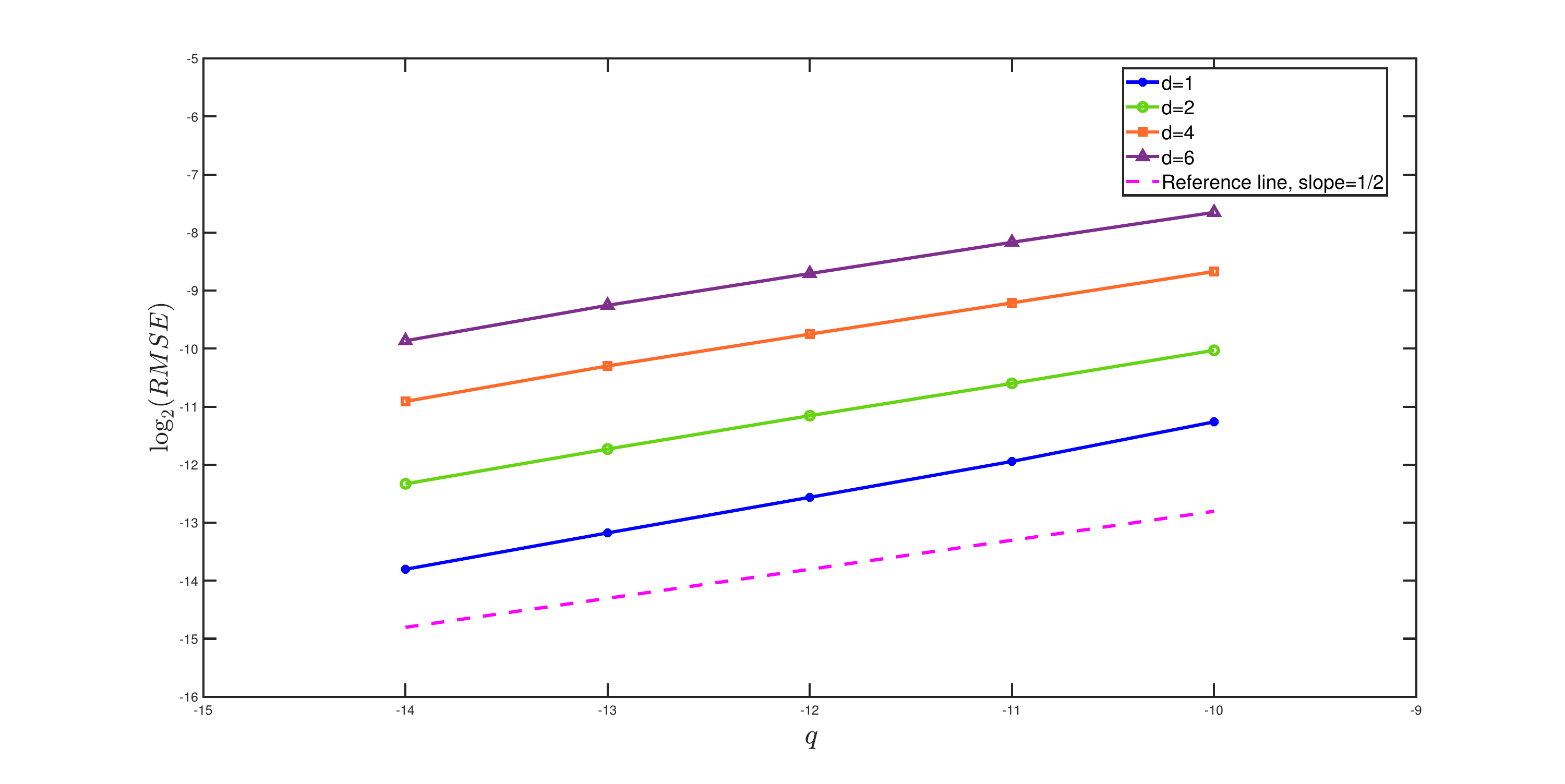}
\captionsetup{font=footnotesize}\vspace{-0.8em}
  \caption{Strong  error v.s. step size  $\Delta$ for different dimension  $d$.}
\label{4}
\end{figure} }

\end{expl}

 \begin{expl}\label{exp7.2}{\rm
 Consider a two-dimensional MV-SDE 
 \begin{align*}
 \mathrm{d}X(t)=f(X(t),\mathcal{L}^{X}_{t})\mathrm{d}t+g(X(t),\mathcal{L}^{X}_{t})\mathrm{d}B_{t},
 \end{align*}
 where
 \begin{align*}
 f(x,\mu)=\begin{pmatrix}-2x_1+x_2-x_1^3+\int x_1 \mu(\mathrm{d}x)\\-2x_2-x_2^3+\int x_2\mu(\mathrm{d}x)\end{pmatrix},~~~g(x,\mu)=\frac{1}{2}\begin{pmatrix}(1-x_1^2)&0\\0&(1-x_2^2)\end{pmatrix},
 \end{align*} where $x=(x_1,x_2)\in \RR^{2}$.
A direct computation implies that  Assumptions \ref{ass1}-\ref{ass6} hold with $\alpha=2$, $2<p\leq 9$, and $2<p_0\leq 5$.  
Choose $  \varphi(u)=6(1+u^2), ~u>0, ~ H=30, ~ \kappa= {1}/{3}. $  This implies  
\begin{align*}
\pi_{\Delta}(x_1,x_2) =\displaystyle\left\{\begin{array}{lcl} (x_1,x_2),~~&~  x_1^2+x_2^2 \leq 5\Delta^{-\frac{1}{3}}-1 ,\\
\displaystyle  \sqrt{\frac{5\Delta^{-\frac{1}{3}}-1} {x_1^2+x_2^2}}(x_1,x_2),~~&~ x_1^2+x_2^2 > 5\Delta^{-\frac{1}{3}}-1.
\end{array}\right.
\end{align*}
Then the TEM scheme is described by
 \begin{equation}\label{eq7.1}
\begin{aligned}
\left\{
\begin{array}{rl}
&(\bar{Y}^{i,M}_{1,0},\bar{Y}^{i,M}_{2,0})=(X^{i}_{1,0}, X^{i}_{2,0}),~~i=1,2,\cdots,M,\\
&(Y^{i,M}_{1,t_{k}},Y^{i,M}_{2,t_{k}})=\pi_{\Delta}(\bar{Y}^{i,M}_{1,t_{k}},\bar{Y}^{i,M}_{2,t_{k}}) ,~~
k=0,1,\cdots,\\
&\bar{Y}^{i,M}_{1,t_{k+1}}=Y^{i,M}_{1,t_k}+\Big(-2Y^{i,M}_{1,t_k}+Y^{i,M}_{2,t_k}-(Y^{i,M}_{1,t_k})^3+\frac{1}{M}\sum_{j=1}^{M}Y^{j,M}_{1,t_k}\Big)\Delta+\frac{1}{2}\big(1-(Y^{i,M}_{1,t_k})^2\big)\Delta B^{i}_{1,t_k},\\
&\bar{Y}^{i,M}_{2,t_{k+1}}=Y^{i,M}_{2,t_k}+\Big(-2Y^{i,M}_{2,t_k}-(Y^{i,M}_{2,t_k})^3+
\frac{1}{M}\sum_{j=1}^{M}Y^{j,M}_{2,t_k}\Big)\Delta+\frac{1}{2}\big(1-(Y^{i,M}_{2,t_k})^2\big)\Delta B^{i}_{2,t_k}.
\end{array}
\right.
\end{aligned}
\end{equation}

By Lemma \ref{L7.1}, the MV-SDE admits a unique invariant probability measure $\mu$. By virtue of Theorems \ref{th7.2} and \ref{Th5.21},
the TEM numerical solution \eqref{eq7.1} has  a unique invariant measure $\mu^{\Delta,M,*}$ approximating  $\mu$ in the $L^2$- Wasserstein metric. To our knowledge, several existing explicit numerical methods in the literature, such as \cite{MR4367675,MR4212406,MR4497846,MR4293705}, cannot treat this case. 

 \begin{figure}[H]
  \centering
\includegraphics[width=12cm,height=4cm]{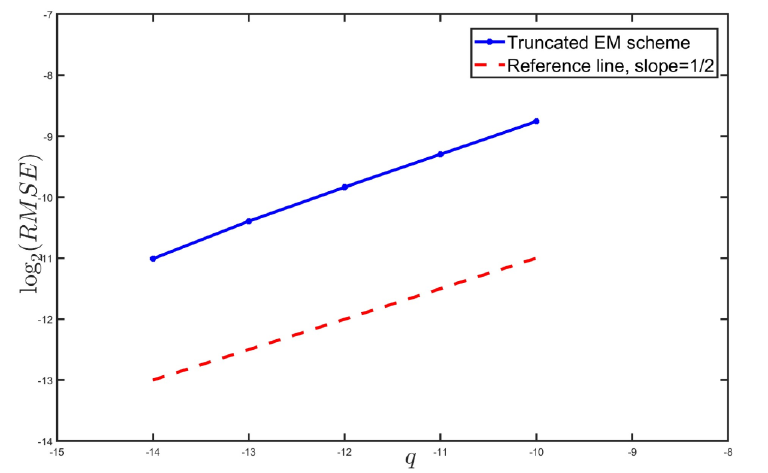}
\captionsetup{font=footnotesize}
  \caption{The blue one is the ~$\log_2(RMSE)$~as a function of~$q\in\{-10,-11,-12,-13,-14\}$, while the red one is the reference line with slope $ {1}/{2}$.}
\label{figure1*}
\end{figure}

Next,  we carry out some numerical experiments to  check the effectiveness of the  TEM  numerical scheme.
Let initial values $(X^{i}_{1,0},X^{i}_{2,0})$, $i=1,2,\cdots M$, obey the distribution $(\mathcal{N}(0,1),\mathcal{N}(0,1))$  independently.
Take  $N=10$, $M=5000$,  and $\Delta=2^{q}$, where $q\in\{-10, -11, -12, -13, -14\}$. Figure \ref{figure1*}  predicts that the TEM numerical solution has  a convergence rate of order 1/2. This supports our theoretical results. Furthermore, Figure \ref{figure16}  plots  the convergence rate of the TEM numerical solution with  $\Delta=2^q$ for particle number  
$M=500,~1000, ~1500, ~2000$, respectively. In Figure \ref{figure16}, the error curves for different $M$ are nearly coincident. This fact supports the theoretical result that  the approximation error is independent of  $M$.
 
Let $\bar{Y}^{M,M}_{t}=(\bar{Y}^{M,M}_{1,t},\bar{Y}^{M,M}_{2,t})$ and $\bar{Z}^{M,M}_{t}=(\bar{Z}^{M,M}_{1,t},\bar{Z}^{M,M}_{2,t})$ denote the TEM numerical solutions of the $M$th particle with the initial distribution $(\mathcal{N}(0,1),\mathcal{N}(0,1))$ and the initial value $(1,1)$, respectively.\vspace{-1em}
\begin{figure}[H]
  \centering
\includegraphics[width=12cm,height=4cm]{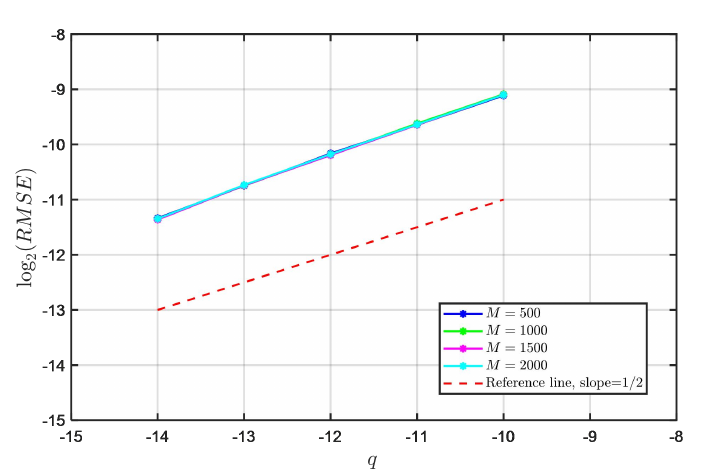}
\captionsetup{font=footnotesize}\vspace{-0.8em}
  \caption{$\log_2(RMSE)$~as a function of~$q\in\{-10,-11,-12,-13,-14\}$ for $M=500,~1000, ~1500, ~2000$, respectively. The red dashed line  is the reference line with slope $ {1}/{2}$.}
\label{figure16}
\end{figure}
Figure \ref{figure7} plots  the empirical density functions of the numerical solutions $\bar{Y}^{M,M}_{t}$ (Figure \ref{figure7} (a)) and $\bar{Z}^{M,M}_{t}$ (Figure \ref{figure7} (b)) with $6000$ sample points at $t=100$ in 3D and 2D settings, respectively. Obviously, (a) and (b) in Figure \ref{figure7} appear to be highly similar. For clarity,  Figure \ref{figure8} compares the empirical cumulative distribution functions (ECDFs) for  $\bar{Y}^{M,M}_{t}$ and $\bar{Z}^{M,M}_{t}$ at $t=100$. On the other hand, using the Kolmogorov-Smirnov (K-S) test we know
that  the empirical distributions of $\bar{Y}^{M,M}_{t}$ and $\bar{Z}^{M,M}_{t}$ at $t=100$ are  from the same distribution with $0.05$ significance level. Thus, the existence and uniqueness of the numerical invariant probability measure $\mu^{\Delta,M,*}$ generated by the TEM scheme is predicted.

\begin{figure}[H]
  \centering
\includegraphics[width=12cm,height=4cm]{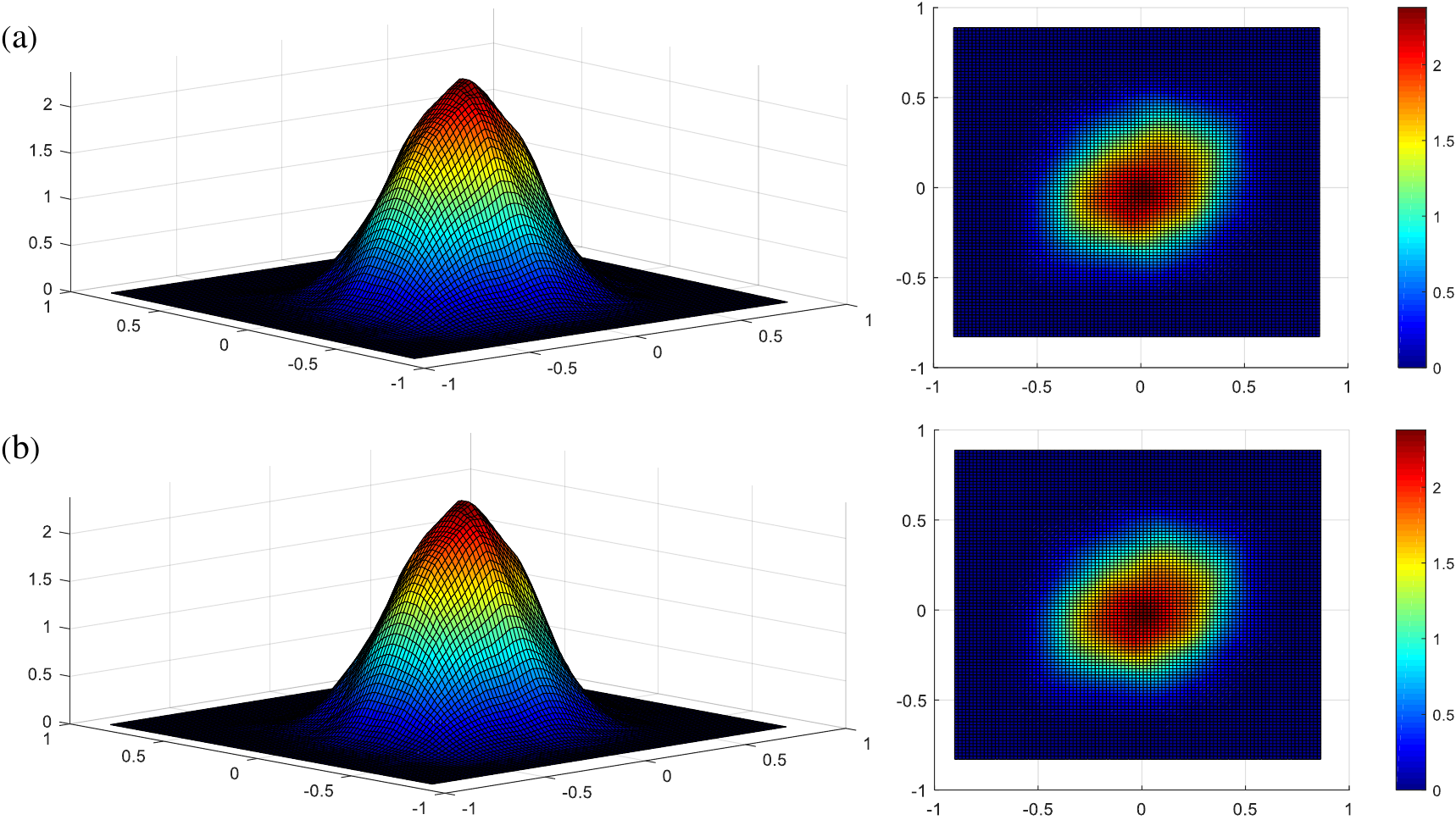}
\captionsetup{font=footnotesize}
  \caption{(a) The empirical density of $\bar{Y}^{M,M}_{100}$ with initial distribution  $(\mathcal{N}(0,1),\mathcal{N}(0,1))$ in $3D$ and $2D$ settings. (b) The empirical density of $\bar{Z}^{M,M}_{100}$ with initial value $(1,1) $ in $3D$ and $2D$ settings.}
\label{figure7}
\end{figure}
\vspace{-1.5em}
\begin{figure}[H]
  \centering
\includegraphics[width=12cm,height=4cm]{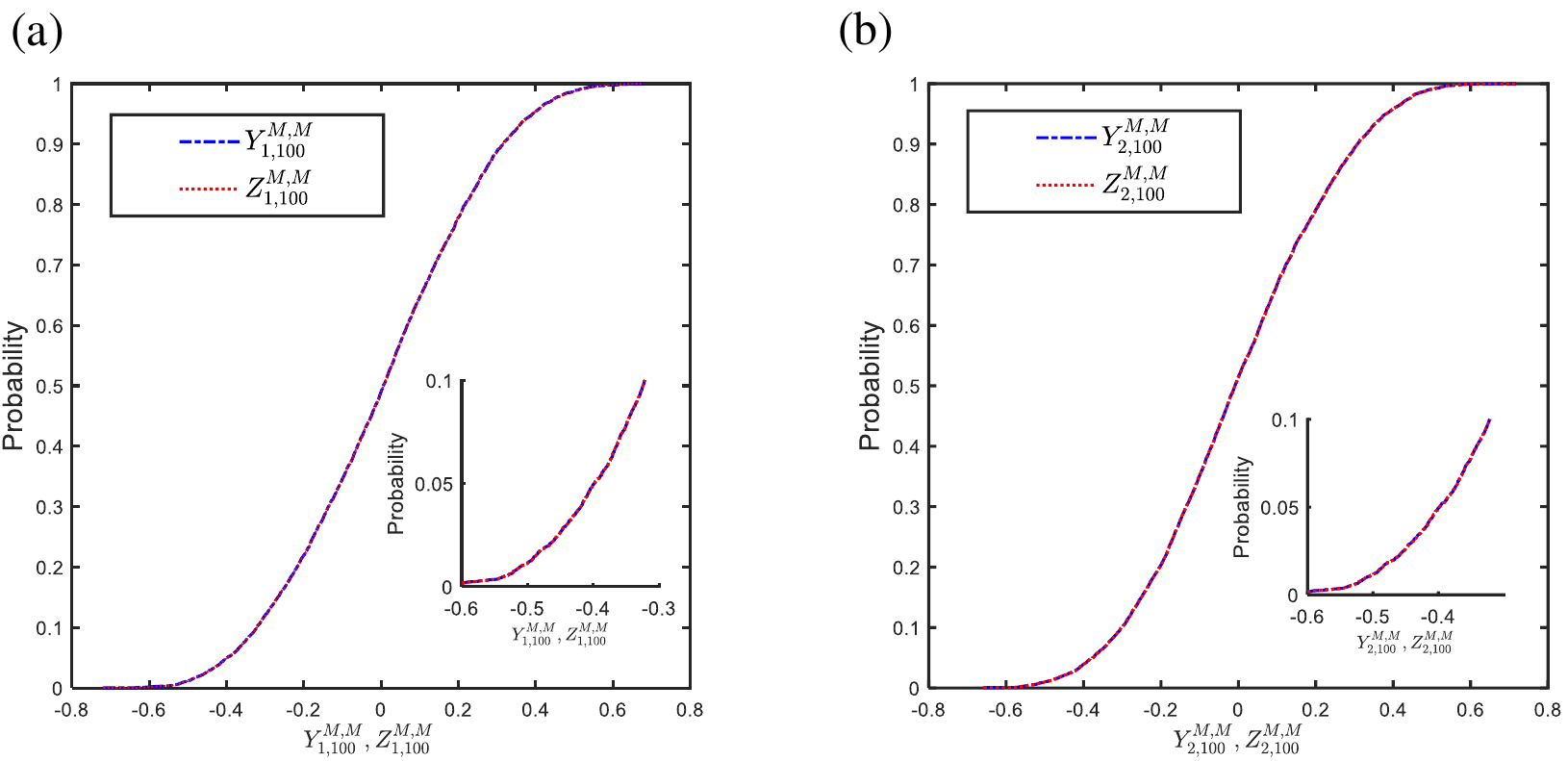}
\captionsetup{font=footnotesize}
  \caption{(a) The ECDFs of $\bar{Y}^{M,M}_{1,100}$ and $\bar{Z}^{M,M}_{1,100}$. (b) The ECDFs of $\bar{Y}^{M,M}_{2,100}$ and $\bar{Z}^{M,M}_{2,100}$. The blue dashed line represents the numerical solution with initial distribution $(\mathcal{N}(0,1),\mathcal{N}(0,1))$ while the red dashed line represents the numerical  solution with initial value (1,1).}
\label{figure8}
\end{figure}
Next, we continue to carry out some numerical experiments to verify the efficiency of the TEM scheme \eqref{eq7.1} in the approximation of invariant probability measures. Let $M=500$ and  the initial distribution obey $(\mathcal{N}(0,1),\mathcal{N}(0,1))$.
We regard  the  numerical solution with $\Delta=2^{-12}$  as the exact solution $X^{i,M}_{t}=(X^{i,M}_{1,t},X^{i,M}_{2,t})$, and compare it with the numerical solution  $\bar{Y}^{i,M}_{t}=(\bar{Y}^{i,M}_{1,t},\bar{Y}^{i,M}_{2,t})$ with  $\Delta=2^{-8}$. 
Figure \ref{figure9} (a) and (b) depict   the empirical density functions of $X^{M,M}_t$ and  $\bar{Y}^{M,M}_{t}$ at $t=100$ with 6000 samples in 2D and 3D settings, respectively. 
Furthermore, Figure \ref{figure10} plots the ECDFs of the exact solution $X^{M,M}_{t}$ with a blue dashed line and the numerical solution $\bar{Y}^{M,M}_{t}$ with a red dashed line at $t=100$ with 6000 samples. On the other hand, using the K-S test, we conclude that the exact invariant measure $\mu$ and the numerical invariant measure  $\mu^{\Delta,M,*}$ generated by the TEM scheme  are from the same distribution with $0.05$ significance level. Therefore, $\mu^{\Delta,M,*}$  approximates $\mu$ effectively.
\vspace{-1em}
\begin{figure}[H]
  \centering
\includegraphics[width=12cm,height=4cm]{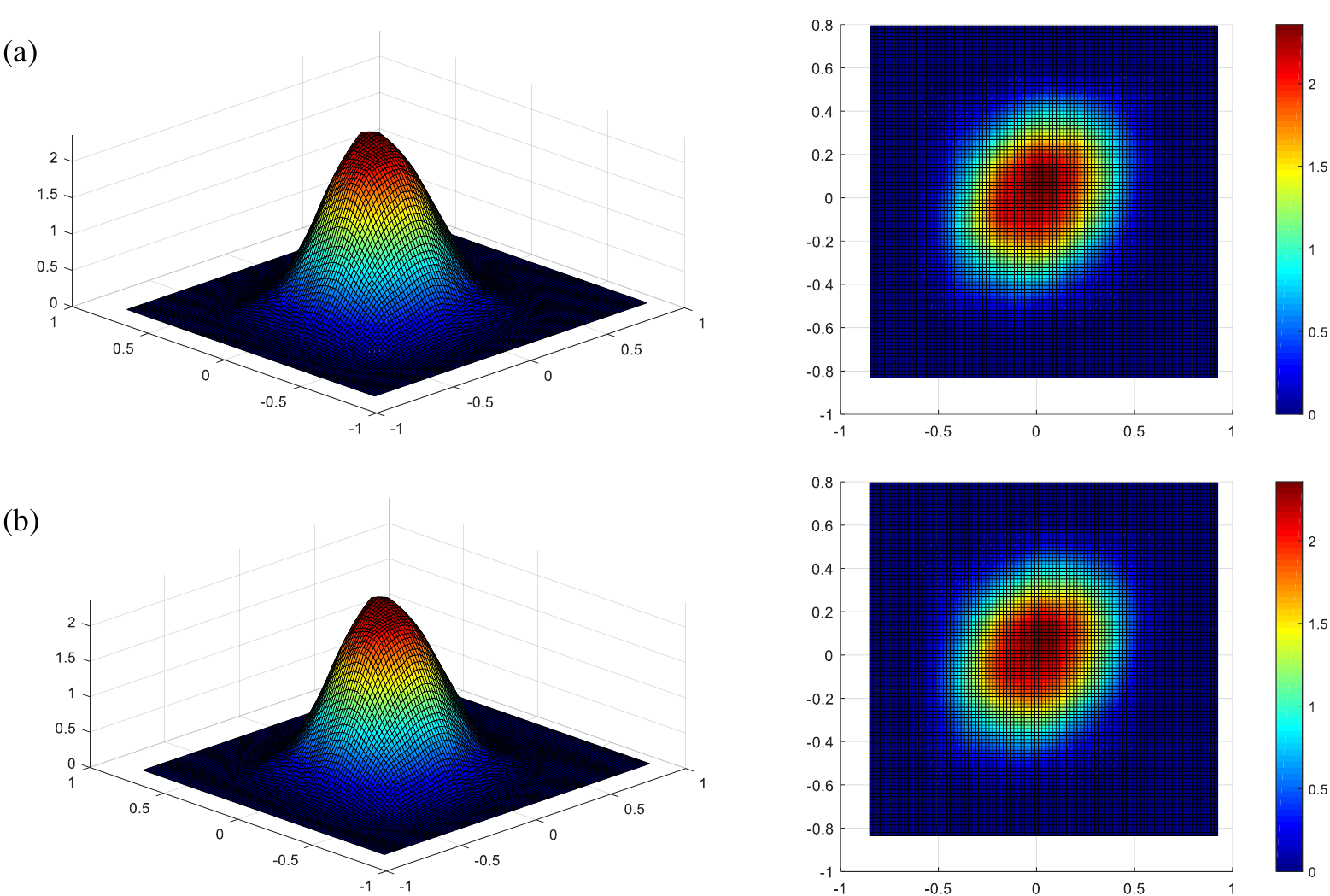}
\captionsetup{font=footnotesize}
  \caption{(a) The density function picture of $\mu$ in 3D and 2D settings. (b) The density function picture of $\mu^{\Delta,M,*}$ in 3D and 2D settings.}
\label{figure9}
\end{figure}
\vspace{-1.5em}
\begin{figure}[H]
  \centering
\includegraphics[width=12cm,height=4cm]{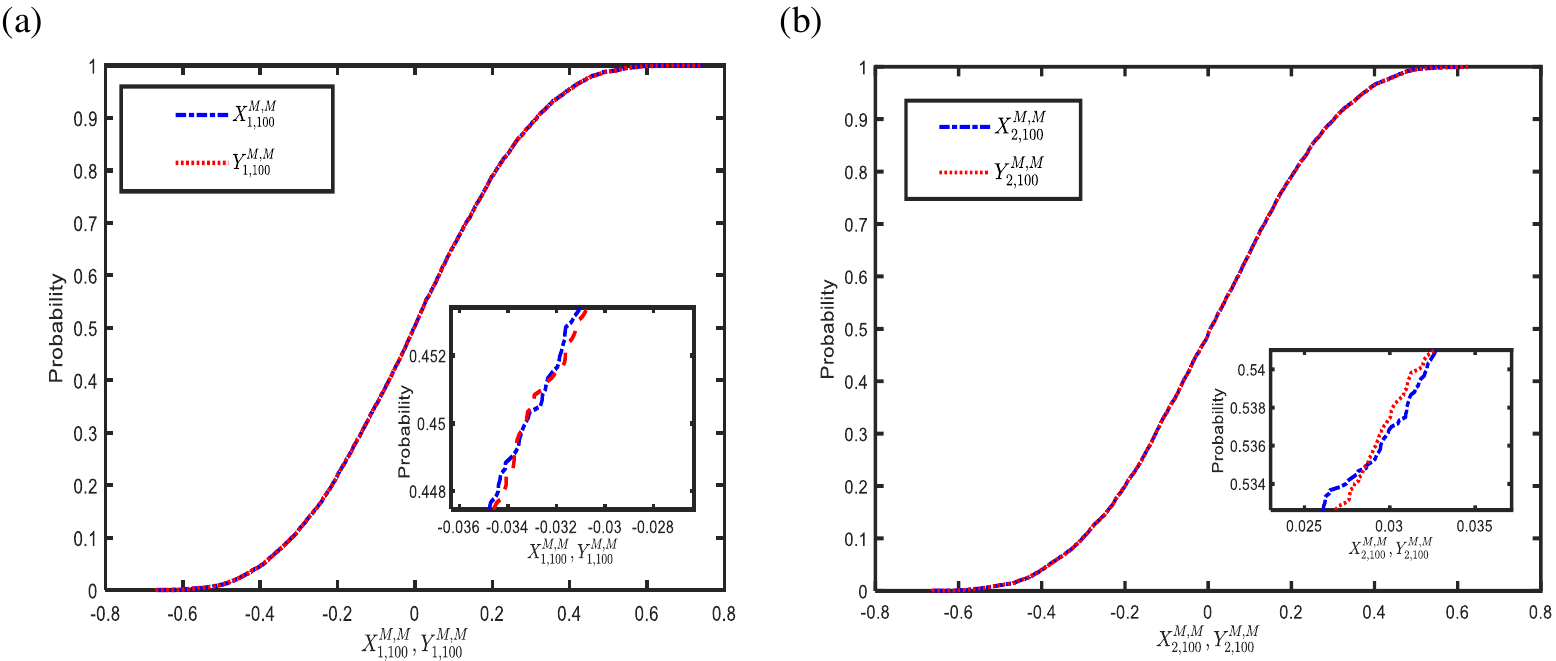}
\captionsetup{font=footnotesize}
  \caption{(a) The ECDFs of $X^{M,M}_{1,100}$ and $\bar{Y}^{M,M}_{1,100}$. (b) The ECDFs of $X^{M,M}_{2,100}$ and $\bar{Y}^{M,M}_{2,100}$. The  blue dashed line represents the exact solution of MV-SDE while the red dashed line represents the numerical solution generated by the TEM scheme.}
\label{figure10}
\end{figure}
}
\end{expl}
{\begin{expl}
{\rm Consider  the  scalar MV-SDE with the non-linear interaction term
\begin{align}\label{eqf6.3}
\mathrm{d}X(t)=\Big(-\frac{5}{2}X(t)- X^3(t)+\frac{1}{1+\exp\{-\mathbb{E}[\arctan(Z+X(t))]|_{Z=X(t)}\}}\Big)\mathrm{d}t+\sqrt{2}\mathrm{d}B_t,
\end{align}
where 
$$f(x,\mu)=-\frac{5}{2}x-x^3+\frac{1}{1+\exp\big\{-\int_{\mathbb{R}}\arctan(x+y)\mu(\mathrm{d}y)\big\}},~~g(x,\mu)=\sqrt{2}.$$
It can be verified that Assumptions \ref{ass1}-\ref{ass6} hold with $\alpha=2$ and any $p,p_0>2$.  We let $p=26$ and choose 
$ \kappa 
= {1}/{4}.$  The corresponding IPS is 
\begin{align*}
\mathrm{d}X^{i,M}_{t}&=\Big[-\frac{5}{2}X^{i,M}_{t}-\big(X^{i,M}_{t}\big)^3+\Big(1+\exp\Big\{-\frac{1}{M}\sum_{j=1}^{M}\arctan\big(X^{i,N}_{t}+X^{j,M}_{t}\big)\Big\}\Big)^{-1}\Big]\mathrm{d}t\nn\
\\&~~~+\sqrt{2}\mathrm{d}B^{i}_{t},~~  i=1, \cdots, M.
\end{align*}
Choose $  \varphi(u)=3(1+u^2), ~u>0, ~~H=30, ~~ \kappa= {1}/{4}. $  This implies  
\begin{align*}
\pi_{\Delta}(x) =\displaystyle\left\{\begin{array}{lcl} x,~~&~  |x| \leq \sqrt{10\Delta^{-\frac{1}{4}}-1} ,\\
\displaystyle  x\sqrt{10\Delta^{-\frac{1}{4}}-1},~~&~ |x| > \sqrt{10\Delta^{-\frac{1}{4}}-1}.
\end{array}\right.
\end{align*}
Then the TEM scheme is defined by
\begin{equation}\label{eq7.3}
\begin{aligned}
\left\{
\begin{array}{rl}
\bar{Y}^{i,M}_{0}&=X^{i}_0, ~~i=1,2,\cdots,M,~~Y^{i,M}_{t_{k}}=\pi_{\Delta}(\bar{Y}^{i,M}_{t_{k}}), ~~ k=0,1,\cdots,\\
\bar{Y}^{i,M}_{t_{k+1}}&=Y^{i,M}_{t_k}+Y^{i,M}_{t_k}\Big(-\frac{5}{2}-\big|Y^{i,M}_{t_k}\big|^2\Big)\Delta
\\&~~~+
\Big(1+\exp\Big\{-\frac{1}{M}\sum_{j=1}^{M}\arctan(Y^{i,M}_{t_k}+Y^{j,M}_{t_k})\Big\}\Big)^{-1}\Delta+\sqrt{2} \Delta B^{i}_{t_k},
\end{array}
\right.
\end{aligned}
\end{equation} 
We compare the above TEM scheme with the tamed scheme  \cite{Zhang} described by
\begin{align}\label{eqf6.5}
Z^{i,M}_{t_{k+1}}&=Z^{i,M}_{t_{k}}+\frac{-\frac{5}{2}Z^{i,M}_{t_{k}}-(Z^{i,M}_{t_{k}})^3}
{1+\sqrt{\Delta}\Big|-\frac{5}{2}Z^{i,M}_{t_{k}}-(Z^{i,M}_{t_{k}})^3\Big|}\Delta\\&~~~
+\Big(1+\exp\Big\{-\frac{1}{M}\sum_{j=1}^{M}\arctan\big(Z^{i,M}_{t_{k}}+Z^{j,M}_{t_{k}}\big)\Big\}\Big)^{-1}\Delta\nn\
+\sqrt{2}\Delta B^{i}_{t_{k}}.
\end{align}
Let initial data  $X^{i}_{0}$ obey the distribution $\mathcal{N}(0,1)$, $i=1,2,\cdots M$,  independently. 
Take  $N=50$, $M=1000$.
First, we use the TEM scheme \eqref{eq7.3} and the tamed EM scheme \eqref{eqf6.5} to carry out numerical experiments. 
Figure $\ref{exp31}~(a)$ predicts that the TEM scheme achieves a   convergence rate of order 1 while the tamed EM scheme only attains a  convergence rate of order 1/2 at time $T=1$ with respect to step size  $\Delta=2^{q}$  for $q= -9,-10,-11,-12,-13 $. 
Figure $\ref{exp31}~(b)$ depicts the numerical errors  from the TEM and tamed EM schemes, respectively, with $\Delta=2^{q}$ for $q\in \{-5, -6, -7, -8, -9, -10, -11, -12, -13, -14\}$, with respect to the running times. Obviously, for the same error $0.01$, the running time of the TEM scheme is significantly shorter than that of the tamed EM scheme in the same computer.
 \vspace{-1em}
\begin{figure}[H]
  \centering
\includegraphics[width=12cm,height=4cm]{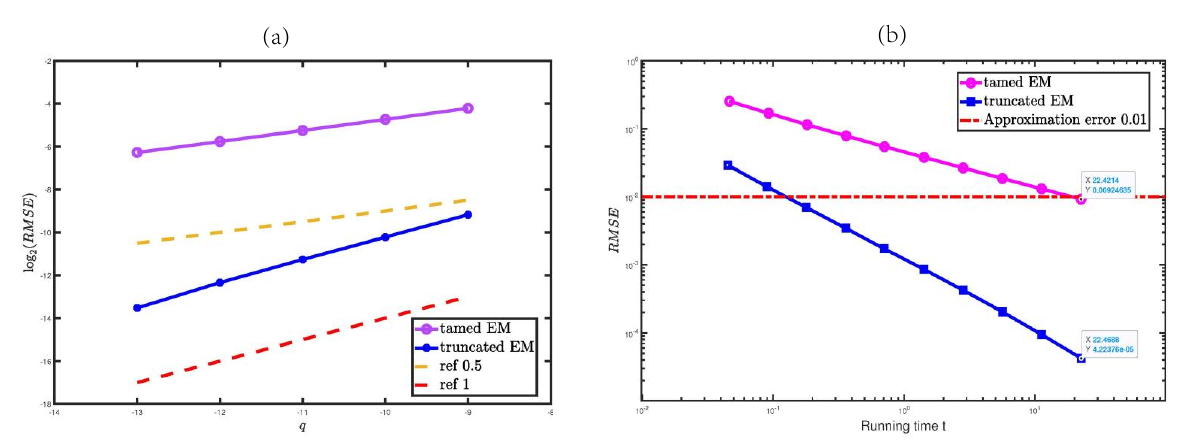}
\captionsetup{font=footnotesize}\vspace{-0.8em}
  \caption{(a) Strong error v.s. step size $\Delta$. (b) Strong error v.s. Running time. }
\label{exp31}
\end{figure}
\vspace{-1em}
\begin{figure}[H]
  \centering
\includegraphics[width=12cm,height=4cm]{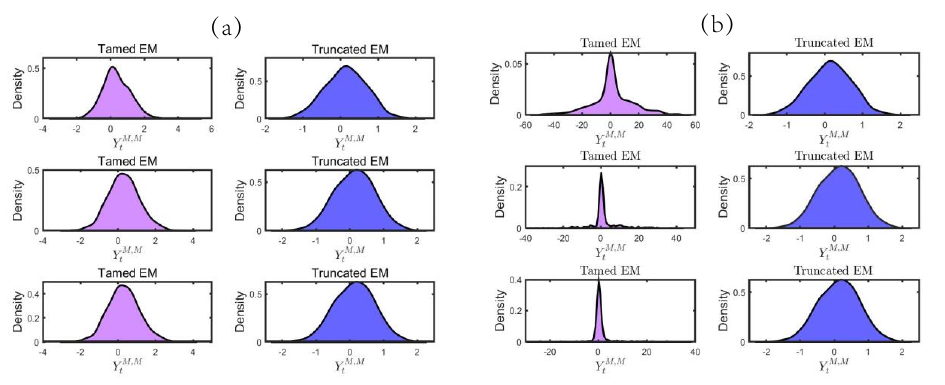}
\captionsetup{font=footnotesize}
  \caption{(a) Density with $X_0\sim N(3,9)$ at $t=2, 20,30$. (b) Density with $X_0\sim N(20,400)$ at $t=2, 20, 30$. }
\label{figure12}
\end{figure}
Let $M=500$, $N=1000$, and $\Delta=2^{-4}$.  We   compare the performance of the TEM scheme \eqref{eq7.3} and the tamed EM scheme \eqref{eqf6.5}   approximating the  invariant distribution of the MV-SDE \eqref{eqf6.3}. Figure \ref{figure12} depicts the empirical densities of numerical solutions generated by the TEM and the tamed EM schemes at  times $t=2, 20, 30$ with  initial distributions   $\mathcal{N}(3,9)$ $(\ref{figure12}~(a))$ and  $\mathcal{N}(20,400)$ $(\ref{figure12}~(b))$.
  Figure \ref{figure12} reveals that  the distributions of the  TEM numerical solutions with different  initial data    converge rapidly  to a stationary state while the Tamed EM scheme fails to do it.
}

\end{expl}}
  \begin{appendix}

\section{Proof of Counterexample } \label{appendix2}
This appendix aims to give a rigorous proof  that the EM numerical solution of MV-SDE \eqref{Ne4}  fails to be exponentially stable whatever the step size $\Delta$ is small.

For any  $\Delta\in (0,1]$, the EM scheme of IPS \eqref{Ne1} is
\begin{equation}\label{Ne5}
\begin{aligned}
\left\{
\begin{array}{rl}
Z^{i,M}_{0}&=X^i_0, ~~~~i=1,\cdots, M,\\
Z^{i,M}_{t_{k+1}}&=Z^{i,M}_{t_k}+\Big[Z^{i,M}_{t_k}\Big(-2-\big|Z^{i,M}_{t_k}\big|\Big)+\frac{1}{M}\sum_{j=1}^{M}Z^{j,M}_{t_k}\Big]\Delta+\frac{1}{2}|Z^{i,M}_{t_k}|^{\frac{3}{2}}\Delta B^{i}_{t_k}
\\&={\Big(Z^{i,M}_{t_{k}}(1-2\Delta)-Z^{i,M}_{t_{k}}|Z^{i,M}_{t_{k}}|\Delta+\frac{1}{2}|Z^{i,M}_{t_{k}}|^{\frac{3}{2}}\Delta B^{i}_{t_{k}}\Big)+\frac{1}{M}\sum_{j=1}^{M}Z^{j,M}_{t_{k}}\Delta.}
\end{array}
\right.
\end{aligned}
\end{equation} 
\begin{lemma}\label{AL2}
Fix any  $\Delta\in (0,1]$.
Assume that  
$
(\mathbb{E}|X_0|^2)^2\geq 20/\Delta^2. 
$
  Then the EM numerical solution given by \eqref{Ne5} satisfies   $\mathbb{E}|Z^{i,M}_{t_k}|^2\geq 2^{k}\mathbb{E}|X_0|^2.$  Moreover, $ \lim_{k\rightarrow\infty}\mathbb{E}|Z^{i,M}_{t_k}|^2=\infty.$ 
\end{lemma}
\begin{proof}
It follows from \eqref{Ne5} that 
\begin{align}\label{eqNN9}
|Z^{i,M}_{t_{k+1}}|^2\geq\mathcal{A}^{i,M}_{t_k}+\mathcal{B}^{i,M}_{t_k},
\end{align}
where \begin{align}
 \mathcal{A}^{i,M}_{t_k}&=\Big(Z^{i,M}_{t_k}(1-2\Delta)-Z^{i,M}_{t_{k}}|Z^{i,M}_{t_k}|\Delta+\frac{1}{2}|Z^{i,M}_{t_k}|^{\frac{3}{2}}\Delta B^{i}_{t_k}\Big)^2,\nn\\ 
\mathcal{B}^{i,M}_{t_k}&=\big(\frac{1}{M}\sum_{j=1}^{M}Z^{j,M}_{t_k}\big)^2\Delta^2+2Z^{i,M}_{t_k}\big(\frac{1}{M}\sum_{j=1}^{M}Z^{j,M}_{t_k}\big)\Delta(1-2\Delta)-2\big(Z^{i,M}_{t_k}\big)^2\big|\frac{1}{M}\sum_{i=1}^{M}Z^{i,M}_{t_k}\big|\Delta^2\nn\
\\&~~~+|Z^{i,M}_{t_k}|^{\frac{3}{2}}\big(\frac{1}{M}\sum_{i=1}^{M}Z^{i,M}_{t_k}\big)\Delta B^{i}_{t_k}\Delta.\label{N7.12}
\end{align}
Furthermore, a direct computation yields that
\begin{align}\label{eqN5}
\mathcal{A}^{i,M}_{t_k}&=|Z^{i,M}_{t_{k}}|^2(1-2\Delta)^2 +|Z^{i,M}_{t_k}|^4\Delta^2+\frac{1}{4}|Z^{i,M}_{t_k}|^3|\Delta B^{i}_{t_k}|^2 -2|Z^{i,M}_{t_k}|^3\Delta(1-2\Delta) \nn\
\\&~~~ + Z^{i,M}_{t_k} |Z^{i,M}_{t_k}|^{\frac{3}{2}}\Delta B^{i}_{t_k}(1-2\Delta)-Z^{i,M}_{t_k} |Z^{i,M}_{t_k}|^{\frac{5}{2}} \Delta B^{i}_{t_k}\Delta.
\end{align}
Using the Young inequality implies that for any $\delta_1>0$, 
\begin{align*}
-2|Z^{i,M}_{t_k}|^3\Delta(1-2\Delta)\geq -\delta_1|Z^{i,M}_{t_k}|^2(1-2\Delta)^2-\frac{1}{\delta_1}|Z^{i,M}_{t_k}|^4\Delta^2.
\end{align*}
Using the above inequality we obtain from \eqref{eqN5} that
\begin{align}\label{eqNN8}
\mathcal{A}^{i,M}_{t_k}&\geq (1-\delta_1)|Z^{i,M}_{t_k}|^2(1-2\Delta)^2+\big(1-\frac{1}{\delta_1}\big)|Z^{i,M}_{t_k}|^4\Delta^2+Z^{i,M}_{t_k} |Z^{i,M}_{t_k}|^{\frac{3}{2}}\Delta B^{i}_{t_k}(1-2\Delta)\nn\
\\&~~~-Z^{i,M}_{t_k} |Z^{i,M}_{t_k}|^{\frac{5}{2}} \Delta B^{i}_{t_k}\Delta.
\end{align}
Owing to the  conditional independence of $Z^{i,M}_{t_k}$ and $\Delta B^{i}_{t_k}$ for any $1\leq i\leq M$, we obtain that
\begin{align}\label{eqNN8*}
\mathbb{E}\Big[Z^{i,M}_{t_k} |Z^{i,M}_{t_k}|^{\frac{3}{2}}\Delta B^{i}_{t_k}(1-2\Delta)\Big]=0,
~~~~\mathbb{E}\Big[Z^{i,M}_{t_k} |Z^{i,M}_{t_k}|^{\frac{5}{2}} \Delta B^{i}_{t_k}\Delta\Big]=0.
\end{align}
Then taking the expectations on both sides of  \eqref{eqNN8} and  using \eqref{eqNN8*}  {yields} that
\begin{align}\label{Ne7.15}
\mathbb{E}\mathcal{A}^{i,M}_{t_k}&\geq (1-\delta_1)(1-2\Delta)^2\mathbb{E}|Z^{i,M}_{t_k}|^2+\big(1-\frac{1}{\delta_1}\big)\Delta^2\mathbb{E}|Z^{i,M}_{t_k}|^4.
\end{align}
Next we  estimate   $\mathbb{E}\mathcal{B}^{i,M}_{t_k}$.
Applying the Young inequality derives that for any $\delta_2>0$ and $\delta_3>0$,
 \begin{align}\label{eqNN7}
2Z^{i,M}_{t_k}\big(\frac{1}{M}\sum_{j=1}^{M}Z^{j,M}_{t_k}\big)\Delta(1-2\Delta)&\geq -\delta_2|Z^{i,M}_{t_k}|^2(1-2\Delta)^2-\frac{\Delta^2}{\delta_2}\big(\frac{1}{M}\sum_{j=1}^{M}Z^{j,M}_{t_k}\big)^2,
 \\ 
-2(Z^{i,M}_{t_k})^2\big|\frac{1}{M}\sum_{j=1}^{M}Z^{j,M}_{t_k}\big|\Delta^2&\geq -\delta_3\big(\frac{1}{M}\sum_{j=1}^{M}Z^{j,M}_{t_k}\big)^2\Delta^2-\frac{\Delta^2}{\delta_3}|Z^{i,M}_{t_k}|^4.
\label{eqNN6}\end{align} 
Owing to the conditional independence between {$Z^{i,M}_{t_k}$} and $\Delta B^{i}_{t_k}$ for any $1\leq i,j\leq M$, we have
\begin{align}\label{Nee7.18}
\mathbb{E}\Big(|Z^ {i,M}_{t_k}|^{\frac{3}{2}}\big(\frac{1}{M}\sum_{j=1}^{M}Z^{j,M}_{t_k}\big)\Delta B^{i}_{t_k}\Delta\Big)=0.
\end{align}
Then inserting \eqref{eqNN7} and \eqref{eqNN6} into \eqref{N7.12} and taking the expectations on both sides of \eqref{N7.12}, we deduce from \eqref{Nee7.18}  that
\begin{align}\label{Ne7.19}
\mathbb{E}\mathcal{B}^{i,M}_{t_k}&\geq \big(1-\delta_3-\frac{1}{\delta_2}\big)\Delta^2\mathbb{E}\big(\frac{1}{M}\sum_{j=1}^{M}Z^{j,M}_{t_k}\big)^2-\delta_2(1-2\Delta)^2\mathbb{E}|Z^{i,M}_{t_k}|^2-\frac{\Delta^2}{\delta_3}\mathbb{E}|Z^{i,M}_{t_k}|^4.
\end{align}
Thus, combing \eqref{eqNN9}, \eqref{Ne7.15} and \eqref{Ne7.19} we obatin
{\begin{align}\label{A7.16}
\mathbb{E}|Z^{i,M}_{t_{k+1}}|^2&\geq(1-\delta_1-\delta_2) (1-2\Delta)^2 \mathbb{E}|Z^{i,M}_{t_k}|^{2}+\big(1-\delta_3-\frac{1}{\delta_2}\big)\Delta^2\mathbb{E}\big(\frac{1}{M}\sum_{j=1}^{M}Z^{j,M}_{t_k}\big)^2\nn\
\\&~~~+\big(1-\frac{1}{\delta_1}-\frac{1}{\delta_3}\big)\Delta^2\mathbb{E}|Z^{i,M}_{t_k}|^4.
\end{align}}
{Letting $\delta_1=\delta_3=4$ and $\delta_2=1$, and using the H\"older inequality we derive that
\begin{align*}
\mathbb{E}|Z^{i,M}_{t_{k+1}}|^2 &\geq \frac{\Delta^2}{2}\mathbb{E}|Z^{i,M}_{t_{k}}|^4-4(1-2\Delta)^2\mathbb{E}|Z^{i,M}_{t_k}|^{2}-4\Delta^2\mathbb{E}\big(\frac{1}{M}\sum_{j=1}^{M}Z^{j,M}_{t_k}\big)^2
\\&\geq \Big( \frac{\Delta^2}{2}\mathbb{E}|Z^{i,M}_{t_{k}}|^2-4(1-2\Delta)^2-4\Delta^2\Big)\mathbb{E}|Z^{i,M}_{t_{k}}|^2,
\end{align*}}
where the last inequality uses the fact that 
{$  \mathbb{E}\big(\frac{1}{M}\sum_{i=1}^{M}Z^{i,M}_{t_{k}}\big)^2\leq \mathbb{E}|Z^{i,M}_{t_{k}}|^2.$ }
Owing to $\Delta\in (0, 1]$, we have 
{$ \mathbb{E}|Z^{i,M}_{t_{k+1}}|^2\geq  2\Big(\frac{\Delta^2}{4}\mathbb{E}|Z^{i,M}_{t_{k}}|^2-4\Big)\mathbb{E}|Z^{i,M}_{t_{k}}|^2.
$ }
{If $ \mathbb{E}|X_0|^2 \geq 20/\Delta^2$, then $ \frac{\Delta^2}{4}\mathbb{E}|Z^{i,M}_{t_{0}}|^2-4\geq 1 $ follows from $\mathbb{E}|Z^{i,M}_{t_{0}}|^2=\mathbb{E}|X_0|^2$. By solving the above difference inequality we have $\mathbb{E}|Z^{i,M}_{t_{1}}|^2\geq 2 \mathbb{E}|X_0|^2 \geq 40/\Delta^2$. Then $ \frac{\Delta^2}{4} \mathbb{E}|Z^{i,M}_{t_{1}}|^2-4 \geq 1 $ holds, which implies  $\mathbb{E}|Z^{i,M}_{t_{2}}|^2\geq 2^2 \mathbb{E}|X_0|^2$.} Repeating this process, we obtain $\mathbb{E}|Z^{i,M}_{t_{k}}|^2\geq 2^{k} \mathbb{E}|X_0|^2$ for all positive integer $k$, which implies that $\lim_{k\rightarrow\infty}\mathbb{E}|Z^{i,M}_{t_k}|^2=\infty$.  
\end{proof}
{The code for Figures \ref{FN1} and \ref{3} is provided as follows:}
\begin{mdframed}[
  linewidth=0.7pt,
  linecolor=black,
  innerleftmargin=3pt,
  innerrightmargin=3pt,
  innertopmargin=1pt,
  innerbottommargin=1pt,
  skipabove=2pt,
  skipbelow=2pt
]
\noindent
\begin{minipage}[t]{0.485\textwidth}
\vspace{0pt}
\begin{lstlisting}[
  language=Matlab,
  basicstyle=\ttfamily\scriptsize,
  numbers=left,
  numberstyle=\tiny,
  numbersep=4pt,
  frame=none,
  breaklines=true,
  columns=fullflexible,
  keepspaces=true,
  aboveskip=0pt,
  belowskip=0pt,
  xleftmargin=10pt,
  lineskip=-1.5pt
]
% Truncated EM approximation
clc; clear; rng(101)

T=10; dt=0.05; M=2000;
Xzero=18; Ndot=T/dt;
Y=[]; Z=[];
Y(:,1)=ones(M,1).*Xzero;
dw=sqrt(dt)*normrnd(0,1,M,Ndot);

for n=1:Ndot
  TH_1=50*dt^(-1/3)/4-1;
  Y_Y=sum(Y(:,n))/M;
  Y(:,n+1)=Y(:,n) ...
   +Y(:,n).*(-2-abs(Y(:,n))).*dt ...
   +Y_Y.*dt ...
   +0.5.*abs(Y(:,n)).^(3/2).*dw(:,n);
  Y(:,n+1)=min(1,TH_1./abs(Y(:,n+1))) ...
   .*Y(:,n+1);
end
\end{lstlisting}
\end{minipage}
\hfill
\vrule width 0.4pt
\hfill
\begin{minipage}[t]{0.485\textwidth}
\vspace{0pt}
\begin{lstlisting}[
  language=Matlab,
  basicstyle=\ttfamily\scriptsize,
  numbers=left,
  numberstyle=\tiny,
  numbersep=4pt,
  frame=none,
  breaklines=true,
  columns=fullflexible,
  keepspaces=true,
  aboveskip=0pt,
  belowskip=0pt,
  xleftmargin=10pt,
  lineskip=-1.5pt
]
% Standard EM approximation
Z(:,1)=ones(M,1).*Xzero;
for n=1:Ndot
  Z_Z=sum(Z(:,n))/M;
  Z(:,n+1)=Z(:,n) ...
   +Z(:,n).*(-2-Z(:,n))*dt ...
   +Z_Z.*dt ...
   +0.5.*abs(Z(:,n)).^(3/2).*dw(:,n);
end

% Figure 1
figure(1); hold on
for i=1:M, plot(0:dt:T,Y(i,:)); end
xlabel('$t$','Interpreter','latex');
ylabel('Values of truncated EM paths');
axis([0 5 -40 40]); grid on
z=ones(1,Ndot+1).*TH_1;
h=plot(0:dt:T,z,'--m','LineWidth',1);
plot(0:dt:T,-z,'--m','LineWidth',1);
legend(h,'Truncation radius');

% Figure 2
figure(2); hold on
for i=1:M, plot(0:dt:T,Z(i,:)); end
xlabel('$t$','Interpreter','latex');
ylabel('Values of EM paths');
axis([0 0.3 -800 10]); grid on
\end{lstlisting}
\end{minipage}
\end{mdframed}

\end{appendix}
\section*{Acknowledgements}
\sloppy

 The authors would like to thank the  editor and referees for their helpful and valuable suggestions. 

Research of Yuanping Cui was supported by the National Natural
Science Foundation of China (No. 12401216).
Research of Xiaoyue Li was supported by the National Natural
Science Foundation of China (No. 12371402, 12526433) and the Tianjin Natural Science Foundation (24JCZDJC00830).  
 Research of Fengyu Wang was supported by  the National Key R\&D Program of China (No. 2022YFA1006000, 2020YFA0712900).

}\end{sloppypar}
\end{document}